\documentclass[12pt]{article}
\usepackage{amssymb,amsmath,amsthm, amsfonts}
\usepackage{graphicx}
\usepackage{epsfig}
\usepackage{tikz}
\usepackage{mathrsfs}

\def\disp{\displaystyle}

\def\dref#1{(\ref{#1})}

\theoremstyle{plain}
\newtheorem{theorem}{Theorem}[section]
\newtheorem{lemma}{Lemma}[section]

\theoremstyle{definition}
\newtheorem{definition}{Definition}[section]

\newtheorem{remark}{Remark}[section]

\numberwithin{equation}{section}

\begin{document}

\title{\bf An optimal result for global existence and boundedness in a three-dimensional Keller-Segel(-Navier)-Stokes system (involving a tensor-valued sensitivity with saturation)
}

\author{
Jiashan Zheng\thanks{Corresponding author.   E-mail address:
 zhengjiashan2008@163.com (J.Zheng)}
 \\
    School of Mathematics and Statistics Science,\\
     Ludong University, Yantai 264025,  P.R.China \\
}
\date{}


\maketitle \vspace{0.3cm}
\noindent
\begin{abstract}
This paper is concerned with the following Keller-Segel(-Navier)-Stokes system with  (rotational flux)
$$
 \left\{
 \begin{array}{l}
   n_t+u\cdot\nabla n=\Delta n-\nabla\cdot(nS(x,n,c)\nabla c),\quad
x\in \Omega, t>0,\\
    c_t+u\cdot\nabla c=\Delta c-c+n,\quad
x\in \Omega, t>0,\\
u_t+\kappa(u \cdot \nabla)u+\nabla P=\Delta u+n\nabla \phi,\quad
x\in \Omega, t>0,\\
\nabla\cdot u=0,\quad
x\in \Omega, t>0\\
 \end{array}\right.\eqno(KSNF)
 $$
 in a bounded domain $\Omega\subset \mathbb{R}^3$ with smooth boundary, where $\kappa\in \mathbb{R}$ is given
constant,
$\phi\in W^{1,\infty}(\Omega)$,
$|S(x,n,c)|\leq C_S(1+n)^{-\alpha}$ and the  parameter $\alpha\geq0$.
If $\alpha>\frac{1}{3}$,
then for all reasonably regular initial data, a corresponding initial-boundary value
problem for $(KSNF)$ possesses a globally defined weak solution. 
This result improves  the result of Wang (Math. Models Methods Appl. Sci., 27(14):2745--2780, 2017), where the global {\bf very weak} solution for system $(KSNF)$ is obtained.
Moreover, if
$\kappa=0$ and $S(x,n,c)=C_S(1+n)^{-\alpha}$, then
 the system  $(KSF)$
exists  at least one global   classical solution which is  bounded in
$\Omega\times(0,\infty)$.
 In comparison to the result
for the corresponding fluid-free system, the {\bf optimal condition} on the parameter $\alpha$ for both {\bf global (weak) existence} and {\bf boundedness}  are obtained.
Our proofs rely on Maximal Sobolev regularity techniques and  a variant of the
natural gradient-like energy functional. 

\end{abstract}

\vspace{0.3cm}


\newpage
\section{Introduction}
Chemotaxis, the biased movement of cells (or organisms) in response to chemical gradients, plays
an important role coordinating cell migration in many biological phenomena (see Hillen and Painter \cite{Hillen}). Let $n$ denote the density of the cells and $c$ present the concentration of the chemical
signal. In 1970s, Keller and Segel (\cite{Keller2710}) proposed a mathematical system
 for chemotaxis phenomena through a
system of parabolic equations. The mathematical model reads as
\begin{equation}
 \left\{\begin{array}{ll}
 n_t=\Delta n-\nabla\cdot( nS(x,n,c)\nabla c),
 \quad
x\in \Omega,~ t>0,\\
 \disp{ c_t=\Delta c- c +n,}\quad
x\in \Omega, ~t>0,\\
 \end{array}\right.\label{722dff344101.ddgghhhff2ffggffggx16677}
\end{equation}
where $S$ is a given chemotactic sensitivity function, which can either be a scalar function, or more
general a tensor valued function (see e.g. Xue and Othmer \cite{Xusddeddff345511215}). During the past four decades, the Keller-Segel models \dref{722dff344101.ddgghhhff2ffggffggx16677} and
its variants have attracted extensive attentions, where the main issue of the investigation was whether the
solutions of the models are bounded or blow-up (see Winkler et al. \cite{Bellomo1216}, Hillen and Painter \cite{Hillen}, Horstmann \cite{Horstmann2710}). For
instance, if $S:=S(n)$  is
a scalar function satisfying
$S(s)\leq  C(1 + s)^{-\alpha}$
 for all $s\geq1$ and some $\alpha > 1 -\frac{2}{N}$ and $C>0$, then all solutions to the corresponding
Neumann problem are global and uniformly bounded (see  Horstmann and Winkler \cite{Horstmann791}). While, if $N\geq2$,
$\Omega\subset R^N$ is
a ball and $S(s) > cs^{-\alpha}$
 for some $\alpha < 1-\frac{2}{N}$ and $c>0,$ then the solution to problem \dref{722dff344101.ddgghhhff2ffggffggx16677} may blow up (see Horstmann and Winkler \cite{Horstmann791}).
  Therefore,
\begin{equation}\alpha=1-\frac{2}{N}
\label{722dff344101.ddff2ffddfffggffggx16677}
\end{equation}
is the critical blow-up exponent, which is related to the presence of a so-called volume-filling effect. 
 For the more related works in this direction, we mention that a corresponding quasilinear
version, the logistic damping or the signal is  consumed
by the cells
has been deeply investigated by  Cie\'{s}lak and Stinner \cite{Cie791,Cie201712791}, Tao and Winkler \cite{Tao794,Winkler79,Winkler72} and Zheng et al. \cite{Zheng00,Zheng33312186,Zhengssssdefr23,Zhengssdefr23,Zhengssssssdefr23}.

As in the classical Keller-Segel model where the chemoattractant is produced by bacteria, the corresponding chemotaxis-fluid model is then Keller-Segel(-Navier)-Stokes  system of the form
\begin{equation}
 \left\{\begin{array}{ll}
   n_t+u\cdot\nabla n=\Delta n-\nabla\cdot(nS(x, n, c)\nabla c),\quad
x\in \Omega, t>0,\\
    c_t+u\cdot\nabla c=\Delta c-c+n,\quad
x\in \Omega, t>0,\\
u_t+\kappa (u\cdot\nabla)u+\nabla P=\Delta u+n\nabla \phi,\quad
x\in \Omega, t>0,\\
\nabla\cdot u=0,\quad
x\in \Omega, t>0,\\
 \disp{(\nabla n-nS(x, n, c))\cdot\nu=\nabla c\cdot\nu=0,u=0,}\quad
x\in \partial\Omega, t>0,\\
\disp{n(x,0)=n_0(x),c(x,0)=c_0(x),u(x,0)=u_0(x),}\quad
x\in \Omega,\\
 \end{array}\right.\label{1.1}
\end{equation}
where 
$n$ and $c$ are defined as before, $\Omega\subset \mathbb{R}^3$ is a bounded    domain with smooth boundary.
 Here $u,P,\phi$ and $\kappa\in \mathbb{R}$ denote, respectively, the velocity field, the associated pressure of the fluid, the potential of the
 gravitational field and the
strength of nonlinear fluid convection. 
And
$S(x, n, c)$ is a chemotactic sensitivity tensor satisfying
\begin{equation}\label{x1.73142vghf48rtgyhu}
S\in C^2(\bar{\Omega}\times[0,\infty)^2;\mathbb{R}^{3\times3})
 \end{equation}
 and
 \begin{equation}\label{x1.73142vghf48gg}|S(x, n, c)|\leq C_S(1 + n)^{-\alpha} ~~~~\mbox{for all}~~ (x, n, c)\in\Omega\times [0,\infty)^2
 \end{equation}
with some $C_S > 0$ and $\alpha> 0$.
Problem  \dref{1.1} is proposed to describe
chemotaxis--fluid interaction in cases when the evolution of the chemoattractant is essentially dominated
by production through cells (see Winkler et al. \cite{Bellomo1216}, Hillen and Painter \cite{Hillen}).



Before going into our mathematical analysis, we recall some important progresses on system \dref{1.1} and
its variants.  The
following chemotaxis--fluid model was proposed by Tuval et al. \cite{Tuval1215}, which is a closely related variant of \dref{1.1}
\begin{equation}
 \left\{\begin{array}{ll}
   n_t+u\cdot\nabla n=\Delta n-\nabla\cdot( nS(x,n,c)\nabla c),\quad
x\in \Omega, t>0,\\
    c_t+u\cdot\nabla c=\Delta c-nf(c),\quad
x\in \Omega, t>0,\\
u_t+\kappa (u\cdot\nabla)u+\nabla P=\Delta u+n\nabla \phi,\quad
x\in \Omega, t>0,\\
\nabla\cdot u=0,\quad
x\in \Omega, t>0,\\
 \end{array}\right.\label{1.1hhjffggjddssggtyy}
\end{equation}
where $f(c)$ is the consumption rate of the oxygen by the cells.
In the last few years, by making use of energy-type functionals, system \dref{1.1hhjffggjddssggtyy} and
its variants have attracted extensive attentions (see e.g.
Chae et. al. \cite{Chaexdd12176},
Duan et. al. \cite{Duan12186},
Liu and Lorz  \cite{Liu1215,Lorz1215},
 Tao and Winkler   \cite{Tao41215,Winkler31215,Winkler61215,Winkler51215}, Zhang and Zheng \cite{Zhang12176} and references therein). For example,
Winkler (\cite{Winkler51215}) established the global existence
of weak solution in a three-dimensional domain when $S(x, n, c) \equiv 1$ and $\kappa\neq0$. 
 For more literatures related to this model, we can refer
to  Tao and Winkler \cite{Tao61215,Tao71215} and the reference therein.

If the chemotactic sensitivity $S(x, n, c)$ is regarded as a tensor  rather than a scalar one (see  Xue and Othmer \cite{Xusddeddff345511215}), \dref{1.1hhjffggjddssggtyy}
turns into a chemotaxis(-Navier)-Stokes system with rotational
flux. Due to the presence of the
tensor-valued sensitivity, the corresponding chemotaxis-Stokes system loses some energy structure,
which plays a key
role in previous studies for the scalar sensitivity case (see Cao \cite{Cao22119}, Winkler \cite{Winkler11215}).
Therefore,
only very few results appear to be available on chemotaxis--Stokes system with such tensor-valued
sensitivities (see e.g. Ishida \cite{Ishida1215},  Wang et al. \cite{Wang11215,Wang21215} and Winkler \cite{Winkler11215}).
In fact, when
assuming $f(c)=c$ and \dref{x1.73142vghf48rtgyhu}--\dref{x1.73142vghf48gg} holds,
 Ishida (\cite{Ishida1215}) proved   that
 \dref{1.1hhjffggjddssggtyy}
%
admits
a bounded global weak solution in 2-dimensions with nonlinear diffusion.
While, in 3-dimensions,
Winkler (see Winkler \cite{Winkler11215}) showed  that the chemotaxis--Stokes system ($\kappa=0$ in the first equation of \dref{1.1hhjffggjddssggtyy}) with the nonlinear diffusion (the coefficient of diffusion satisfies $m > \frac{7}{6}$)  possesses at least one bounded weak solution
which stabilizes to the spatially homogeneous equilibrium  $(\frac{1}{|\Omega|}\int_{\Omega}n_0, 0, 0)$.

In contrast to the large number of the existed results on \dref{1.1hhjffggjddssggtyy},
the mathematical analysis of \dref{1.1} regarding global and bounded solutions is far from
trivial, since,  on the one hand its Navier-Stokes subsystem lacks complete existence theory (see Wiegner \cite{Wiegnerdd79})
and on the other hand the previously mentioned properties for Keller-Segel system can still
emerge (see Wang, Xiang et. al.  \cite{Peng55667,Wang21215,Wangss21215,Wangssddss21215}, Zheng \cite{Zhengsddfff00,Zhenddddgssddsddfff00}). In fact, in 2-dimensional, if $S=S(x, n, c)$ is a tensor-valued sensitivity fulfilling \dref{x1.73142vghf48rtgyhu}
and \dref{x1.73142vghf48gg}, Wang and Xiang (\cite{Wang21215}) proved that Stokes-version ($\kappa=0$ in the first equation of \dref{1.1}) of system \dref{1.1}
admits a unique global classical solution which is bounded. These condition for $\alpha$ is optimal according to \dref{722dff344101.ddff2ffddfffggffggx16677}.
 And similar results are also valid for the three-dimensional Stokes-version ($\kappa=0$ in the first equation of \dref{1.1}) of system \dref{1.1}
 with $\alpha>\frac{1}{2}$ (see Wang and Xiang \cite{Wangss21215}). While if  3-dimensional, Wang and Liu (\cite{LiuZhLiuLiuandddgddff4556}) showed
 that Keller-Segel-Navier-Stokes ($\kappa\neq0$ in the first equation of \dref{1.1})  system \dref{1.1} admits a global weak solutions for
 tensor-valued sensitivity
$S(x, n, c)$ satisfying \dref{x1.73142vghf48rtgyhu}
 and \dref{x1.73142vghf48gg} with $\alpha > \frac{3}{7}$. Recently, due to the lack of enough regularity and compactness properties for the first equation, by using the idea originating  from Winkler (see Winkler \cite{Winklerddfff51215}), Wang (see Wang \cite{Wangssddss21215})
 obtained the global {\bf very weak} solutions system \dref{1.1} under the assumption that $S$  satisfies \dref{x1.73142vghf48rtgyhu} and \dref{x1.73142vghf48gg} with  $\alpha > \frac{1}{3}$, which in light of the known
results for the fluid-free system mentioned above is an optimal restriction on $\alpha$ (see \dref{722dff344101.ddff2ffddfffggffggx16677}). However, for the  global (strongly than the result of \cite{Wangssddss21215}) {\bf  weak} solutions is still open.
 In this paper,
 we try to obtain the enough regularity and compactness properties (see Lemmas \ref{lemmddaghjsffggggsddgghhmk4563025xxhjklojjkkk}, \ref{4455lemma45630hhuujjuuyytt} and \ref{qqqqlemma45630hhuujjuuyytt}), then  show that system \dref{1.1} possesses a globally defined {\bf weak} solution (see Definition \ref{df1}), which improves the result of \cite{Wangssddss21215}. Moreover, with the help of Maximal Sobolev regularity and some carefully analysis,
 if $S:=S(n)=C_S(1+n)^{-\alpha}$ is  a scalar function which satisfies that $\alpha>\frac{1}{3},$ the {\bf boundedness} of solution to Keller-Segel-Stokes ($\kappa=0$ in the first equation of \dref{1.1}) system
 \dref{1.1} is also obtained. Recalling the condition \dref{722dff344101.ddff2ffddfffggffggx16677} for global existence in the fluid-free setting, as implied by the
previously mentioned studied (see Horstmann and Winkler \cite{Horstmann791}), this result appears to be optimal with respect to $\alpha$.

 We sketch here the main ideas and methods used in this article.
One  novelty of this
paper is that we use the
Maximal Sobolev regularity (see Hieber and Pr\"{u}ss \cite{Hieber}) approach to show the existence of bounded solutions. The
Maximal Sobolev regularity approach
 has been widely used to obtain the existence of bounded solutions of  the quasilinear
parabolic--parabolic Keller--Segel system with
{\bf logistic source } (see e.g. Cao \cite{Cao} and Zheng \cite{Zhengssdddssddddkkllssssssssdefr23}). However, it seems that no one used such method to obtain
the existence of bounded solutions to  Keller--Segel-Stokes system. We should pointed that the idea of this paper can also be used to deal with Keller--Segel-Stokes system with {\bf nonlinear diffusion} (see Zheng \cite{Zhenddddgssddsddfff00}). In fact, by using the idea of this paper, one can prove that if
the coefficient of diffusion satisfies $m > \frac{4}{3}$, then Keller--Segel-Stokes system (with {\bf nonlinear diffusion})
exists  at least one global  weak solution which is  bounded in
$\Omega\times(0,\infty)$. The conditions $m > \frac{4}{3}$ is also  optimal due to
the fact that the 3D fluid-free system  admits a global bounded classical solution for $m > \frac{4}{3}$ (see the Introduction of Tao and Winkler \cite{Tao794}).



Throughout this paper,
we assume that 
\begin{equation}
\phi\in W^{1,\infty}(\Omega)
\label{dd1.1fghyuisdakkkllljjjkk}
\end{equation}
 and the initial data
$(n_0, c_0, u_0)$ fulfills
\begin{equation}\label{ccvvx1.731426677gg}
\left\{
\begin{array}{ll}
\displaystyle{n_0\in C^\kappa(\bar{\Omega})~~\mbox{for certain}~~ \kappa > 0~~ \mbox{with}~~ n_0\geq0 ~~\mbox{in}~~\Omega},\\
\displaystyle{c_0\in W^{1,\infty}(\Omega)~~\mbox{with}~~c_0\geq0~~\mbox{in}~~\bar{\Omega},}\\
\displaystyle{u_0\in D(A^\gamma_{r})~~\mbox{for~~ some}~~\gamma\in ( \frac{3}{4}, 1)~~\mbox{and any}~~ {r}\in (1,\infty),}\\
\end{array}
\right.
\end{equation}
where $A_{r}$ denotes the Stokes operator with domain $D(A_{r}) := W^{2,{r}}(\Omega)\cap  W^{1,{r}}_0(\Omega)
\cap L^{r}_{\sigma}(\Omega)$,
and
$L^{r}_{\sigma}(\Omega) := \{\varphi\in  L^{r}(\Omega)|\nabla\cdot\varphi = 0\}$ for ${r}\in(1,\infty)$
 (\cite{Sohr}).

In the context of these assumptions, the first of our main results asserts global weak existence of a
solution in the following sense.
\begin{theorem}\label{theorem3}
Let  $\Omega\subset \mathbb{R}^3$ be a bounded    domain with smooth boundary,
 \dref{dd1.1fghyuisdakkkllljjjkk} and \dref{ccvvx1.731426677gg}
 hold, and suppose that $S$ satisfies \dref{x1.73142vghf48rtgyhu} and \dref{x1.73142vghf48gg}
with some
\begin{equation}\label{x1.73142vghf48}\alpha>\frac{1}{3}.
\end{equation}
Then 
 the problem \dref{1.1} possesses at least
one global weak solution $(n, c, u, P)$
 in the sense of Definition \ref{df1}. 
\end{theorem}
\begin{remark}
(i)  From Theorem \ref{theorem3}, we conclude that  if
 the algebraic saturation with $\alpha > \frac{1}{3}$ is sufficient to guarantee
the existence of global (weak) solutions. Compared to the results \dref{722dff344101.ddff2ffddfffggffggx16677}, we know
such a restriction on $\alpha$ is optimal.

(ii) Obviously, $  \frac{3}{7}>\frac{1}{2}$,  Theorem \ref{theorem3} improves the results of Liu and Wang (\cite{LiuZhLiuLiuandddgddff4556}),  who showed the global weak existence of solutions in
 the cases $S(x, n, c)$ satisfying \dref{x1.73142vghf48rtgyhu}
 and \dref{x1.73142vghf48gg} with $\alpha > \frac{3}{7}$.
\end{remark}
Moreover, if in addition we assume that $\kappa=0$ and $S(x,n,c)= C_S(1+n)^{-\alpha}$, then the solutions will actually be bounded:
\begin{theorem}\label{theoremddffggg3}
Let  $\Omega\subset \mathbb{R}^3$ be a bounded    domain with smooth boundary,
 \dref{dd1.1fghyuisdakkkllljjjkk} and \dref{ccvvx1.731426677gg}
 hold. Moreover, assume that $\kappa=0$ and $S(x,n,c)= C_S(1+n)^{-\alpha}$,
then for any choice of $n_0, c_0$ and $u_0$ fulfilling \dref{ccvvx1.731426677gg},
 the problem \dref{1.1}  possesses a global classical solution $(n, c, u, P)$ for which $n, c$
and $u$ are bounded in
$\Omega\times(0,\infty)$ in the sense that there exists $C > 0$ fulfilling
\begin{equation}
\|n(\cdot, t)\|_{L^\infty(\Omega)}+\|c(\cdot, t)\|_{W^{1,\infty}(\Omega)}+\| u(\cdot, t)\|_{W^{1,\infty}(\Omega)}\leq C~~ \mbox{for all}~~ t>0.
\label{1.163072xggttyyu}
\end{equation}
\end{theorem}
\begin{remark}
(i) If $u\equiv0$,  Theorem \ref{theoremddffggg3} is coincides with
Theorem 4.1 of \cite{Horstmann791}, which is optimal according to
the fact that the 3D fluid-free system \dref{722dff344101.ddgghhhff2ffggffggx16677} admits a global bounded classical solution  for $\alpha>\frac{1}{3}$
 as
mentioned before.

(ii) The condition of $S(x,n,c)= C_S(1+n)^{-\alpha}$ can be replaced by $S:=S(n)$ which satisfies  $S(n)\leq C_S(1+n)^{-\alpha}.$
\end{remark}

This paper is organized as follows. In Section 2, we firstly give  the definition of weak solutions to \dref{1.1}, the  regularized problems of \dref{1.1} and
  some preliminary properties.
%
Section 3  and Section 4 will be devoted to an analysis of regularized problems of \dref{1.1}.
Next,  on the basis of the compactness properties thereby implied, in Section 5 and Section 6 we can 
pass to the limit along an adequate sequence of numbers $\varepsilon = \varepsilon_j\searrow0$
and thereby verify the Theorem  \ref{theorem3}. In Section 7, in view of the Maximal Sobolev regularity techniques, we will show Theorem \ref{theoremddffggg3} by applying the standard Alikakos-Moser iteration.
Indeed,  by using the Maximal Sobolev regularity techniques, we firstly,
establish an energy-type inequality which will play a key role in the
derivation of further estimates.  Then, we develop some $L^p$-estimate techniques to raise the a priori estimate of solutions from $L^1(\Omega)\rightarrow
L^{q_0}(\Omega) (q_0>\frac{3}{2})$, and then use the standard Alikakos-Moser iteration and  the standard parabolic regularity
arguments to show Theorem \ref{theoremddffggg3}.

\section{Preliminaries}
In light of the strongly nonlinear term $(u \cdot \nabla)u$, the problem \dref{1.1} has no
classical solutions in general, and thus we consider its weak solutions. The concept of (global) weak solution for \dref{1.1}  we shall purse in
this sequel will be given in the follows.
%
%
\begin{definition}\label{df1}
Let $T > 0$ and $(n_0, c_0, u_0)$ fulfills
\dref{ccvvx1.731426677gg}.
Then a triple of functions $(n, c, u)$ is
called a weak solution of \dref{1.1} if the following conditions are satisfied
\begin{equation}
 \left\{\begin{array}{ll}
   n\in L_{loc}^1(\bar{\Omega}\times[0,T)),\\
    c \in L_{loc}^1([0,T); W^{1,1}(\Omega)),\\
u \in  L_{loc}^1([0,T); W^{1,1}(\Omega)),\\
 \end{array}\right.\label{dffff1.1fghyuisdakkklll}
\end{equation}
where $n\geq 0$ and $c\geq 0$ in
$\Omega\times(0, T)$ as well as $\nabla\cdot u = 0$ in the distributional sense in
 $\Omega\times(0, T)$,
moreover,
\begin{equation}\label{726291hh}
\begin{array}{rl}
 &u\otimes u \in L^1_{loc}(\bar{\Omega}\times [0, \infty);\mathbb{R}^{3\times 3})~~\mbox{and}~~~ n~\mbox{belong to}~~ L^1_{loc}(\bar{\Omega}\times [0, \infty)),\\
  &cu,~ ~nu ~~\mbox{and}~~nS(x,n,c)\nabla c~ \mbox{belong to}~~
L^1_{loc}(\bar{\Omega}\times [0, \infty);\mathbb{R}^{3})
\end{array}
\end{equation}
and
\begin{equation}
\begin{array}{rl}\label{eqx45xx12112ccgghh}
\disp{-\int_0^{T}\int_{\Omega}n\varphi_t-\int_{\Omega}n_0\varphi(\cdot,0)  }=&\disp{-
\int_0^T\int_{\Omega}\nabla n\cdot\nabla\varphi+\int_0^T\int_{\Omega}n
S(x,n,c)\nabla c\cdot\nabla\varphi}\\
&+\disp{\int_0^T\int_{\Omega}nu\cdot\nabla\varphi}\\
\end{array}
\end{equation}
for any $\varphi\in C_0^{\infty} (\bar{\Omega}\times[0, T))$ satisfying
 $\frac{\partial\varphi}{\partial\nu}= 0$ on $\partial\Omega\times (0, T)$
  as well as
  \begin{equation}
\begin{array}{rl}\label{eqx45xx12112ccgghhjj}
\disp{-\int_0^{T}\int_{\Omega}c\varphi_t-\int_{\Omega}c_0\varphi(\cdot,0)  }=&\disp{-
\int_0^T\int_{\Omega}\nabla c\cdot\nabla\varphi-\int_0^T\int_{\Omega}c\varphi+\int_0^T\int_{\Omega}n\varphi+
\int_0^T\int_{\Omega}cu\cdot\nabla\varphi}\\
\end{array}
\end{equation}
for any $\varphi\in C_0^{\infty} (\bar{\Omega}\times[0, T))$  and
\begin{equation}
\begin{array}{rl}\label{eqx45xx12112ccgghhjjgghh}
\disp{-\int_0^{T}\int_{\Omega}u\varphi_t-\int_{\Omega}u_0\varphi(\cdot,0) -\kappa
\int_0^T\int_{\Omega} u\otimes u\cdot\nabla\varphi }=&\disp{-
\int_0^T\int_{\Omega}\nabla u\cdot\nabla\varphi-
\int_0^T\int_{\Omega}n\nabla\phi\cdot\varphi}\\
\end{array}
\end{equation}
for any $\varphi\in C_0^{\infty} (\bar{\Omega}\times[0, T);\mathbb{R}^3)$ fulfilling
$\nabla\varphi\equiv 0$ in
 $\Omega\times(0, T)$.
 If $\Omega\times (0,\infty)\longrightarrow \mathbb{R}^5$ is a weak solution of \dref{1.1} in
 $\Omega\times(0, T)$ for all $T > 0$, then we call
$(n, c, u)$ a global weak solution of \dref{1.1}.
\end{definition}

Our goal is to construct solutions of \dref{1.1} as limits of solutions to appropriately regularized problems.
To achieve this, 
in order to deal with the strongly nonlinear term $(u \cdot \nabla)u$,
%
we introduce the following approximating equation of \dref{1.1}:
\begin{equation}
 \left\{\begin{array}{ll}
   n_{\varepsilon t}+u_{\varepsilon}\cdot\nabla n_{\varepsilon}=\Delta n_{\varepsilon}-\nabla\cdot(n_{\varepsilon}F'_{\varepsilon}(n_{\varepsilon})S_\varepsilon(x, n_{\varepsilon}, c_{\varepsilon})\nabla c_{\varepsilon}),\quad
x\in \Omega, t>0,\\
    c_{\varepsilon t}+u_{\varepsilon}\cdot\nabla c_{\varepsilon}=\Delta c_{\varepsilon}-c_{\varepsilon}+F_{\varepsilon}(n_{\varepsilon}),\quad
x\in \Omega, t>0,\\
u_{\varepsilon t}+\nabla P_{\varepsilon}=\Delta u_{\varepsilon}-\kappa (Y_{\varepsilon}u_{\varepsilon} \cdot \nabla)u_{\varepsilon}+n_{\varepsilon}\nabla \phi,\quad
x\in \Omega, t>0,\\
\nabla\cdot u_{\varepsilon}=0,\quad
x\in \Omega, t>0,\\
 \disp{\nabla n_{\varepsilon}\cdot\nu=\nabla c_{\varepsilon}\cdot\nu=0,u_{\varepsilon}=0,\quad
x\in \partial\Omega, t>0,}\\
\disp{n_{\varepsilon}(x,0)=n_0(x),c_{\varepsilon}(x,0)=c_0(x),u_{\varepsilon}(x,0)=u_0(x)},\quad
x\in \Omega,\\
 \end{array}\right.\label{1.1fghyuisda}
\end{equation}
where
\begin{equation}
F_{\varepsilon}(s)=\frac{1}{\varepsilon}\ln(1+\varepsilon s)~~\mbox{for all}~~s \geq 0~~\mbox{and}~~\varepsilon> 0
\label{1.ffggvbbnxxccvvn1}
\end{equation}
as well as
\begin{equation}
\begin{array}{ll}
S_\varepsilon(x, n, c) = \rho_\varepsilon(x)S(x, n, c),~~ x\in\bar{\Omega},~~n\geq0,~~c\geq0
 \end{array}\label{3.10gghhjuuloollyuigghhhyy}
\end{equation}
and
\begin{equation}
 \begin{array}{ll}
 Y_{\varepsilon}w := (1 + \varepsilon A)^{-1}w ~~~~\mbox{for all}~~ w\in L^2_{\sigma}(\Omega)
 \end{array}\label{aasddffgg1.1fghyuisda}
\end{equation}
is the standard Yosida approximation.
Here $(\rho_\varepsilon)_{\varepsilon\in(0,1)} \in C^\infty_0 (\Omega)$
  be a family of standard cut-off functions satisfying $0\leq\rho_\varepsilon\leq 1$
   in $\Omega$
 and $\rho_\varepsilon\rightarrow1$ in $\Omega$
 as $\varepsilon\rightarrow0$.


The local solvability of \dref{1.1fghyuisda} can be derived by  a suitable
extensibility criterion and a slight modification of the well-established fixed point arguments in Lemma 2.1 of \cite{Winkler51215} (see also \cite{Winkler11215}, Lemma 2.1 of \cite{Painter55677}), so here we omit the proof.

%
%
%
\begin{lemma}\label{lemma70}
Assume
that $\varepsilon\in(0,1).$
%
Then there exist $T_{max,\varepsilon}\in  (0,\infty]$ and
a classical solution $(n_\varepsilon, c_\varepsilon, u_\varepsilon, P_\varepsilon)$ of \dref{1.1fghyuisda} in
$\Omega\times(0, T_{max,\varepsilon})$ such that
\begin{equation}
 \left\{\begin{array}{ll}
 n_\varepsilon\in C^0(\bar{\Omega}\times[0,T_{max,\varepsilon}))\cap C^{2,1}(\bar{\Omega}\times(0,T_{max,\varepsilon})),\\
  c_\varepsilon\in  C^0(\bar{\Omega}\times[0,T_{max,\varepsilon}))\cap C^{2,1}(\bar{\Omega}\times(0,T_{max,\varepsilon})),\\
  u_\varepsilon\in  C^0(\bar{\Omega}\times[0,T_{max,\varepsilon}))\cap C^{2,1}(\bar{\Omega}\times(0,T_{max,\varepsilon})),\\
  P_\varepsilon\in  C^{1,0}(\bar{\Omega}\times(0,T_{max,\varepsilon})),\\
   \end{array}\right.\label{1.1ddfghyuisda}
\end{equation}
 classically solving \dref{1.1fghyuisda} in $\Omega\times[0,T_{max,\varepsilon})$.
%
Moreover,  $n_\varepsilon$ and $c_\varepsilon$ are nonnegative in
$\Omega\times(0, T_{max,\varepsilon})$, and
\begin{equation}
\|n_\varepsilon(\cdot, t)\|_{L^\infty(\Omega)}+\|c_\varepsilon(\cdot, t)\|_{W^{1,\infty}(\Omega)}+\|A^\gamma u_\varepsilon(\cdot, t)\|_{L^{2}(\Omega)}\rightarrow\infty~~ \mbox{as}~~ t\rightarrow T_{max,\varepsilon},
\label{1.163072x}
\end{equation}
where $\gamma$ is given by \dref{ccvvx1.731426677gg}.
\end{lemma}
\begin{lemma}(\cite{Winkler792,Zhengddkkllssssssssdefr23})\label{llssdrffmmggnnccvvccvvkkkkgghhkkllvvlemma45630}
Let $(e^{\tau\Delta})_{\tau\geq0}$ be the Neumann heat semigroup in $\Omega$ and $p>3$. Then there exist positive constants
$c_1:=c_1(\Omega),$ $c_2:=c_2(\Omega)$ and $c_3:=c_3(\Omega)$ such that
\begin{equation}
\|\nabla e^{\tau\Delta}\varphi\|_{L^p(\Omega)} \leq c_1(\Omega)\|\nabla\varphi\|_{L^p(\Omega)} ~~~\mbox{for all}~~~ \tau > 0~~~\mbox{and any}~~~\varphi\in W^{1,p}(\Omega)
\label{gghjjccmmllffvvggcvvvvbbjjkkdffzjscz2.5297x9630xxy}
\end{equation}
and
\begin{equation}
\|\nabla e^{\tau\Delta}\varphi\|_{L^p(\Omega)} \leq c_2(1+\tau^{-\frac{1}{2}})\|\varphi\|_{L^\infty(\Omega)} ~~~\mbox{for all}~~~ \tau > 0~~~\mbox{and each}~~~\varphi\in L^\infty(\Omega)
\label{llmmmgghjjccmmllffvvggcvvvvbbjjkkdffzjscz2.5297x9630xxy}
\end{equation}
as well as
\begin{equation}
\begin{array}{rl}
&\| e^{\tau\Delta}\nabla\cdot\varphi\|_{L^\infty(\Omega)} \\
\leq& c_3(1+\tau^{-\frac{1}{2}-\frac{3}{2p}})\|\varphi\|_{L^p(\Omega)} ~\mbox{for all}~ \tau > 0~\mbox{and all}~\varphi\in C^{1}(\bar{\Omega}; \mathbb{R}^N)~\mbox{fulfilling}~\varphi\cdot\nu=0~\mbox{on}~\partial\Omega.\\
\end{array}
\label{gggghjjccmmllffvvggcvvvvbbjjkkdffzjscz2.5297x9630xxy}
\end{equation}
\end{lemma}

\section{A priori estimates for the regularized problem \dref{1.1fghyuisda} which is independent of $\varepsilon$ }
In this section, we are going to establish an iteration step to develop the main ingredient of our result. The iteration depends on a series of a priori estimate.
In order to proceed, firstly, we recall some properties for $F_\varepsilon$  and $F'_\varepsilon$, which paly important rule in showing Theorem \ref{theorem3}.

\begin{lemma}\label{fvfgfffffflemma45}
Assume $F_{\varepsilon}$ is given by \dref{1.ffggvbbnxxccvvn1}.
Then
\begin{equation}
0\leq F'_{\varepsilon}(s)=\frac{1}{1+\varepsilon s}\leq1~~\mbox{for all}~~s \geq 0~~\mbox{and}~~\varepsilon> 0
\label{1.ffggvddfghhghhhhbbnxxccvvn1}
\end{equation}
as well as
\begin{equation}
\lim_{\varepsilon\rightarrow 0^{+}}F_{\varepsilon}(s)=s,~~~\lim_{\varepsilon\rightarrow 0^{+}}F'_{\varepsilon}(s)=1~~\mbox{for all}~~s \geq 0
\label{1.ffggvddfghhghhhhbddgbnxxccvvn1}
\end{equation}
and
\begin{equation}
0\leq F_{\varepsilon}(s)\leq s~~\mbox{for all}~~s \geq 0.
\label{ffggg1.ffggvddfghhghhhhbbhhjjjnxxccvvn1}
\end{equation}

\end{lemma}
\begin{proof}
Recalling \dref{1.ffggvbbnxxccvvn1}, by tedious
but simple calculations, we can derive \dref{1.ffggvddfghhghhhhbbnxxccvvn1}--\dref{ffggg1.ffggvddfghhghhhhbbhhjjjnxxccvvn1}.
\end{proof}
The proof of this lemma is very similar to that of Lemmas 2.2 and 2.6  of  \cite{Tao41215} (see also Lemma 3.2 of \cite{Wangssddss21215}), so we omit its proof here.
\begin{lemma}\label{fvfgfflemma45}
There exists 
$\lambda > 0$ independent of $\varepsilon$ such that the solution of \dref{1.1fghyuisda} satisfies
%
%
\begin{equation}
\int_{\Omega}{n_{\varepsilon}}+\int_{\Omega}{c_{\varepsilon}}\leq \lambda~~\mbox{for all}~~ t\in(0, T_{max,\varepsilon}).
\label{ddfgczhhhh2.5ghju48cfg924ghyuji}
\end{equation}
%
%
\end{lemma}
\begin{lemma}\label{lemmaghjffggssddgghhmk4563025xxhjklojjkkk}
Let $\alpha>\frac{1}{3}$.
Then there exists $C>0$ independent of $\varepsilon$ such that the solution of \dref{1.1fghyuisda} satisfies
\begin{equation}
\begin{array}{rl}
&\disp{\int_{\Omega} n_{\varepsilon}^{2\alpha }+\int_{\Omega}   c_{\varepsilon}^2+\int_{\Omega}  | {u_{\varepsilon}}|^2\leq C~~~\mbox{for all}~~ t\in (0, T_{max,\varepsilon}).}\\
\end{array}
\label{czfvgb2.5ghhjuyuccvviihjj}
\end{equation}
Moreover, for $T\in(0, T_{max,\varepsilon})$, it holds that
one can find a constant $C > 0$ independent of $\varepsilon$ such that
\begin{equation}
\begin{array}{rl}
&\disp{\int_{0}^T\int_{\Omega} \left[  n_{\varepsilon}^{2\alpha-2} |\nabla {n_{\varepsilon}}|^2+ |\nabla {c_{\varepsilon}}|^2+ |\nabla {u_{\varepsilon}}|^2\right]\leq C.}\\
\end{array}
\label{bnmbncz2.5ghhjuyuivvbnnihjj}
\end{equation}
\end{lemma}
\begin{proof}
The proof consists two cases.

Case£º $2\alpha\neq1:$ 
We first obtain from $\nabla\cdot u_\varepsilon=0$ in
 $\Omega\times (0, T_{max,\varepsilon})$ and straightforward calculations
that
\begin{equation}
\begin{array}{rl}
&\disp{sign(2\alpha-1)\frac{1}{{2\alpha }}\frac{d}{dt}\| n_{\varepsilon} \|^{{{2\alpha }}}_{L^{{2\alpha }}(\Omega)}}\\
&\disp{+
sign(2\alpha-1)(2\alpha-1)\int_{\Omega}  n_{\varepsilon}^{2\alpha-2} |\nabla n_{\varepsilon}|^2}\\
=&\disp{-
\int_{\Omega}sign(2\alpha-1)  n_{\varepsilon}^{2\alpha-1}\nabla\cdot(n_{\varepsilon}F'_{\varepsilon}(n_{\varepsilon})S_\varepsilon(x, n_{\varepsilon}, c_{\varepsilon})\cdot\nabla c_{\varepsilon})}\\
\leq&\disp{sign(2\alpha-1)(2\alpha-1)
\int_{\Omega}  n_{\varepsilon}^{2\alpha-2}n_{\varepsilon}F'_{\varepsilon}(n_{\varepsilon})|S_\varepsilon(x, n_{\varepsilon}, c_{\varepsilon})||\nabla n_{\varepsilon}||\nabla c_{\varepsilon}|}
\end{array}
\label{55hhjjcffghhhjkkllz2.5}
\end{equation}
for all $t\in(0, T_{max,\varepsilon}).$
  Therefore, due to \dref{1.ffggvddfghhghhhhbbnxxccvvn1},  
 in light of \dref{x1.73142vghf48gg} and \dref{1.ffggvbbnxxccvvn1}, with the help of  the Young inequality, we can estimate the right of \dref{55hhjjcffghhhjkkllz2.5} by following
\begin{equation}
\begin{array}{rl}
&\disp{sign(2\alpha-1)(2\alpha-1)
\int_{\Omega}  n_{\varepsilon}^{2\alpha-2}n_{\varepsilon}F'_{\varepsilon}(n_{\varepsilon})|S_\varepsilon(x, n_{\varepsilon}, c_{\varepsilon})||\nabla n_{\varepsilon}||\nabla c_{\varepsilon}|}\\
\leq&\disp{sign(2\alpha-1)(2\alpha-1)
\int_{\Omega}  n_{\varepsilon}^{2\alpha-2}n_{\varepsilon}C_S(1 + n_{\varepsilon})^{-\alpha}|\nabla n_{\varepsilon}||\nabla c_{\varepsilon}|}\\
\leq&\disp{sign(2\alpha-1)\frac{2\alpha-1}{2}\int_{\Omega}  n_{\varepsilon}^{2\alpha-2} |\nabla n_{\varepsilon}|^2}\\
&\disp{+\frac{|2\alpha-1|}{2}C_S^2\int_{\Omega}  n_{\varepsilon}^{2\alpha-2}n_{\varepsilon}^2(1 + n_{\varepsilon})^{-2\alpha}|\nabla c_{\varepsilon}|^2}\\
\leq&\disp{sign(2\alpha-1)\frac{2\alpha-1}{2}\int_{\Omega}  n_{\varepsilon}^{2\alpha-2} |\nabla n_{\varepsilon}|^2}\\
&\disp{+\frac{|2\alpha-1|}{2}C_S^2\int_{\Omega}|\nabla c_{\varepsilon}|^2~~\mbox{for all}~~ t\in(0, T_{max,\varepsilon}),}
\end{array}
\label{55hhjjcz2.5}
\end{equation}
where in the last inequality we have used the fact that
$  n_{\varepsilon}^{2\alpha-2}n_{\varepsilon}^2(1 + n_{\varepsilon})^{-2\alpha}\leq1$ for all
$\varepsilon\geq0,$
$n_{\varepsilon}$ and $\alpha\geq0$.
Inserting \dref{55hhjjcz2.5} into \dref{55hhjjcffghhhjkkllz2.5}, we conclude that
\begin{equation}
\begin{array}{rl}
&\disp{sign(2\alpha-1)\frac{1}{{2\alpha }}\frac{d}{dt}\|{n_{\varepsilon}}\|^{{{2\alpha }}}_{L^{{2\alpha }}(\Omega)}+
sign(2\alpha-1)\frac{2\alpha-1}{2}\int_{\Omega}  n_{\varepsilon}^{2\alpha-2} |\nabla n_{\varepsilon}|^2}\\
\leq&\disp{\frac{|2\alpha-1|}{2}C_S^2\int_{\Omega}|\nabla c_{\varepsilon}|^2~~\mbox{for all}~~ t\in(0, T_{max,\varepsilon}).}
\end{array}
\label{55hhjjcffghhhjkkllzddfghh2.5}
\end{equation}
To track the time evolution of $c_\varepsilon$,
taking ${c_{\varepsilon}}$ as the test function for the second  equation of \dref{1.1fghyuisda}, using $\nabla\cdot u_\varepsilon=0$ and \dref{ffggg1.ffggvddfghhghhhhbbhhjjjnxxccvvn1}, with the help of the H\"{o}lder  inequality yields  that
\begin{equation}
\begin{array}{rl}
\disp\frac{1}{{2}}\disp\frac{d}{dt}\|{c_{\varepsilon}}\|^{{{2}}}_{L^{{2}}(\Omega)}+
\int_{\Omega} |\nabla c_{\varepsilon}|^2+ \int_{\Omega} | c_{\varepsilon}|^2=&\disp{\int_{\Omega} F_{\varepsilon}(n_{\varepsilon})c_{\varepsilon}}\\
\leq&\disp{\int_{\Omega} n_{\varepsilon}c_{\varepsilon}}\\
\leq&\disp{\|n_{\varepsilon}\|_{L^{\frac{6}{5}}(\Omega)}\|c_{\varepsilon}\|_{L^{6}(\Omega)}~~\mbox{for all}~~ t\in(0, T_{max,\varepsilon}).}\\
\end{array}
\label{hhxxcdfvvjjcz2.5}
\end{equation}

An application of the Sobolev embedding $W^{1,2}(\Omega)\hookrightarrow L^6(\Omega)$ in the three-dimensional setting, in view of
\dref{ddfgczhhhh2.5ghju48cfg924ghyuji}, there exist positive constants $C_1$ and $C_2$
such that
\begin{equation}
\begin{array}{rl}
\disp \|c_{\varepsilon}\|_{L^{6}(\Omega)}^2\leq&\disp{C_1\|\nabla c_{\varepsilon}\|_{L^{2}(\Omega)}^2+C_1\|c_{\varepsilon}\|_{L^{1}(\Omega)}^2}\\
\leq&\disp{C_1\|\nabla c_{\varepsilon}\|_{L^{2}(\Omega)}^2+C_2~~\mbox{for all}~~ t\in(0, T_{max,\varepsilon}).}\\
\end{array}
\label{hhxxcdfvvdfghhhjjcz2.5}
\end{equation}
Thus by means of the Young inequality and \dref{hhxxcdfvvdfghhhjjcz2.5},
we proceed to estimate
\begin{equation}
\begin{array}{rl}
\disp\frac{1}{{2}}\disp\frac{d}{dt}\|{c_{\varepsilon}}\|^{{{2}}}_{L^{{2}}(\Omega)}+
\int_{\Omega} |\nabla c_{\varepsilon}|^2+ \int_{\Omega} | c_{\varepsilon}|^2\leq&\disp{\frac{1}{2C_1}\|c_{\varepsilon}\|_{L^{6}(\Omega)}^2+\frac{C_1}{2}\| n_{\varepsilon} \|_{L^{\frac{6}{5}}(\Omega)}^2}\\
\leq&\disp{\frac{1}{2}\|\nabla c_{\varepsilon}\|_{L^{2}(\Omega)}^2+\frac{C_1}{2}\| n_{\varepsilon} \|_{L^{\frac{6}{5}}(\Omega)}^2+C_3
~~\mbox{for all}~~ t\in(0, T_{max,\varepsilon})}\\
\end{array}
\label{hhxxcdfvvjjcz2.5dddrr}
\end{equation}
and some positive constant $C_3$ independent of $\varepsilon$.
Therefore,
\begin{equation}
\begin{array}{rl}
\disp\frac{1}{{2}}\disp\frac{d}{dt}\|{c_{\varepsilon}}\|^{{{2}}}_{L^{{2}}(\Omega)}+\frac{1}{2}
\int_{\Omega} |\nabla c_{\varepsilon}|^2+ \int_{\Omega} | c_{\varepsilon}|^2\leq&\disp{\frac{C_1}{2}\| n_{\varepsilon} \|_{L^{\frac{6}{5}}(\Omega)}^2+C_3
~~\mbox{for all}~~ t\in(0, T_{max,\varepsilon}).}\\
\end{array}
\label{hhxxcdfvvjjcz2.5ddsdfrfggggdrr}
\end{equation}
To estimate $\|n_{\varepsilon}\|_{L^{\frac{6}{5}(\Omega)}}$ for all $t\in(0, T_{max,\varepsilon})$, 
we should notice that $ \alpha>\frac{1}{3}$ which ensures
that $\frac{2}{6\alpha-1}<2$, in light of \dref{ddfgczhhhh2.5ghju48cfg924ghyuji},  and hence the Gagliardo--Nirenberg and the Young inequalities allow us to estimate that
%
%
 for any $\delta_1>0,$
\begin{equation}
\begin{array}{rl}
&\disp\| n_{\varepsilon} \|_{L^{\frac{6}{5}}(\Omega)}^2\\
=&\disp{\| n_{\varepsilon}^{\alpha }\|_{L^{\frac{6}{5\alpha}}(\Omega)}^{\frac{2}{\alpha }}}\\
\leq&\disp{C_4(\|\nabla  n_{\varepsilon}^{\alpha }\|_{L^{2}(\Omega)}^{\frac{2}{6\alpha-1}}\|  n_{\varepsilon}^{\alpha }\|_{L^{\frac{1}{\alpha }}(\Omega)}^{\frac{2}{\alpha }-\frac{2}{6\alpha-1}}+\|  n_{\varepsilon}^{\alpha }\|_{L^{\frac{1}{\alpha }}(\Omega)}^{\frac{2}{\alpha }})}\\
\leq&\disp{\delta_1\|\nabla  n_{\varepsilon}^{\alpha }\|_{L^{2}(\Omega)}^{2}+C_5~~\mbox{for all}~~ t\in(0, T_{max,\varepsilon})}\\
\end{array}
\label{ddddfgcz2.5ghju4cddfff8cfg924gjjkkkhyuji}
\end{equation}
with some positive constants  $C_4$ and $C_5$ independent of $\varepsilon$.
This together with \dref{hhxxcdfvvjjcz2.5ddsdfrfggggdrr} contributes to
\begin{equation}
\begin{array}{rl}
\disp\frac{1}{{2}}\disp\frac{d}{dt}\|{c_{\varepsilon}}\|^{{{2}}}_{L^{{2}}(\Omega)}+\frac{1}{2}
\int_{\Omega} |\nabla c_{\varepsilon}|^2+ \int_{\Omega} | c_{\varepsilon}|^2\leq&\disp{\frac{C_1}{2}\delta_1\|\nabla  n_{\varepsilon}^{\alpha }\|_{L^{2}(\Omega)}^{2}+C_6
~~\mbox{for all}~~ t\in(0, T_{max,\varepsilon})}\\
\end{array}
\label{hhxxcdfvvjjcz2ffghhkl.5ddsdfrfggggdrr}
\end{equation}
and some positive constant $C_6$.
Taking  an evident linear combination of the inequalities provided by \dref{55hhjjcffghhhjkkllzddfghh2.5} and  \dref{hhxxcdfvvjjcz2ffghhkl.5ddsdfrfggggdrr}, one can
obtain that
\begin{equation}
\begin{array}{rl}
\disp & sign(2\alpha-1)\disp\frac{1}{{2\alpha }}\disp\frac{d}{dt}\|{ n_{\varepsilon} }\|^{{{2\alpha }}}_{L^{{2\alpha }}(\Omega)}+
|2\alpha-1|C_S^2\disp\frac{d}{dt}\|{c_{\varepsilon}}\|^{{{2}}}_{L^{{2}}(\Omega)}\\
\disp &+\disp\frac{|2\alpha-1|}{2}C_S^2
\int_{\Omega} |\nabla c_{\varepsilon}|^2+ 2|2\alpha-1|C_S^2\int_{\Omega} | c_{\varepsilon}|^2\\
\disp &+(sign(2\alpha-1)\disp\frac{2\alpha-1}{2}-C_1\delta_1\alpha^2|2\alpha-1|C_S^2)\disp\int_{\Omega}  n_{\varepsilon}^{2\alpha-2} |\nabla n_{\varepsilon}|^2\\
\leq&\disp{C_7
~~\mbox{for all}~~ t\in(0, T_{max,\varepsilon})}\\
\end{array}
\label{hhxxcdfvvjjczddfgghjj2ffghhkl.5ddsdfrfggggdrr}
\end{equation}
and some positive constant $C_7.$
Since $sign(2\alpha-1)\disp\frac{2\alpha-1}{2}=\disp\frac{|2\alpha-1|}{2},$ we may choose $\delta=\frac{1}{4}\frac{1}{C_1\alpha^2C_S^2}$ in \dref{hhxxcdfvvjjczddfgghjj2ffghhkl.5ddsdfrfggggdrr} then implies that
\begin{equation}
\begin{array}{rl}
\disp & sign(2\alpha-1)\disp\frac{1}{{2\alpha }}\disp\frac{d}{dt}\|{ n_{\varepsilon} }\|^{{{2\alpha }}}_{L^{{2\alpha }}(\Omega)}+
|2\alpha-1|C_S^2\disp\frac{d}{dt}\|{c_{\varepsilon}}\|^{{{2}}}_{L^{{2}}(\Omega)}\\
\disp &+\disp\frac{|2\alpha-1|}{2}C_S^2
\int_{\Omega} |\nabla c_{\varepsilon}|^2+ 2|2\alpha-1|C_S^2\int_{\Omega} | c_{\varepsilon}|^2\\
&+\disp\frac{|2\alpha-1|}{4}\disp\int_{\Omega}  n_{\varepsilon}^{2\alpha-2} |\nabla n_{\varepsilon}|^2\\
\leq&\disp{C_7
~~\mbox{for all}~~ t\in(0, T_{max,\varepsilon}).}\\
\end{array}
\label{hhxxcdfvvjjczddfgghjjddfgghhh2ffghhkl.5ddsdfrfggggdrr}
\end{equation}
If $2\alpha>1$, then $sign(2\alpha-1)=1>0,$
thus,
 integrating \dref{hhxxcdfvvjjczddfgghjjddfgghhh2ffghhkl.5ddsdfrfggggdrr} in time, we can obtain
 \begin{equation}
\begin{array}{rl}
&\disp{\int_{\Omega} n_{\varepsilon}^{2\alpha }+\int_{\Omega}   c_{\varepsilon}^2\leq C_7~~~\mbox{for all}~~ t\in (0, T_{max,\varepsilon})}\\
\end{array}
\label{czfvgb2.5ghhjuyuddfghhccvhhjkkviihjj}
\end{equation}
and
\begin{equation}
\begin{array}{rl}
&\disp{\int_{0}^T\int_{\Omega} \left[  n_{\varepsilon}^{2\alpha-2} |\nabla {n_{\varepsilon}}|^2+ |\nabla {c_{\varepsilon}}|^2\right]\leq C_7~~\mbox{for all}~~ T\in(0, T_{max,\varepsilon})}\\
\end{array}
\label{bnmbncz2.5ghhjuddfghhddfggyhjkklluivvbnnihjj}
\end{equation}
and some positive constant $C_7.$
%
%
%
 While if $2\alpha<1$, then $sign(2\alpha-1)=-1<0$, hence, in view of \dref{ddfgczhhhh2.5ghju48cfg924ghyuji},
 integrating \dref{hhxxcdfvvjjczddfgghjjddfgghhh2ffghhkl.5ddsdfrfggggdrr} in time and employing  the H\"{o}lder inequality, we conclude that there exists a positive constant $C_8$ such that
\begin{equation}
\begin{array}{rl}
&\disp{\int_{\Omega} n_{\varepsilon}^{2\alpha }+\int_{\Omega}   c_{\varepsilon}^2\leq C_8~~~\mbox{for all}~~ t\in (0, T_{max,\varepsilon})}\\
\end{array}
\label{czfvgb2.5ghhjuyuddfghhccvhhjkkviihjj}
\end{equation}
and
\begin{equation}
\begin{array}{rl}
&\disp{\int_{0}^T\int_{\Omega} \left[  n_{\varepsilon}^{2\alpha-2} |\nabla {n_{\varepsilon}}|^2+ |\nabla {c_{\varepsilon}}|^2\right]\leq C_8~~\mbox{for all}~~ T\in(0, T_{max,\varepsilon}).}\\
\end{array}
\label{bnmbncz2.5ghhjuddfghhddfggyhjkklluivvbnnihjj}
\end{equation}
Case£º $2\alpha=1:$
Using the first equation of \dref{1.1fghyuisda} and \dref{1.ffggvbbnxxccvvn1}, from integration by parts and applying \dref{x1.73142vghf48gg} and using \dref{1.ffggvddfghhghhhhbbnxxccvvn1}, we obtain
\begin{equation}
 \begin{array}{rl}
&\disp\frac{d}{dt}\disp\int_{\Omega} n_{\varepsilon} \ln  n_{\varepsilon}\\
 =&\disp{\int_{\Omega}n_{\varepsilon t} \ln  n_{\varepsilon}+
\int_{\Omega}n_{\varepsilon t}}\\
=&\disp{\int_{\Omega}\Delta  n_{\varepsilon}  \ln  n_{\varepsilon}-
\int_{\Omega}\ln  n_{\varepsilon}\nabla\cdot(n_{\varepsilon}F'_{\varepsilon}(n_{\varepsilon})S_\varepsilon(x, n_{\varepsilon}, c_{\varepsilon})
\cdot\nabla c_{\varepsilon})}\\
\leq&\disp{-\int_{\Omega} \frac{|\nabla n_{\varepsilon}|^2}{n_{\varepsilon}}+
\int_{\Omega} C_S(1 + n_{\varepsilon})^{-\alpha}\frac{n_{\varepsilon}}{ n_{\varepsilon} }|\nabla n_{\varepsilon}||\nabla c_{\varepsilon}|~~~\mbox{for all}~~ t\in (0, T_{max,\varepsilon}),}\\
\end{array}\label{vgccvsckkcvvsbbbbhsvvbbsddaqwswddaassffssff3.10deerfgghhjuuloollgghhhyhh}
\end{equation}
which combined with the Young inequality and $2\alpha=1$ implies that
\begin{equation}
 \begin{array}{rl}
\disp\frac{d}{dt}\disp\int_{\Omega} n_{\varepsilon} \ln  n_{\varepsilon} +\disp\frac{1}{2}
\disp\int_{\Omega} \frac{|\nabla n_{\varepsilon}|^2}{n_{\varepsilon}} \leq&\disp{\disp\frac{1}{2}C_S^2
\disp\int_{\Omega}|\nabla c_{\varepsilon}|^2~~~\mbox{for all}~~ t\in (0, T_{max,\varepsilon}).}\\
\end{array}\label{vgccvsckkcvvsbbbbhsvvbbsddaqwswddaassffssff3.10defgghjjerfgghhjuuloollgghhhyhh}
\end{equation}
On the  other hand, due to  $2\alpha=1$  yields to $\alpha>\frac{1}{3}$, employing almost exactly the same arguments as in the proof of \dref{hhxxcdfvvjjcz2.5}--\dref{hhxxcdfvvjjczddfgghjj2ffghhkl.5ddsdfrfggggdrr} (the minor necessary changes
are left as an easy exercise to the reader),  we conclude the estimate
\begin{equation}
\begin{array}{rl}
&\disp{\int_{\Omega} n_{\varepsilon} \ln  n_{\varepsilon} +\int_{\Omega}   c_{\varepsilon}^2\leq C_9~~~\mbox{for all}~~ t\in (0, T_{max,\varepsilon})}\\
\end{array}
\label{czfvgb2.5ghhjuyuddfghhccvhjkkkkhhjkkklllkviihjj}
\end{equation}
and
\begin{equation}
\begin{array}{rl}
&\disp{\int_{0}^T\int_{\Omega} \left[ \frac{|\nabla n_{\varepsilon}|^2}{n_{\varepsilon}}+ |\nabla {c_{\varepsilon}}|^2\right]\leq C_9~~\mbox{for all}~~ T\in(0, T_{max,\varepsilon}).}\\
\end{array}
\label{bnmbncz2.5ghhjuddfghhffghhddfggyhjkkllujjkkivvbnnihjj}
\end{equation}

Now, multiplying the
third equation of \dref{1.1fghyuisda} by $u_\varepsilon$, integrating by parts and using $\nabla\cdot u_{\varepsilon}=0$
\begin{equation}
\frac{1}{2}\frac{d}{dt}\int_{\Omega}{|u_{\varepsilon}|^2}+\int_{\Omega}{|\nabla u_{\varepsilon}|^2}= \int_{\Omega}n_{\varepsilon}u_{\varepsilon}\cdot\nabla \phi~~\mbox{for all}~~ t\in(0, T_{max,\varepsilon}).
\label{ddddfgcz2.5ghju48cfg924ghyuji}
\end{equation}
Here we use the H\"{o}lder inequality, the Young inequality, \dref{dd1.1fghyuisdakkkllljjjkk} and the continuity of the embedding $W^{1,2}(\Omega)\hookrightarrow L^6(\Omega)$ and  to
find $C_{10} $ and $C_{11}> 0$ such that
\begin{equation}
\begin{array}{rl}
\disp\int_{\Omega}n_{\varepsilon}u_{\varepsilon}\cdot\nabla \phi\leq&\disp{\|\nabla \phi\|_{L^\infty(\Omega)}\| n_{\varepsilon} \|_{L^{\frac{6}{5}}(\Omega)}\| u_{\varepsilon}\|_{L^{6}(\Omega)}}\\
\leq&\disp{C_{10}\|\nabla \phi\|_{L^\infty(\Omega)}\| n_{\varepsilon} \|_{L^{\frac{6}{5}}(\Omega)}\|\nabla u_{\varepsilon}\|_{L^{2}(\Omega)}}\\
\leq&\disp{\frac{1}{2}\|\nabla u_{\varepsilon}\|_{L^{2}(\Omega)}^2+C_{11}\| n_{\varepsilon} \|_{L^{\frac{6}{5}}(\Omega)}^2~~\mbox{for all}~~ t\in(0, T_{max,\varepsilon}).}\\
\end{array}
\label{ddddfgcz2.5ghju48cfg924ghyuji}
\end{equation}
Next, observing that \dref{ddfgczhhhh2.5ghju48cfg924ghyuji}, in view of $\alpha>\frac{1}{3}$,  by \dref{ddddfgcz2.5ghju4cddfff8cfg924gjjkkkhyuji}
 and using the
Young inequality and the Gagliardo--Nirenberg inequality
  yields
\begin{equation}
\begin{array}{rl}
\disp{\int_{\Omega}n_{\varepsilon}u_{\varepsilon}\cdot\nabla \phi}\leq &\disp{\frac{1}{2}\|\nabla u_{\varepsilon}\|_{L^{2}(\Omega)}^2+C_8\|\nabla n_{\varepsilon}^{\alpha }\|_{L^{2}(\Omega)}^{\frac{2}{6\alpha-1}}\|  n_{\varepsilon}^{\alpha }\|_{L^{\frac{1}{\alpha }}(\Omega)}^{\frac{2}{\alpha }-\frac{2}{6\alpha-1}}}\\
\leq &\disp{\frac{1}{2}\|\nabla u_{\varepsilon}\|_{L^{2}(\Omega)}^2+\|\nabla  n_{\varepsilon}^{\alpha }\|_{L^{2}(\Omega)}^{2}+C_{12}~~\mbox{for all}~~ t\in(0, T_{max,\varepsilon})}\\
\end{array}
\label{ddddfgcxccdd2.5ghju4cvvbbttthdfff8cfg924ghyuji}
\end{equation}
and some positive constant $C_{12}.$
Now, inserting  \dref{ddddfgcz2.5ghju48cfg924ghyuji} and \dref{ddddfgcxccdd2.5ghju4cvvbbttthdfff8cfg924ghyuji} into \dref{ddddfgcz2.5ghju48cfg924ghyuji} and using \dref{bnmbncz2.5ghhjuddfghhddfggyhjkklluivvbnnihjj} and \dref{bnmbncz2.5ghhjuddfghhffghhddfggyhjkkllujjkkivvbnnihjj}, one have
\begin{equation}
\begin{array}{rl}
&\disp{\int_{\Omega}   |u_{\varepsilon}|^2\leq C_{13}~~~\mbox{for all}~~ t\in (0, T_{max,\varepsilon})}\\
\end{array}
\label{czfvgb2.5ghhjuyghjjjuddfghhccvjkkklllhhjkkviihjj}
\end{equation}
and
\begin{equation}
\begin{array}{rl}
&\disp{\int_{0}^T\int_{\Omega}  |\nabla {u_{\varepsilon}}|^2\leq C_{13}~~\mbox{for all}~~ T\in(0, T_{max,\varepsilon})}\\
\end{array}
\label{bnmbncz2.5ghhjuddfghhffghhhjjkklhddfggyhjkklluivvbnnihjj}
\end{equation}
and some positive constant $C_{14}.$
Finally, collecting \dref{czfvgb2.5ghhjuyuddfghhccvhhjkkviihjj}--\dref{bnmbncz2.5ghhjuddfghhddfggyhjkklluivvbnnihjj}, \dref{czfvgb2.5ghhjuyuddfghhccvhjkkkkhhjkkklllkviihjj}--\dref{bnmbncz2.5ghhjuddfghhffghhddfggyhjkkllujjkkivvbnnihjj} and
 \dref{czfvgb2.5ghhjuyghjjjuddfghhccvjkkklllhhjkkviihjj}--\dref{bnmbncz2.5ghhjuddfghhffghhhjjkklhddfggyhjkklluivvbnnihjj}, we can get \dref{czfvgb2.5ghhjuyuccvviihjj}--\dref{bnmbncz2.5ghhjuyuivvbnnihjj}.
\end{proof}

With the help of Lemma \ref{lemmaghjffggssddgghhmk4563025xxhjklojjkkk}, in light of the Gagliardo--Nirenberg inequality and an application of well-known arguments
from parabolic regularity theory, we can derive the following Lemma:
\begin{lemma}\label{lemmddaghjsffggggsddgghhmk4563025xxhjklojjkkk}
Let $\alpha>\frac{1}{3}$. Then for each $T\in(0, T_{max,\varepsilon})$,
 there exists $C>0$ independent of $\varepsilon$ such that the solution of \dref{1.1fghyuisda} satisfies
\begin{equation}
\begin{array}{rl}
&\disp{\int_{0}^T\int_{\Omega}\left[|\nabla n_{\varepsilon}|^{\frac{3\alpha+1}{2}}+ n_{\varepsilon}^{\frac{6\alpha+2}{3}}\right]\leq C(T+1)~~\mbox{if}~~~\frac{1}{3}<\alpha\leq\frac{1}{2},}\\
\end{array}
\label{bnmbncz2.ffghh5ghhjuyuivvbnnihjj}
\end{equation}
\begin{equation}
\begin{array}{rl}
&\disp{\int_{0}^T\int_{\Omega}\left[|\nabla n_{\varepsilon}|^{\frac{10\alpha}{3+2\alpha}}+ n_{\varepsilon}^{\frac{10\alpha}{3}}\right]\leq C(T+1)~~\mbox{if}~~~\frac{1}{2}<\alpha<1}\\
\end{array}
\label{11bnmbncz2.ffghh5ghhjuyuivvbnnihjj}
\end{equation}
as well as
\begin{equation}
\begin{array}{rl}
&\disp{\int_{0}^T\int_{\Omega} \left[|\nabla {n_{\varepsilon}}|^2+ n_{\varepsilon}^{\frac{10}{3}}\right]\leq C(T+1)~~\mbox{if}~~~\alpha\geq1}\\
\end{array}
\label{bnmbncz2.fffgghhfghh5ghhjuyuivvbnnihjj}
\end{equation}

and
\begin{equation}
\begin{array}{rl}
&\disp{\int_{0}^T\int_{\Omega} \left[c_{\varepsilon}^{\frac{10}{3}}+ |u_{\varepsilon}|^{\frac{10}{3}}\right]\leq C(T+1).}\\
\end{array}
\label{bnmbncz2.ffgddffffhh5ghhjuyuivvbnnihjj}
\end{equation}
\end{lemma}
\begin{proof}
Case $\frac{1}{3}<\alpha\leq\frac{1}{2}:$  Due to \dref{ddfgczhhhh2.5ghju48cfg924ghyuji}, \dref{czfvgb2.5ghhjuyuccvviihjj}
and \dref{bnmbncz2.5ghhjuyuivvbnnihjj}, in light of the Gagliardo--Nirenberg inequality, for some $C_1$ and $C_2> 0$ which are independent of $\varepsilon$, one may verify that
\begin{equation}
\begin{array}{rl}
&\disp\int_{0}^T\disp\int_{\Omega} n_{\varepsilon}^{\frac{6\alpha+2}{3}} \\
=&\disp{\int_{0}^T\| { n_{\varepsilon}^{\alpha }}\|^{{\frac{6\alpha+2}{3\alpha }}}_{L^{\frac{6\alpha+2}{3\alpha }}(\Omega)}}\\
\leq&\disp{C_{1}\int_{0}^T\left(\| \nabla{ n_{\varepsilon}^{\alpha }}\|^{2}_{L^{2}(\Omega)}\|{ n_{\varepsilon}^{\alpha }}\|^{{\frac{2}{3\alpha }}}_{L^{\frac{1}{\alpha }}(\Omega)}+
\|{ n_{\varepsilon}^{\alpha }}\|^{{\frac{6\alpha+2}{3\alpha }}}_{L^{\frac{1}{\alpha }}(\Omega)}\right)}\\
\leq&\disp{C_{2}(T+1)~~\mbox{for all}~~ T > 0.}\\
\end{array}
\label{ddffbnmbnddfgcz2ddfvgbhh.htt678ddfghhhyuiihjj}
\end{equation}
Therefore, employing  the
H\"{o}lder inequality (with two exponents $\frac{4}{3\alpha+1}$ and $\frac{4}{3-3\alpha}$), we conclude  that there exists a positive
 constant $C_3$ 
such that
\begin{equation}
\begin{array}{rl}
\disp\int_{0}^T\disp\int_{\Omega}|\nabla n_{\varepsilon}|^{\frac{3\alpha+1}{2}}
\leq&\disp{\left[\int_{0}^T\disp\int_{\Omega} n_{\varepsilon}^{2\alpha-2}|\nabla n_{\varepsilon}|^2\right]^{\frac{3\alpha+1}{4}}
\left[\int_{0}^T\disp\int_{\Omega} n_{\varepsilon}^{\frac{6\alpha+2}{3}}\right]^{\frac{3-3\alpha}{4}} }\\
\leq&\disp{C_{3}(T+1)~~\mbox{for all}~~ T > 0.}\\
\end{array}
\label{5555ddffbnmbncz2ddfvgffgtyybhh.htt678ghhjjjddfghhhyuiihjj}
\end{equation}

Case $\frac{1}{2}<\alpha<1:$  Again by  \dref{ddfgczhhhh2.5ghju48cfg924ghyuji}, \dref{czfvgb2.5ghhjuyuccvviihjj}
and \dref{bnmbncz2.5ghhjuyuivvbnnihjj} and  the Gagliardo--Nirenberg inequality the
H\"{o}lder inequality (with two exponents $\frac{3+2\alpha}{5\alpha}$ and $\frac{3+2\alpha}{3-3\alpha}$), we derive that there exist positive constants
$C_4,C_5$ and $C_6$ such that
\begin{equation}
\begin{array}{rl}
&\disp\int_{0}^T\disp\int_{\Omega} n_{\varepsilon}^{\frac{10\alpha}{3}}\\ =&\disp{\int_{0}^T\| { n_{\varepsilon}^{\alpha }}\|^{{\frac{10}{3 }}}_{L^{\frac{10}{3 }}(\Omega)}}\\
\leq&\disp{C_{4}\int_{0}^T\left(\| \nabla{ n_{\varepsilon}^{\alpha }}\|^{2}_{L^{2}(\Omega)}\|{ n_{\varepsilon}^{\alpha }}
\|^{{\frac{4 }{3}}}_{L^{2}(\Omega)}+
\|{ n_{\varepsilon}^{\alpha }}\|^{{\frac{10\alpha}{3}}}_{L^{2}(\Omega)}\right)}\\
\leq&\disp{C_{5}(T+1)~~\mbox{for all}~~ T > 0}\\
\end{array}
\label{ddffbnmbnddfgcffghhz2ddfvgbhh.htt678ddfghhhyuiihjj}
\end{equation}
and
\begin{equation}
\begin{array}{rl}
\disp\int_{0}^T\disp\int_{\Omega}|\nabla n_{\varepsilon}|^{\frac{10\alpha }{3+2\alpha}}
\leq&\disp{\left[\int_{0}^T\disp\int_{\Omega} n_{\varepsilon}^{2\alpha-2}|\nabla n_{\varepsilon}|^2\right]^{\frac{5\alpha}{3+2\alpha}}
\left[\int_{0}^T\disp\int_{\Omega} n_{\varepsilon}^{\frac{10\alpha}{3}}\right]^{\frac{3-3\alpha}{3+2\alpha}} }\\
\leq&\disp{C_{6}(T+1)~~\mbox{for all}~~ T > 0.}\\
\end{array}
\label{5555ddffbnmbncz2ddfvffghjjgffgtyybhh.htt678ghhjjjddfghhhyuiihjj}
\end{equation}
%
Case $\alpha\geq1:$
Multiply the first equation in $\dref{1.1fghyuisda}$ by $ n_{\varepsilon}$, in view of  \dref{1.ffggvbbnxxccvvn1}
 and  using $\nabla\cdot u_\varepsilon=0$, we derive
\begin{equation}
\begin{array}{rl}
&\disp{\frac{1}{{2}}\frac{d}{dt}\|{ n_{\varepsilon} }\|^{{{{2}}}}_{L^{{{2}}}(\Omega)}+
\int_{\Omega}  |\nabla n_{\varepsilon}|^2}\\
=&\disp{-
\int_{\Omega}  n_{\varepsilon}\nabla\cdot(n_{\varepsilon}F'_{\varepsilon}(n_{\varepsilon})S_\varepsilon(x, n_{\varepsilon}, c_{\varepsilon})\cdot\nabla c_{\varepsilon})}\\
\leq&\disp{
\int_{\Omega}  n_{\varepsilon}F'_{\varepsilon}(n_{\varepsilon})|S_\varepsilon(x, n_{\varepsilon}, c_{\varepsilon})||\nabla n_{\varepsilon}||\nabla c_{\varepsilon}|~~\mbox{for all}~~ t\in(0, T_{max,\varepsilon}).}
\end{array}
\label{55hhjjcffghhhjkklddfgggffgglffghhhz2.5}
\end{equation}
Recalling \dref{x1.73142vghf48gg} and \dref{1.ffggvbbnxxccvvn1} and using $\alpha\geq1$, by Young inequality, we derive that
\begin{equation}
\begin{array}{rl}
&\disp\int_{\Omega} n_{\varepsilon}F'_{\varepsilon}(n_{\varepsilon})|S_\varepsilon(x, n_{\varepsilon}, c_{\varepsilon})||\nabla n_{\varepsilon}||\nabla c_{\varepsilon}|\\
\leq&\disp{C_S\int_{\Omega} |\nabla n_{\varepsilon}||\nabla c_{\varepsilon}|}\\
\leq&\disp{\frac{1}{2}\int_{\Omega} |\nabla n_{\varepsilon}|^2+\frac{C_S^2}{2}\int_{\Omega} |\nabla c_{\varepsilon}|^2~~\mbox{for all}~~ t\in(0, T_{max,\varepsilon}).}
\end{array}
\label{55hhjjcffghhhjkkllfffghggggghgghjjjhfghhhz2.5}
\end{equation}
Here we have use the fact that
$$ n_{\varepsilon}F'_{\varepsilon}(n_{\varepsilon})|S_\varepsilon(x, n_{\varepsilon}, c_{\varepsilon})|\leq C_S n_{\varepsilon}(1 + n_{\varepsilon})^{-1}\leq C_S.$$
Therefore, collecting  \dref{55hhjjcffghhhjkklddfgggffgglffghhhz2.5} and \dref{55hhjjcffghhhjkkllfffghggggghgghjjjhfghhhz2.5} and using \dref{bnmbncz2.5ghhjuyuivvbnnihjj}, we conclude that
 \begin{equation}
\begin{array}{rl}
&\disp{\int_{\Omega} n_{\varepsilon}^{2}\leq C_{7}~~~\mbox{for all}~~ t\in (0, T_{max,\varepsilon})}\\
\end{array}
\label{czfvgb2.5ghhjuyucfkllllfhhhhhhhjjggcvviihjj}
\end{equation}
and
\begin{equation}
\begin{array}{rl}
&\disp{\int_{0}^{ T}\int_{\Omega}  |\nabla {n_{\varepsilon}}|^2\leq C_{7}(T+1).}\\
\end{array}
\label{bnmbncz2.5ghhjugghjllllljdfghhjjyuivvbnklllnihjj}
\end{equation}
Hence, due to \dref{czfvgb2.5ghhjuyucfkllllfhhhhhhhjjggcvviihjj}--\dref{bnmbncz2.5ghhjugghjllllljdfghhjjyuivvbnklllnihjj}, \dref{czfvgb2.5ghhjuyuccvviihjj}
and \dref{bnmbncz2.5ghhjuyuivvbnnihjj}, in light of the Gagliardo--Nirenberg inequality, we derive that there exist positive constants $C_{8},C_{9},C_{10},C_{11},C_{12}$ and $C_{14}$  such that
\begin{equation}
\begin{array}{rl}
\disp\int_{0}^T\disp\int_{\Omega}  n_{\varepsilon}^{\frac{10}{3}} \leq&\disp{C_{8}\int_{0}^T\left(\| \nabla{ n_{\varepsilon}}\|^{2}_{L^{2}(\Omega)}\|{ n_{\varepsilon} }\|^{{\frac{4}{3}}}_{L^{2}(\Omega)}+
\|{  n_{\varepsilon} }\|^{{\frac{10}{3}}}_{L^{2}(\Omega)}\right)}\\
\leq&\disp{C_{9}(T+1)~~\mbox{for all}~~ T > 0}\\
\end{array}
\label{ddffbnmbnddfgffddfghhggjjkkuuiicz2ddfvgbhh.htt678ddfghhhyuiihjj}
\end{equation}
as well as
\begin{equation}
\begin{array}{rl}
\disp\int_{0}^T\disp\int_{\Omega} c_{\varepsilon}^{\frac{10}{3}} \leq&\disp{C_{10}\int_{0}^T\left(\| \nabla{ c_{\varepsilon}}\|^{2}_{L^{2}(\Omega)}\|{ c_{\varepsilon}}\|^{{\frac{4}{3}}}_{L^{2}(\Omega)}+
\|{ c_{\varepsilon}}\|^{{\frac{10}{3}}}_{L^{2}(\Omega)}\right)}\\
\leq&\disp{C_{11}(T+1)~~\mbox{for all}~~ T > 0}\\
\end{array}
\label{ddffbnmbnddfgffggjjkkuuiicz2ddfvgbhh.hkklllhhjkktt678ddfghhhyuiihjj}
\end{equation}
and
\begin{equation}
\begin{array}{rl}
\disp\int_{0}^T\disp\int_{\Omega} |u_{\varepsilon}|^{\frac{10}{3}} \leq&\disp{C_{12}\int_{0}^T\left(\| \nabla{ u_{\varepsilon}}\|^{2}_{L^{2}(\Omega)}\|{ u_{\varepsilon}}\|^{{\frac{4}{3}}}_{L^{2}(\Omega)}+
\|{ u_{\varepsilon}}\|^{{\frac{10}{3}}}_{L^{2}(\Omega)}\right)}\\
\leq&\disp{C_{14}(T+1)~~\mbox{for all}~~ T > 0.}\\
\end{array}
\label{ddffbnmbnddfgffggjjkkuuiicz2dvgbhh.t8ddhhhyuiihjj}
\end{equation}
Finally, combined with \dref{ddffbnmbnddfgcz2ddfvgbhh.htt678ddfghhhyuiihjj}--\dref{5555ddffbnmbncz2ddfvffghjjgffgtyybhh.htt678ghhjjjddfghhhyuiihjj} and \dref{bnmbncz2.5ghhjugghjllllljdfghhjjyuivvbnklllnihjj}--\dref{ddffbnmbnddfgffggjjkkuuiicz2dvgbhh.t8ddhhhyuiihjj}, we can get the results.
\end{proof}

\section{The global solvability of regularized problem \dref{1.1fghyuisda}}
The main task of this section is to prove the global solvability of regularized problem \dref{1.1fghyuisda}. To this end, we firstly, need to establish some $\varepsilon$-dependent estimates for $n_{\varepsilon},c_{\varepsilon}$ and $u_{\varepsilon}$.

\subsection{A priori estimates for the regularized problem \dref{1.1fghyuisda} which depends on $\varepsilon$}

In this subsection, on the basis of Lemma \ref{lemmaghjffggssddgghhmk4563025xxhjklojjkkk}, we thereby obtain some regularity properties for $n_\varepsilon,c_\varepsilon$ and $u_\varepsilon$
in the following form.
%
%
\begin{lemma}\label{lemmaghjssddgghhmk4563025xxhjklojjkkk}
Let $\alpha>\frac{1}{3}$.
Then there exists $C:=C(\varepsilon)>0$ depending on $\varepsilon$ such that the solution of \dref{1.1fghyuisda} satisfies
\begin{equation}
\begin{array}{rl}
&\disp{\int_{\Omega}n^{2\alpha+2}_{\varepsilon}+\int_{\Omega}  | \nabla{u_{\varepsilon}}|^2\leq C~~~\mbox{for all}~~ t\in (0, T_{max,\varepsilon}).}\\
\end{array}
\label{czfvgb2.5ghhjuyucffggcvviihjj}
\end{equation}
In addition,
for each $T\in(0, T_{max,\varepsilon})$, one can find a constant $C > 0$ depends on $\varepsilon$ such that
\begin{equation}
\begin{array}{rl}
&\disp{\int_{0}^{ T}\int_{\Omega} \left[ n_{\varepsilon}^{2\alpha } |\nabla {n_{\varepsilon}}|^2+ |\Delta {u_{\varepsilon}}|^2\right]\leq C.}\\
\end{array}
\label{bnmbncz2.5ghhjugghjjjjyuivvbnnihjj}
\end{equation}
\end{lemma}
\begin{proof}
Multiply the first equation in $\dref{1.1fghyuisda}$ by $ n_{\varepsilon}^{1+2\alpha}$, in view of  \dref{1.ffggvbbnxxccvvn1}
 and  using $\nabla\cdot u_\varepsilon=0$, we derive
\begin{equation}
\begin{array}{rl}
&\disp{\frac{1}{{2\alpha+2}}\frac{d}{dt}\|{ n_{\varepsilon} }\|^{{{2\alpha+2}}}_{L^{{2\alpha+2}}(\Omega)}+
(1+2\alpha)\int_{\Omega}  n_{\varepsilon}^{2\alpha }|\nabla n_{\varepsilon}|^2}\\
=&\disp{-
\int_{\Omega}  n_{\varepsilon}^{1+2\alpha}\nabla\cdot(n_{\varepsilon}F'_{\varepsilon}(n_{\varepsilon})S_\varepsilon(x, n_{\varepsilon}, c_{\varepsilon})\cdot\nabla c_{\varepsilon})}\\
\leq&\disp{(1+2\alpha)
\int_{\Omega}  n_{\varepsilon}^{2\alpha }n_{\varepsilon}F'_{\varepsilon}(n_{\varepsilon})|S_\varepsilon(x, n_{\varepsilon}, c_{\varepsilon})||\nabla n_{\varepsilon}||\nabla c_{\varepsilon}|~~\mbox{for all}~~ t\in(0, T_{max,\varepsilon}).}
\end{array}
\label{55hhjjcffghhhjkklddffgglffghhhz2.5}
\end{equation}
Recalling \dref{x1.73142vghf48gg} and \dref{1.ffggvbbnxxccvvn1}, by Young inequality, one can see that
\begin{equation}
\begin{array}{rl}
&\disp(1+2\alpha)\int_{\Omega}  n_{\varepsilon}^{2\alpha }n_{\varepsilon}F'_{\varepsilon}(n_{\varepsilon})|S_\varepsilon(x, n_{\varepsilon}, c_{\varepsilon})||\nabla n_{\varepsilon}||\nabla c_{\varepsilon}|\\
\leq&\disp{\frac{1}{\varepsilon}C_S(1+2\alpha)\int_{\Omega} n_{\varepsilon}^{2\alpha }(1 + n_{\varepsilon})^{-\alpha} |\nabla n_{\varepsilon}||\nabla c_{\varepsilon}|}\\
\leq&\disp{\frac{1}{\varepsilon}C_S(1+2\alpha)\int_{\Omega} n_{\varepsilon}^{\alpha }|\nabla n_{\varepsilon}||\nabla c_{\varepsilon}|}\\
\leq&\disp{\frac{(1+2\alpha)}{2}\int_{\Omega}  n_{\varepsilon}^{2\alpha }|\nabla n_{\varepsilon}|^2+C_1\int_{\Omega} |\nabla c_{\varepsilon}|^2~~\mbox{for all}~~ t\in(0, T_{max,\varepsilon}),}
\end{array}
\label{55hhjjcffghhhjkkllfffghggggghhfghhhz2.5}
\end{equation}
where $C_1$ is a positive constant, as all subsequently appearing constants $C_2, C_3, \ldots$ possibly depend on
 $\varepsilon$.
  Here we have used the fact that
$ F'_{\varepsilon}(n_{\varepsilon})\leq\disp\frac{1}{\varepsilon n_{\varepsilon}}$.
  Inserting \dref{55hhjjcffghhhjkkllfffghggggghhfghhhz2.5} into \dref{55hhjjcffghhhjkklddffgglffghhhz2.5} and using \dref{bnmbncz2.5ghhjuyuivvbnnihjj}, we derive that
 \begin{equation}
\begin{array}{rl}
&\disp{\int_{\Omega}n^{2\alpha+2}\leq C_2~~~\mbox{for all}~~ t\in (0, T_{max,\varepsilon})}\\
\end{array}
\label{czfvgb2.5ghhjuyucffhhhhhhhjjggcvviihjj}
\end{equation}
and
\begin{equation}
\begin{array}{rl}
&\disp{\int_{0}^{ T}\int_{\Omega}  n_{\varepsilon}^{2\alpha } |\nabla {n_{\varepsilon}}|^2\leq C_2~~~\mbox{for all}~~~T<T_{max,\varepsilon}.}\\
\end{array}
\label{bnmbncz2.5ghhjugghjjdfghhjjyuivvbnklllnihjj}
\end{equation}

Now, due to  $D(1 + \varepsilon A)  :=W^{2,2}(\Omega) \cap W_{0,\sigma}^{1,2}(\Omega)\hookrightarrow L^\infty(\Omega),$ by \dref{czfvgb2.5ghhjuyuccvviihjj}, we derive that for some  $C_3> 0$ and $C_4 > 0$,
\begin{equation}
\|Y_{\varepsilon}u_{\varepsilon}\|_{L^\infty(\Omega)}=\|(I+\varepsilon A)^{-1}u_{\varepsilon}\|_{L^\infty(\Omega)}\leq C_3\|u_{\varepsilon}(\cdot,t)\|_{L^2(\Omega)}\leq C_4~~\mbox{for all}~~t\in(0,T_{max,\varepsilon}).
\label{ssdcfvgdhhjjdfghgghjjnnhhkklld911cz2.5ghju48}
\end{equation}
Next,  testing the projected Stokes equation $u_{\varepsilon t} +Au_{\varepsilon} =  \mathcal{P}[-\kappa (Y_{\varepsilon}u_{\varepsilon} \cdot \nabla)u_{\varepsilon}+n_{\varepsilon}\nabla \phi]$ by $Au_{\varepsilon}$, we derive
%
\begin{equation}
\begin{array}{rl}
&\disp{\frac{1}{{2}}\frac{d}{dt}\|A^{\frac{1}{2}}u_{\varepsilon}\|^{{{2}}}_{L^{{2}}(\Omega)}+
\int_{\Omega}|Au_{\varepsilon}|^2 }\\
=&\disp{ \int_{\Omega}Au_{\varepsilon}\mathcal{P}(-\kappa
(Y_{\varepsilon}u_{\varepsilon} \cdot \nabla)u_{\varepsilon})+ \int_{\Omega}\mathcal{P}(n_{\varepsilon}\nabla\phi) Au_{\varepsilon}}\\
\leq&\disp{ \frac{1}{2}\int_{\Omega}|Au_{\varepsilon}|^2+\kappa^2\int_{\Omega}
|(Y_{\varepsilon}u_{\varepsilon} \cdot \nabla)u_{\varepsilon}|^2+ \|\nabla\phi\|^2_{L^\infty(\Omega)}\int_{\Omega}n_{\varepsilon}^2~~\mbox{for all}~~t\in(0,T_{max,\varepsilon}).}\\
\end{array}
\label{ddfghgghjjnnhhkklld911cz2.5ghju48}
\end{equation}

On the other hand, in light of the Gagliardo--Nirenberg inequality, the Young inequality and \dref{ssdcfvgdhhjjdfghgghjjnnhhkklld911cz2.5ghju48}, there exists a positive constant $C_5$
such that
\begin{equation}
\begin{array}{rl}
\kappa^2\disp\int_{\Omega}
|(Y_{\varepsilon}u_{\varepsilon} \cdot \nabla)u_{\varepsilon}|^2\leq&\disp{ \kappa^2\|Y_{\varepsilon}u_{\varepsilon}\|^2_{L^\infty(\Omega)}\int_{\Omega}|\nabla u_{\varepsilon}|^2}\\
\leq&\disp{ \kappa^2\|Y_{\varepsilon}u_{\varepsilon}\|^2_{L^\infty(\Omega)}\int_{\Omega}|\nabla u_{\varepsilon}|^2}\\
\leq&\disp{ C_5\int_{\Omega}|\nabla u_{\varepsilon}|^2~~\mbox{for all}~~t\in(0,T_{max,\varepsilon}).}\\
\end{array}
\label{ssdcfvgddfghgghjjnnhhkklld911cz2.5ghju48}
\end{equation}
Here we have the well-known fact that $\|A(\cdot)\|_{L^{2}(\Omega)}$ defines a norm
equivalent to $\|\cdot\|_{W^{2,2}(\Omega)}$ on $D(A)$ (see Theorem 2.1.1 of \cite{Sohr}).
Now, recalling that
$\|A^{\frac{1}{2}}u_{\varepsilon}\|^{{{2}}}_{L^{{2}}(\Omega)} = \|\nabla u_{\varepsilon}\|^{{{2}}}_{L^{{2}}(\Omega)},$ therefore,
substituting \dref{ssdcfvgddfghgghjjnnhhkklld911cz2.5ghju48} into \dref{ddfghgghjjnnhhkklld911cz2.5ghju48} yields
\begin{equation}
\begin{array}{rl}
&\disp{\frac{1}{{2}}\frac{d}{dt}\|\nabla u_{\varepsilon}\|^{{{2}}}_{L^{{2}}(\Omega)}+
\int_{\Omega}|\Delta u_{\varepsilon}|^2 \leq C_6\int_{\Omega}|\nabla u_{\varepsilon}|^2+ \|\nabla\phi\|^2_{L^\infty(\Omega)}\int_{\Omega}
n_{\varepsilon}^2~~\mbox{for all}~~t\in(0,T_{max,\varepsilon}).}\\
\end{array}
\label{ddfgghhddfghgghjjnnhhkklld911cz2.5ghju48}
\end{equation}
In view of $\alpha>\frac{1}{3}$  yields to $2\alpha+2>\frac{8}{3}>2,$ thus,
collecting \dref{czfvgb2.5ghhjuyucffhhhhhhhjjggcvviihjj} and \dref{ddfgghhddfghgghjjnnhhkklld911cz2.5ghju48} and applying some basic calculation,  we can get
 the results.
\end{proof}

\begin{lemma}\label{xccffgghhlemma4563025xxhjkloghyui}
Under the assumptions of Theorem \ref{theorem3}, it holds that
there exists $C:=C(\varepsilon)> 0$ depends on $\varepsilon$ such that
\begin{equation}
\int_{\Omega}{|\nabla c_{\varepsilon}(\cdot,t)|^2}\leq C~~\mbox{for all}~~ t\in(0, T_{max,\varepsilon})
\label{ddxxxcvvddcvddffbbggddczv.5ghcfg924ghyuji}
\end{equation}
and
\begin{equation}
\int_0^{T}\int_{\Omega}{|\Delta c_{\varepsilon}|^2}\leq C~~\mbox{for all}~~ T\in(0, T_{max,\varepsilon}).
\label{gddffghhddddcz2.gghh5ghju48cfg924}
\end{equation}
\end{lemma}
\begin{proof}
Firstly,  testing the second equation in \dref{1.1fghyuisda} against $-\Delta c_{\varepsilon}$, employing the Young inequality and using \dref{ffggg1.ffggvddfghhghhhhbbhhjjjnxxccvvn1} yields
\begin{equation}
\begin{array}{rl}
\disp{\frac{1}{{2}}\frac{d}{dt} \|\nabla c_{\varepsilon}\|^{{{2}}}_{L^{{2}}(\Omega)}}= &\disp{\int_{\Omega}  -\Delta c_{\varepsilon}(\Delta c_{\varepsilon}-c_{\varepsilon}+F_{\varepsilon}(n_{\varepsilon})-u_{\varepsilon}\cdot\nabla  c_{\varepsilon})}
\\
=&\disp{-\int_{\Omega}  |\Delta c_{\varepsilon}|^2-\int_{\Omega} |\nabla c_{\varepsilon}|^{2}-\int_\Omega F_{\varepsilon}(n_{\varepsilon})\Delta c_{\varepsilon}-\int_\Omega (u_{\varepsilon}\cdot\nabla  c_{\varepsilon})\Delta c_{\varepsilon}}\\
\leq&\disp{-\frac{1}{4}\int_{\Omega}  |\Delta c_{\varepsilon}|^2-\int_{\Omega} |\nabla c_{\varepsilon}|^{2}+
\int_\Omega n_{\varepsilon}^2+\int_\Omega |u_{\varepsilon}\cdot\nabla  c_{\varepsilon}||\Delta c_{\varepsilon}|}\\
\end{array}
\label{cz2.5ghju48156}
\end{equation}
for all $t\in(0,T_{max,\varepsilon})$.
Next, one needs to estimate the last term on the right-hand side of \dref{cz2.5ghju48156}.
Indeed, in view of  the Sobolev embedding theorem ($W^{1,2}(\Omega)\hookrightarrow L^6(\Omega)$), then
applying  \dref{czfvgb2.5ghhjuyucffggcvviihjj} and \dref{czfvgb2.5ghhjuyuccvviihjj},  we derive from the H\"{o}lder inequality,  the Gagliardo--Nirenberg inequality and the Young inequality that   there exist positive constants
$C_1,C_2,C_3$ and $C_4$ such that
 \begin{equation}
\begin{array}{rl}
\disp{\int_\Omega |u_{\varepsilon}\cdot\nabla  c_{\varepsilon}||\Delta c_{\varepsilon}|}
\leq&\disp{\|u_{\varepsilon}\|_{L^{6}(\Omega)}\|\nabla c_{\varepsilon}\|_{L^{3}(\Omega)}}\|\Delta c_{\varepsilon}\|_{L^{2}(\Omega)}\\
\leq&\disp{C_1\|\nabla c_{\varepsilon}\|_{L^{3}(\Omega)}}\|\Delta c_{\varepsilon}\|_{L^{2}(\Omega)}\\
\leq&\disp{C_2(\|\Delta c_{\varepsilon}\|^{\frac{3}{4}}_{L^2(\Omega)}\| c_{\varepsilon}\|^{\frac{1}{4}}_{L^2(\Omega)}+\| c_{\varepsilon}\|^{2}_{L^2(\Omega)})\|\Delta c_{\varepsilon}\|_{L^{2}(\Omega)}}\\
\leq&\disp{C_3(\|\Delta c_{\varepsilon}\|^{\frac{7}{4}}_{L^2(\Omega)}+\|\Delta c_{\varepsilon}\|_{L^{2}(\Omega)})}\\
\leq&\disp{\frac{1}{4}\|\Delta c_{\varepsilon}\|^{2}_{L^2(\Omega)}+C_4~~\mbox{for all}~~ t\in(0, T_{max,\varepsilon}).}\\
\end{array}
\label{dd11cfvggcz2.5ghju48156}
\end{equation}
Inserting 
\dref{dd11cfvggcz2.5ghju48156}     into \dref{cz2.5ghju48156} and  using \dref{czfvgb2.5ghhjuyucffggcvviihjj}, one obtains
\dref{ddxxxcvvddcvddffbbggddczv.5ghcfg924ghyuji}
and \dref{gddffghhddddcz2.gghh5ghju48cfg924}.
This completes the proof of Lemma \ref{xccffgghhlemma4563025xxhjkloghyui}.
\end{proof}

\begin{lemma}\label{sss222lemma4444556645630223} Let $\alpha>\frac{1}{3}$.
Assume the hypothesis of Theorem \ref{theorem3} holds.
Then there exists a positive constant $C:=C(\varepsilon)$ depends on $\varepsilon$ such that 
 the solution of \dref{1.1fghyuisda} from Lemma \ref{lemma70} satisfies
\begin{equation}
\begin{array}{rl}
\|A^\gamma u_{\varepsilon}(\cdot, t)\|_{L^2(\Omega)}\leq&\disp{C~~ \mbox{for all}~~ t\in(0,T_{max,\varepsilon})}\\
\end{array}
\label{cz2.57151ccvhhjjjkkkuuiivhccvvhjjjkkhhggjjllll}
\end{equation}
as well as
\begin{equation}
\begin{array}{rl}
\| u_{\varepsilon}(\cdot, t)\|_{L^{\infty}(\Omega)}\leq&\disp{C~~ \mbox{for all}~~ t\in(0,T_{max,\varepsilon})}\\
\end{array}
\label{cz2.57ghhhh151ccvhhjjjkkkuuiivhccvvhjjjkkhhggjjllll}
\end{equation}
and
\begin{equation}
\begin{array}{rl}
\|\nabla c_{\varepsilon}(\cdot, t)\|_{L^{q}(\Omega)}\leq&\disp{C~~ \mbox{for all}~~ t\in(0,T_{max,\varepsilon})}\\
\end{array}
\label{cz2ddff.57151ccvhhjjjkkkuuiivhccvvhjjjkkhhggjjllll}
\end{equation}
and some $3<q<6.$
\end{lemma}
 \begin{proof}
Let $h_{\varepsilon}(x,t)=\mathcal{P}[n_{\varepsilon}\nabla \phi-\kappa (Y_{\varepsilon}u_{\varepsilon} \cdot \nabla)u_{\varepsilon} ]$.
Due to $\alpha>\frac{1}{3}$, then  along with \dref{czfvgb2.5ghhjuyucffggcvviihjj},
\dref{dd1.1fghyuisdakkkllljjjkk} and \dref{ssdcfvgdhhjjdfghgghjjnnhhkklld911cz2.5ghju48},
there exist  positive constants $q_0>\frac{3}{2}$ and
$C_{1}$ such that
 \begin{equation}
\|n_{\varepsilon}(\cdot,t)\|_{L^{q_0}(\Omega)}\leq C_{1} ~~~\mbox{for all}~~ t\in(0,T_{max,\varepsilon})
\label{334444zjsczddff2.52yyujiii97dfggggx96302222114}
\end{equation}
and
 \begin{equation}
\|h_{\varepsilon}(\cdot,t)\|_{L^{q_0}(\Omega)}\leq C_{1} ~~~\mbox{for all}~~ t\in(0,T_{max,\varepsilon}).
\label{33444cfghhh29fggggx96302222114}
\end{equation}
%
 Hence, by $q_0>\frac{3}{2}$,
 we pick an arbitrary $\gamma\in (\frac{3}{4}, 1)$ and therefore, $-\gamma-\frac{3}{2}(\frac{1}{q_0}-\frac{1}{2})>-1$.
 Then in view of  the  smoothing properties of the
Stokes semigroup (\cite{Giga1215}), we derive  that for some $C_{2} > 0$ and $C_{3} > 0$, we have
\begin{equation}
\begin{array}{rl}
\|A^\gamma u_{\varepsilon}(\cdot, t)\|_{L^2(\Omega)}\leq&\disp{\|A^\gamma
e^{-tA}u_0\|_{L^2(\Omega)} +\int_0^t\|A^\gamma e^{-(t-\tau)A}h_{\varepsilon}(\cdot,\tau)d\tau\|_{L^2(\Omega)}d\tau}\\
\leq&\disp{\|A^\gamma u_0\|_{L^2(\Omega)} +C_{2}\int_0^t(t-\tau)^{-\gamma-\frac{3}{2}(\frac{1}{q_0}-\frac{1}{2})}e^{-\lambda(t-\tau)}\|h_{\varepsilon}(\cdot,\tau)\|_{L^{q_0}(\Omega)}d\tau}\\
\leq&\disp{C_{3}~~ \mbox{for all}~~ t\in(0,T_{max,\varepsilon}).}\\
\end{array}
\label{cz2.57151ccvvhccvvhjjjkkhhggjjllll}
\end{equation}
Observe that $\gamma>\frac{3}{4},$
 $D(A^\gamma)$ is continuously embedded into $L^\infty(\Omega)$, therefore, due to \dref{cz2.57151ccvvhccvvhjjjkkhhggjjllll}, we derive that there exists a positive constant $C_{4}$ such that
 \begin{equation}
\begin{array}{rl}
\|u_{\varepsilon}(\cdot, t)\|_{L^\infty(\Omega)}\leq  C_{4}~~ \mbox{for all}~~ t\in(0,T_{max,\varepsilon}).\\
\end{array}
\label{cz2.5jkkcvvvhjkfffffkhhgll}
\end{equation}
On the there hand, observing that \dref{ddxxxcvvddcvddffbbggddczv.5ghcfg924ghyuji}, with the help of the Sobolev imbedding theorem, we derive for any $l<6$, there exists a positive constant $C_{5}$ such that
\begin{equation}
\begin{array}{rl}
\|c_{\varepsilon}(\cdot, t)\|_{L^l(\Omega)}\leq  C_{5}~~ \mbox{for all}~~ t\in(0,T_{max,\varepsilon}),\\
\end{array}
\label{cz2.5715jkkjjkkkkcvccvvhj4456ddffff777jjkddfffffkhhgll}
\end{equation}
which together with the H\"{o}lder inequality implies that for any fixed $\tilde{q}\in(3,6)$
\begin{equation}
\begin{array}{rl}
\|c_{\varepsilon}(\cdot, t)\|_{L^{\tilde{q}}(\Omega)}\leq  C_{6}~~ \mbox{for all}~~ t\in(0,T_{max,\varepsilon}).\\
\end{array}
\label{cz2.5715jkkjjkkkkcvccvvhj4456777jjkddfffffkhhgll}
\end{equation}
Now, involving the variation-of-constants formula
for $c_{\varepsilon}$ and applying $\nabla\cdot u_{\varepsilon}=0$ in $x\in \Omega, t>0$, we have
\begin{equation}
c_{\varepsilon}(t)=e^{t(\Delta-1)}c_0 +\int_{0}^{t}e^{(t-s)(\Delta-1)}(F_\varepsilon(n_{\varepsilon}(s))+\nabla \cdot(u_{\varepsilon}(s) c_{\varepsilon}(s)) ds,~~ t\in(0, T_{max,\varepsilon}),
\label{5555hhjjjfghbnmcz2.5ghju48cfg924ghyuji}
\end{equation}
which implies that
\begin{equation}
\begin{array}{rl}
&\disp{\|\nabla c_{\varepsilon}(\cdot, t)\|_{L^{q}(\Omega)}}\\
\leq&\disp{\|\nabla e^{t(\Delta-1)} c_0\|_{L^{q}(\Omega)}+
\int_{0}^t\|\nabla e^{(t-s)(\Delta-1)}F_\varepsilon(n_{\varepsilon}(s))\|_{L^q(\Omega)}ds+\int_{0}^t\|\nabla e^{(t-s)(\Delta-1)}\nabla \cdot(u_{\varepsilon}(s)
c_{\varepsilon}(s))\|_{L^q(\Omega)}ds,}\\
\end{array}
\label{44444zjccfgghfgjcvbscz2.5297x96301ku}
\end{equation}
where $3<q <\min\{\frac{3q_0}{(3-q_0)_{+}},\tilde{q}\}$.
To deal with the right-hand side of \dref{44444zjccfgghfgjcvbscz2.5297x96301ku}, in view of \dref{ccvvx1.731426677gg}, we first use Lemma \ref{llssdrffmmggnnccvvccvvkkkkgghhkkllvvlemma45630}    to get that
\begin{equation}
\begin{array}{rl}
\|\nabla e^{t(\Delta-1)} c_0\|_{L^{q}(\Omega)}\leq &\disp{C_{7}~~ \mbox{for all}~~ t\in(0,T_{max,\varepsilon}).}\\
\end{array}
\label{zjccffgbhjcghhhjjjvvvbscz2.5297x96301ku}
\end{equation}
Since \dref{334444zjsczddff2.52yyujiii97dfggggx96302222114} and \dref{cz2.5715jkkjjkkkkcvccvvhj4456777jjkddfffffkhhgll}  yields to
$$-\frac{1}{2}-\frac{3}{2}(\frac{1}{q_0}-\frac{1}{q})>-1,$$
 together with this, in view of \dref{ffggg1.ffggvddfghhghhhhbbhhjjjnxxccvvn1},
 using Lemma \ref{llssdrffmmggnnccvvccvvkkkkgghhkkllvvlemma45630} again,  the second term of the right-hand side is estimated
as
\begin{equation}
\begin{array}{rl}
&\disp{\int_{0}^t\|\nabla e^{(t-s)(\Delta-1)}F_\varepsilon(n_{\varepsilon}(s))\|_{L^{q}(\Omega)}ds}\\
\leq&\disp{C_{8}\int_{0}^t[1+(t-s)^{-\frac{1}{2}-\frac{3}{2}(\frac{1}{q_0}-\frac{1}{q})}] e^{-(t-s)}\|n_{\varepsilon}(s)\|_{L^{q_0}(\Omega)}ds}\\
\leq&\disp{C_{9}~~ \mbox{for all}~~ t\in(0,T_{max,\varepsilon}).}\\
\end{array}
\label{zjccffgbhjcvvvbscz2.5297x96301ku}
\end{equation}

Finally we will deal with the third term on the right-hand
side of \dref{44444zjccfgghfgjcvbscz2.5297x96301ku}. Indeed, we choose $0 < \iota < \frac{1}{2}$  
 satisfying $\frac{1}{2} + \frac{3}{2}(\frac{1}{\tilde{q}}-\frac{1}{q}) <\iota$  and $\tilde{\kappa}\in(0, \frac{1}{2}-\iota)$. In view of the H\"{o}lder inequality,
 then we derive from Lemma \ref{llssdrffmmggnnccvvccvvkkkkgghhkkllvvlemma45630} and \dref{cz2.5715jkkjjkkkkcvccvvhj4456777jjkddfffffkhhgll}  and \dref{cz2.5jkkcvvvhjkfffffkhhgll} that there exist constants $C_{10},C_{11},C_{12}$ and $C_{13}$ such that
\begin{equation}
\begin{array}{rl}
&\disp{\int_{0}^t\|\nabla e^{(t-s)(\Delta-1)}\nabla \cdot(u_{\varepsilon}(s) c_{\varepsilon}(s))\|_{L^{\tilde{q}}(\Omega)}ds}\\
\leq&\disp{C_{10}\int_{0}^t\|(-\Delta+1)^\iota e^{(t-s)(\Delta-1)}\nabla \cdot(u_\varepsilon(s) c_\varepsilon(s))\|_{L^{q}(\Omega)}ds}\\
\leq&\disp{C_{11}\int_{0}^t(t-s)^{-\iota-\frac{1}{2}-\tilde{\kappa}} e^{-\lambda(t-s)}\|u_\varepsilon(s) c_\varepsilon(s)\|_{L^{\tilde{q}}(\Omega)}ds}\\
\leq&\disp{C_{12}\int_{0}^t(t-s)^{-\iota-\frac{1}{2}-\tilde{\kappa}} e^{-\lambda(t-s)}\|u_\varepsilon(s)\|_{L^{\infty}(\Omega)}\| c_\varepsilon(s)\|_{L^{\tilde{q}}(\Omega)}ds}\\
\leq&\disp{C_{13}~~ \mbox{for all}~~ t\in(0,T_{max,\varepsilon}).}\\
\end{array}
\label{zjccffgbhjcvddfgghhvvbscz2.5297x96301ku}
\end{equation}
Here we have used the fact that
$$\begin{array}{rl}\disp\int_{0}^t(t-s)^{-\iota-\frac{1}{2}-\tilde{\kappa}} e^{-\lambda(t-s)}ds\leq&\disp{\int_{0}^{\infty}\sigma^{-\iota-\frac{1}{2}-\tilde{\kappa}} e^{-\lambda\sigma}d\sigma<+\infty.}\\
\end{array}
$$
Finally, collecting  \dref{44444zjccfgghfgjcvbscz2.5297x96301ku}--\dref{zjccffgbhjcvddfgghhvvbscz2.5297x96301ku},
 we can obtain there  exists a positive constant $C_{14}$ such that
  \begin{equation}
\int_{\Omega}|\nabla {c_{\varepsilon}}(t)|^{q}\leq C_{14}~~\mbox{for all}~~ t\in(0, T_{max,\varepsilon})~~\mbox{and some}~~q\in(3,\min\{\frac{3q_0}{(3-q_0)_{+}},\tilde{q}\}).
\label{ffgbbcz2.5ghjusseeeddd48cfg924ghyuji}
\end{equation}
The proof Lemma \ref{sss222lemma4444556645630223} is complete.
\end{proof}
Then we shall establish global existence in approximate problem \dref{1.1fghyuisda} by using Lemmas \ref{lemmaghjssddgghhmk4563025xxhjklojjkkk}--\ref{xccffgghhlemma4563025xxhjkloghyui}.
\begin{lemma}\label{kkklemmaghjmk4563025xxhjklojjkkk}
Let $\alpha>\frac{1}{3}$. Then
for all $\varepsilon\in(0,1),$ the solution of  \dref{1.1fghyuisda} is global in time.
\end{lemma}
\begin{proof}
Assuming that $T_{max,\varepsilon}$ be finite for some $\varepsilon\in(0,1)$.
Fix $T\in (0, T_{max,\varepsilon})$.  Let $M(T):=\sup_{t\in(0,T)}\|n_{\varepsilon}(\cdot,t)\|_{L^\infty(\Omega)}$ and $\tilde{h}_{\varepsilon}:=F'_{\varepsilon}(n_{\varepsilon})S_\varepsilon(x, n_{\varepsilon}, c_{\varepsilon})\nabla c_{\varepsilon}+u_\varepsilon$. Then by  Lemma \ref{sss222lemma4444556645630223}, \dref{x1.73142vghf48gg}  and \dref{1.ffggvddfghhghhhhbbnxxccvvn1},
there exists $C_1 > 0$ such that
\begin{equation}
\begin{array}{rl}
\|\tilde{h}_{\varepsilon}(\cdot, t)\|_{L^{q}(\Omega)}\leq&\disp{C_1~~ \mbox{for all}~~ t\in(0,T_{max,\varepsilon})~~\mbox{and some }~~3<q<6.}\\
\end{array}
\label{cz2ddff.57151ccvhhjjjkkkuuifghhhivhccvvhjjjkkhhggjjllll}
\end{equation}
Hence, due to the fact that $\nabla\cdot u_{\varepsilon}=0$,  again,  by means of an
associate variation-of-constants formula for $v$, we can derive
\begin{equation}
n_{\varepsilon}(t)=e^{(t-t_0)\Delta}n_{\varepsilon}(\cdot,t_0)-\int_{t_0}^{t}e^{(t-s)\Delta}\nabla\cdot(n_{\varepsilon}(\cdot,s)\tilde{h}_{\varepsilon}(\cdot,s)) ds,~~ t\in(t_0, T),
\label{5555fghbnmcz2.5ghjjjkkklu48cfg924ghyuji}
\end{equation}
where $t_0 := (t-1)_{+}$.
If $t\in(0,1]$,
by virtue of the maximum principle, we derive that
\begin{equation}
\begin{array}{rl}
\|e^{(t-t_0)\Delta}n_{\varepsilon}(\cdot,t_0)\|_{L^{\infty}(\Omega)}\leq &\disp{\|n_0\|_{L^{\infty}(\Omega)},}\\
\end{array}
\label{zjccffgbhjffghhjcghhhjjjvvvbscz2.5297x96301ku}
\end{equation}
while if $t > 1$ then with the help of the  $L^p$-$L^q$ estimates for the Neumann heat semigroup and Lemma \ref{fvfgfflemma45}, we conclude that
\begin{equation}
\begin{array}{rl}
\|e^{(t-t_0)\Delta}n_{\varepsilon}(\cdot,t_0)\|_{L^{\infty}(\Omega)}\leq &\disp{C_2(t-t_0)^{-\frac{3}{2}}\|n_{\varepsilon}(\cdot,t_0)\|_{L^{1}(\Omega)}\leq C_3.}\\
\end{array}
\label{zjccffgbhjffghhjcghghjkjjhhjjjvvvbscz2.5297x96301ku}
\end{equation}
Finally, we fix an arbitrary $p\in(3,q)$ and then once more invoke known smoothing
properties of the
Stokes semigroup  and the H\"{o}lder inequality to find $C_4 > 0$ such that
\begin{equation}
\begin{array}{rl}
&\disp\int_{t_0}^t\| e^{(t-s)\Delta}\nabla\cdot(n_{\varepsilon}(\cdot,s)\tilde{h}_{\varepsilon}(\cdot,s)\|_{L^\infty(\Omega)}ds\\
\leq&\disp C_4\int_{t_0}^t(t-s)^{-\frac{1}{2}-\frac{3}{2p}}\|n_{\varepsilon}(\cdot,s)\tilde{h}_{\varepsilon}(\cdot,s)\|_{L^p(\Omega)}ds\\
\leq&\disp C_4\int_{t_0}^t(t-s)^{-\frac{1}{2}-\frac{3}{2p}}\| n_{\varepsilon}(\cdot,s)\|_{L^{\frac{pq}{q-p}}(\Omega)}\|\tilde{h}_{\varepsilon}(\cdot,s)\|_{L^{q}(\Omega)}ds\\
\leq&\disp C_4\int_{t_0}^t(t-s)^{-\frac{1}{2}-\frac{3}{2p}}\| u_{\varepsilon}(\cdot,s)\|_{L^{\infty}(\Omega)}^b\| u_{\varepsilon}(\cdot,s)\||_{L^1(\Omega)}^{1-b}\|\tilde{h}_{\varepsilon}(\cdot,s)\|_{L^{q}(\Omega)}ds\\
\leq&\disp C_5M^b(T)~~\mbox{for all}~~ t\in(0, T),\\
\end{array}
\label{ccvbccvvbbnnndffghhjjvcvvbccfbbnfgbghjjccmmllffvvggcvvvvbbjjkkdffzjscz2.5297x9630xxy}
\end{equation}
where $b:=\frac{pq-q+p}{pq}\in(0,1)$
and
$$C_5:=C_4C_1^{2-b}\int_{0}^{1}\sigma^{-\frac{1}{2}-\frac{3}{2p}}d\sigma.$$
Since $p>3$, we conclude that
$-\frac{1}{2}-\frac{3}{2p}>-1$.
In combination with \dref{5555fghbnmcz2.5ghjjjkkklu48cfg924ghyuji}--\dref{ccvbccvvbbnnndffghhjjvcvvbccfbbnfgbghjjccmmllffvvggcvvvvbbjjkkdffzjscz2.5297x9630xxy} and using the definition of $M(T)$
we obtain
$C_6 > 0$ such that
\begin{equation}
\begin{array}{rl}
&\disp  M(T)\leq C_6+C_6M^b(T)~~\mbox{for all}~~ T\in(0, T_{max,\varepsilon}).\\
\end{array}
\label{ccvbccvvbbnnndffghhjjvcvvfghhhbccfbbnfgbghjjccmmllffvvggcvvvvbbjjkkdffzjscz2.5297x9630xxy}
\end{equation}
Hence, in view of $b<1$, with  some basic calculation, in light of  $T\in (0, T_{max,\varepsilon})$ was arbitrary,
we can get
\begin{equation}
\begin{array}{rl}
\|n_{\varepsilon}(\cdot, t)\|_{L^{\infty}(\Omega)}\leq&\disp{C_7~~ \mbox{for all}~~ t\in(0,T_{max,\varepsilon}).}\\
\end{array}
\label{cz2.57ghhhh151ccvhhjjjkkkffgghhuuiivhccvvhjjjkkhhggjjllll}
\end{equation}
In order to prove the boundedness of $\|\nabla c_{\varepsilon}(\cdot, t)\|_{L^\infty(\Omega)}$, 
we rewrite the variation-of-constants formula for $c_{\varepsilon}$ in the form
$$c_{\varepsilon}(\cdot, t) = e^{t(\Delta-1) }c_0 +\int_{0}^te^{(t-s)(\Delta-1)}[F_{\varepsilon}(n_{\varepsilon})(s)-u_{\varepsilon}(s)\cdot\nabla  c_{\varepsilon}(s)]ds~~ \mbox{for all}~~ t\in(0,T_{max,\varepsilon}).$$
Now, we choose $\theta\in(\frac{1}{2}+\frac{3}{2q},1),$ where $3<q<6$ (see \dref{ffgbbcz2.5ghjusseeeddd48cfg924ghyuji}),
 then the domain of the fractional power $D((-\Delta + 1)^\theta)\hookrightarrow W^{1,\infty}(\Omega)$. Hence, in view of $L^p$-$L^q$ estimates associated heat semigroup,  \dref{cz2.57ghhhh151ccvhhjjjkkkuuiivhccvvhjjjkkhhggjjllll}, \dref{cz2ddff.57151ccvhhjjjkkkuuiivhccvvhjjjkkhhggjjllll} and \dref{ffggg1.ffggvddfghhghhhhbbhhjjjnxxccvvn1}, we derive  that there exist positive constants $\lambda,C_{8},C_{9},C_{10}$ and $C_{11}$ such that
\begin{equation}
\begin{array}{rl}
&\| c_{\varepsilon}(\cdot, t)\|_{W^{1,\infty}(\Omega)}\\
\leq&\disp{C_{8}\|(-\Delta+1)^\theta c_{\varepsilon}(\cdot, t)\|_{L^{q}(\Omega)}}\\
\leq&\disp{C_{9}t^{-\theta}e^{-\lambda t}\|c_0\|_{L^{q}(\Omega)}+C_{9}\int_{0}^t(t-s)^{-\theta}e^{-\lambda(t-s)}
\|(F_{\varepsilon}(n_{\varepsilon})-u_{\varepsilon} \cdot \nabla c_{\varepsilon})(s)\|_{L^{q}(\Omega)}ds}\\
\leq&\disp{C_{10}+C_{10}\int_{0}^t(t-s)^{-\theta}e^{-\lambda(t-s)}[\|n_{\varepsilon}(s)\|_{L^q(\Omega)}+\|u_{\varepsilon}(s)\|_{L^\infty(\Omega)}
\|\nabla c_{\varepsilon}(s)\|_{L^q(\Omega)}]ds}\\
\leq&\disp{C_{11}~~ \mbox{for all}~~ t\in(0,T_{max,\varepsilon}).}\\
\end{array}
\label{zjccffgbhjcvvvbscz2.5297x96301ku}
\end{equation}
Here we have used the H\"{o}lder inequality as well as
$$\int_{0}^t(t-s)^{-\theta}e^{-\lambda(t-s)}\leq \int_{0}^{\infty}\sigma^{-\theta}e^{-\lambda\sigma}d\sigma<+\infty.$$
In view of  \dref{cz2.57151ccvhhjjjkkkuuiivhccvvhjjjkkhhggjjllll},
\dref{zjccffgbhjcvvvbscz2.5297x96301ku} and \dref{cz2.57ghhhh151ccvhhjjjkkkffgghhuuiivhccvvhjjjkkhhggjjllll}, we apply Lemma \ref{lemma70} to reach a contradiction.
\end{proof}

\section{Regularity properties of time derivatives}

In order to prove
the limit functions $n,c$ and $u$ gained below (see Section 6), we will rely on an
additional regularity estimate for $n_\varepsilon F'_{\varepsilon}(n_{\varepsilon})S_\varepsilon(x, n_{\varepsilon}, c_{\varepsilon})\nabla c_\varepsilon,u_\varepsilon\cdot\nabla c_\varepsilon,n_\varepsilon u_\varepsilon$ and
$c_\varepsilon u_\varepsilon$.

\begin{lemma}\label{4455lemma45630hhuujjuuyytt}
Let $\alpha>\frac{1}{3}$,
\dref{dd1.1fghyuisdakkkllljjjkk} and \dref{ccvvx1.731426677gg}
 hold.
 Then for any $T>0, $
  one can find $C > 0$ independent of $\varepsilon$ such that 
\begin{equation}
 \begin{array}{ll}
  \disp\int_0^T\int_{\Omega}\left[|n_\varepsilon F'_{\varepsilon}(n_{\varepsilon})S_\varepsilon(x, n_{\varepsilon}, c_{\varepsilon})\nabla c_\varepsilon|^{\frac{3\alpha+1}{2}} +|n_\varepsilon u_\varepsilon|^{\frac{10(3\alpha+1)}{9(\alpha+2)}}\right]\leq C(T+1)~~\mbox{if}~~~\frac{1}{3}<\alpha\leq\frac{1}{2},\\
   \end{array}\label{1.1dddfgbhnjmdfgeddvbnmklllhyussddisda}
\end{equation}
\begin{equation}
 \begin{array}{ll}
  \disp\int_0^T\int_{\Omega}\left[|n_\varepsilon F'_{\varepsilon}(n_{\varepsilon})S_\varepsilon(x, n_{\varepsilon}, c_{\varepsilon})\nabla c_\varepsilon|^{\frac{10\alpha}{3+2\alpha}}+|n_\varepsilon u_\varepsilon|^{\frac{10\alpha}{3(\alpha+1)}} \right]\leq C(T+1)~~\mbox{if}~~~\frac{1}{2}<\alpha<1\\
   \end{array}\label{1.1dddfgbhnjmdfgeddvbnhhjjjmklllhyussddisda}
\end{equation}
as well as
\begin{equation}
 \begin{array}{ll}
  \disp\int_0^T\int_{\Omega}\left[|n_\varepsilon F'_{\varepsilon}(n_{\varepsilon})S_\varepsilon(x, n_{\varepsilon}, c_{\varepsilon})\nabla c_\varepsilon|^{2}+|n_\varepsilon u_\varepsilon|^{\frac{5}{3}}\right] \leq C(T+1)~~\mbox{if}~~~\alpha\geq1\\
   \end{array}\label{1.1dddgghhjfgbhnjmdfgeddvbnhhjjjmklllhyussddisda}
\end{equation}
and
\begin{equation}
 \begin{array}{ll}
  \disp\int_0^T\int_{\Omega}\left[|u_\varepsilon\cdot\nabla c_\varepsilon|^{\frac{5}{4}} +|c_\varepsilon u_\varepsilon |^{\frac{5}{3}}\right]  \leq C(T+1).\\
   \end{array}\label{1.1dddfgbhnjmdfgeddvbnmklllhyussdddfgggdisda}
\end{equation}
\end{lemma}
\begin{proof}
Firstly, by \dref{x1.73142vghf48gg}, \dref{1.ffggvddfghhghhhhbbnxxccvvn1}
and \dref{3.10gghhjuuloollyuigghhhyy}, we derive that
$$n_\varepsilon F'_{\varepsilon}(n_{\varepsilon})S_\varepsilon(x, n_{\varepsilon}, c_{\varepsilon})\leq C_Sn_{\varepsilon}^{(1-\alpha)_{+}}$$
with $(1-\alpha)_{+}=\max\{0,1-\alpha\}.$
Case $\frac{1}{3}<\alpha\leq\frac{1}{2}$:  It is not difficult to verify that
$$\frac{2}{3\alpha+1}=\frac{1}{2}+\frac{3}{6\alpha+2}(1-\alpha)$$
and
$$\frac{9(\alpha+2)}{10(3\alpha+1)}=\frac{3}{10}+\frac{3}{6\alpha+2}.$$
From this and
by \dref{bnmbncz2.ffghh5ghhjuyuivvbnnihjj}, and recalling \dref{ddffbnmbnddfgffggjjkkuuiicz2dvgbhh.t8ddhhhyuiihjj} and the H\"{o}lder inequality, we can obtain
\dref{1.1dddfgbhnjmdfgeddvbnmklllhyussddisda}. Other cases, can be proved very similarly. Therefore, we omit it.
\end{proof}

To prepare our subsequent compactness properties of
$(n_\varepsilon, c_\varepsilon,u_\varepsilon)$ by means of the Aubin-Lions lemma (see Simon \cite{Simon}), we use Lemmas \ref{fvfgfflemma45}-\ref{lemmddaghjsffggggsddgghhmk4563025xxhjklojjkkk} to obtain
the following regularity property with respect to the time variable.
\begin{lemma}\label{qqqqlemma45630hhuujjuuyytt}
Let $\alpha>\frac{1}{3}$,
\dref{dd1.1fghyuisdakkkllljjjkk} and \dref{ccvvx1.731426677gg}
 hold.
 Then for any $T>0, $
  one can find $C > 0$ independent if $\varepsilon$ such that 
\begin{equation}
 \begin{array}{ll}
\disp\int_0^T\|\partial_tn_\varepsilon(\cdot,t)\|_{(W^{2,4}(\Omega))^*}dt  \leq C(T+1)
~~\mbox{if}~~~\frac{1}{3}<\alpha\leq\frac{8}{21},\\
   \end{array}\label{1.1ddfgeddvbnmklllhyuisda}
\end{equation}
\begin{equation}
 \begin{array}{ll}
\disp\int_0^T\|\partial_tn_\varepsilon(\cdot,t)\|_{({W^{1,\frac{10(3\alpha+1)}{21\alpha-8}}(\Omega)})^*}^{\frac{10(3\alpha+1)}{9(\alpha+2)}}dt  \leq C(T+1)
~~\mbox{if}~~~\frac{8}{21}<\alpha\leq\frac{1}{2},\\
   \end{array}\label{1.1ddfgeddvbnmkffgghlllhyuisda}
\end{equation}
\begin{equation}
 \begin{array}{ll}
\disp\int_0^T\|\partial_tn_\varepsilon(\cdot,t)\|_{({W^{1,\frac{10\alpha}{7\alpha-3}}(\Omega)})^*}^{\frac{10\alpha}{3(\alpha+1)}}dt  \leq C(T+1)
~~\mbox{if}~~~\frac{1}{2}<\alpha<1,\\
   \end{array}\label{1.1ddfgeddgghhvbnmklllhyuisda}
\end{equation}
\begin{equation}
 \begin{array}{ll}
\disp\int_0^T\|\partial_tn_\varepsilon(\cdot,t)\|_{({W^{1,\frac{5}{2}}(\Omega)})^*}^{\frac{5}{3}}dt  \leq C(T+1)
~~\mbox{if}~~~\alpha\geq1\\
   \end{array}\label{1.1ddfgeddvbnjjklmklllhyuisda}
\end{equation}
as well as
\begin{equation}
 \begin{array}{ll}
  \disp\int_0^T\|\partial_tc_\varepsilon(\cdot,t)\|_{(W^{1,5}(\Omega))^*}^{\frac{5}{4}}dt  \leq C(T+1)\\
   \end{array}\label{wwwwwqqqq1.1dddfgbhnjmdfgeddvbnmklllhyussddisda}
\end{equation}
and
\begin{equation}
 \begin{array}{ll}
  \disp\int_0^T\|\partial_tu_\varepsilon(\cdot,t)\|_{(W^{1,5}_{0,\sigma}(\Omega))^*}^{\frac{5}{4}}dt  \leq C(T+1).\\
   \end{array}\label{wwwwwqqqq1.1dddfgkkllbhddffgggnjmdfgeddvbnmklllhyussddisda}
\end{equation}
\end{lemma}
\begin{proof} In the proof, we only prove the case $\frac{8}{21}<\alpha\leq\frac{1}{2},$ since, other case can be proved similarly.
Firstly, an elementary calculation ensures that
 \begin{equation}\frac{3\alpha+1}{2}>\frac{10(3\alpha+1)}{9(\alpha+2)}~~~\mbox{and}~~~1=\frac{9(\alpha+2)}{10(3\alpha+1)}+\frac{21\alpha-8}{10(3\alpha+1)}.
 \label{gbhncvbmdcfvgcz2.5ghju4ddfghh8}
\end{equation}
Next,
testing the first equation of \dref{1.1fghyuisda}
 by certain   $\varphi\in C^{\infty}(\bar{\Omega})$, we have
 \begin{equation}
\begin{array}{rl}
&\disp\left|\int_{\Omega}(n_{\varepsilon,t})\varphi\right|\\
 =&\disp{\left|\int_{\Omega}\left[\Delta n_{\varepsilon}-\nabla\cdot(n_{\varepsilon}F'_{\varepsilon}(n_{\varepsilon})S_\varepsilon(x, n_{\varepsilon}, c_{\varepsilon})\nabla c_{\varepsilon})-u_{\varepsilon}\cdot\nabla n_{\varepsilon}\right]\varphi\right|}
\\
=&\disp{\left|\int_\Omega \left[-\nabla n_{\varepsilon}\cdot\nabla\varphi+n_{\varepsilon}F'_{\varepsilon}(n_{\varepsilon})S_\varepsilon(x, n_{\varepsilon}, c_{\varepsilon})\nabla c_{\varepsilon}\cdot\nabla\varphi+ n_{\varepsilon}u_{\varepsilon}\cdot\nabla  \varphi\right]\right|}\\
\leq&\disp{\left\{\|\nabla n_{\varepsilon}\|_{L^{\frac{10(3\alpha+1)}{9(\alpha+2)}}(\Omega)}+ \|n_{\varepsilon}F'_{\varepsilon}(n_{\varepsilon})S_\varepsilon(x, n_{\varepsilon}, c_{\varepsilon})\nabla c_{\varepsilon}\|_{L^{\frac{10(3\alpha+1)}{9(\alpha+2)}}(\Omega)}+ \|n_{\varepsilon}u_{\varepsilon}\|_{L^{\frac{10(3\alpha+1)}{9(\alpha+2)}}(\Omega)}
\right\}\|\varphi\|_{W^{1,\frac{10(3\alpha+1)}{21\alpha-8}}(\Omega)}}\\
\end{array}
\label{gbhncvbmdcfvgcz2.5ghju48}
\end{equation}
for all $t>0$.
Along with \dref{bnmbncz2.ffghh5ghhjuyuivvbnnihjj} and \dref{1.1dddfgbhnjmdfgeddvbnmklllhyussddisda}, further implies that
\begin{equation}
\begin{array}{rl}
&\disp\int_0^T\|\partial_tn_\varepsilon(\cdot,t)\|_{({W^{1,\frac{10(3\alpha+1)}{21\alpha-8}}(\Omega)})^*}^{\frac{10(3\alpha+1)}{9(\alpha+2)}}dt \\
\leq&\disp{\int_0^T\left\{\|\nabla n_{\varepsilon}\|_{L^{\frac{10(3\alpha+1)}{9(\alpha+2)}}(\Omega)}+ \|n_{\varepsilon}F'_{\varepsilon}(n_{\varepsilon})S_\varepsilon(x, n_{\varepsilon}, c_{\varepsilon})\nabla c_{\varepsilon}\|_{L^{\frac{10(3\alpha+1)}{9(\alpha+2)}}(\Omega)}+ \|n_{\varepsilon}u_{\varepsilon}\|_{L^{\frac{10(3\alpha+1)}{9(\alpha+2)}}(\Omega)}
\right\}}^{\frac{10(3\alpha+1)}{9(\alpha+2)}}dt
\\
\leq&\disp{C_1\int_0^T\left\{\|\nabla n_{\varepsilon}\|_{L^{\frac{10(3\alpha+1)}{9(\alpha+2)}}(\Omega)}^{\frac{10(3\alpha+1)}{9(\alpha+2)}}+ \|n_{\varepsilon}F'_{\varepsilon}(n_{\varepsilon})S_\varepsilon(x, n_{\varepsilon}, c_{\varepsilon})\nabla c_{\varepsilon}\|_{L^{\frac{10(3\alpha+1)}{9(\alpha+2)}}(\Omega)}^{\frac{10(3\alpha+1)}{9(\alpha+2)}}+ \|n_{\varepsilon}u_{\varepsilon}\|_{L^{\frac{10(3\alpha+1)}{9(\alpha+2)}}(\Omega)}^{\frac{10(3\alpha+1)}{9(\alpha+2)}}
\right\}}dt\\
\end{array}
\label{gbhncvbmdcfvgczffghhh2.5ghju48}
\end{equation}
where $C_1$ is a positive constant independent of $\varepsilon$.
Finally,  \dref{1.1ddfgeddvbnmklllhyuisda} is a consequence of  \dref{bnmbncz2.ffghh5ghhjuyuivvbnnihjj}, \dref{1.1dddfgbhnjmdfgeddvbnmklllhyussddisda}, \dref{gbhncvbmdcfvgcz2.5ghju4ddfghh8} and the H\"{o}lder ineqaulity.
\end{proof}

%
%


\section{Passing to the limit. Proof of Theorem  \ref{theorem3}}

Based on above lemmas and by extracting suitable subsequences in a standard
way, we could see the solution of \dref{1.1} is indeed globally solvable.
\begin{lemma}\label{lemma45630223}

Let
 \dref{x1.73142vghf48rtgyhu}, \dref{x1.73142vghf48gg} 
 and
\dref{dd1.1fghyuisdakkkllljjjkk} and \dref{ccvvx1.731426677gg}
 hold, and suppose that $\alpha>\frac{1}{3}.$ 
 There exists $(\varepsilon_j)_{j\in \mathbb{N}}\subset (0, 1)$ such that $\varepsilon_j\searrow 0$ as $j\rightarrow\infty$, and such that as $\varepsilon: = \varepsilon_j\searrow 0$
we have
\begin{equation}
n_\varepsilon\rightarrow n ~~\mbox{a.e.}~~\mbox{in}~~\Omega\times(0,\infty)~~\mbox{and in}~~ L_{loc}^{r}(\bar{\Omega}\times[0,\infty))~~\mbox{with}~~r= \left\{\begin{array}{ll}
 \frac{3\alpha+1}{2}~~\mbox{if}~~\frac{1}{3}<\alpha\leq\frac{1}{2},\\
  \frac{10\alpha}{3+2\alpha}~~\mbox{if}~~\frac{1}{2}<\alpha<1,\\
  2~~\mbox{if}~~\alpha\geq1,\\
   \end{array}\right.\label{zjscz2.5297x963ddfgh0ddfggg6662222tt3}
\end{equation}
\begin{equation}
\nabla n_\varepsilon\rightharpoonup \nabla n ~~\mbox{in}~~\Omega\times(0,\infty)~~\mbox{and in}~~ L_{loc}^{r}(\bar{\Omega}\times[0,\infty))~~\mbox{with}~~r= \left\{\begin{array}{ll}
 \frac{3\alpha+1}{2}~~\mbox{if}~~\frac{1}{3}<\alpha\leq\frac{1}{2},\\
  \frac{10\alpha}{3+2\alpha}~~\mbox{if}~~\frac{1}{2}<\alpha<1,\\
  2~~\mbox{if}~~\alpha\geq1,\\
   \end{array}\right.\label{zjscz2.5297x963ddfgh0ddgghjjfggg6662222tt3}
\end{equation}
\begin{equation}
c_\varepsilon\rightarrow c ~~\mbox{in}~~ L^{2}_{loc}(\bar{\Omega}\times[0,\infty))~~\mbox{and}~~\mbox{a.e.}~~\mbox{in}~~\Omega\times(0,\infty),
 \label{zjscz2.fgghh5297x963ddfgh0ddfggg6662222tt3}
\end{equation}
\begin{equation}
\nabla c_\varepsilon\rightarrow \nabla c ~~\mbox{a.e.}~~\mbox{in}~~\Omega\times(0,\infty),
 \label{1.1ddhhyujiiifgghhhge666ccdf2345ddvbnmklllhyuisda}
\end{equation}
\begin{equation}
u_\varepsilon\rightarrow u~~\mbox{in}~~ L_{loc}^2(\bar{\Omega}\times[0,\infty))~~\mbox{and}~~\mbox{a.e.}~~\mbox{in}~~\Omega\times(0,\infty)
 \label{zjscz2.5297x96302222t666t4}
\end{equation}
as well as
\begin{equation}
\nabla c_\varepsilon\rightharpoonup \nabla c~~\begin{array}{ll}
 \mbox{in}~~ L_{loc}^{2}(\bar{\Omega}\times[0,\infty))
   \end{array}\label{1.1ddfgghhhge666ccdf2345ddvbnmklllhyuisda}
\end{equation}
and
\begin{equation}
 \nabla u_\varepsilon\rightharpoonup \nabla u ~~\mbox{ in}~~L^{2}_{loc}(\bar{\Omega}\times[0,\infty))
 \label{zjscz2.5297x96366602222tt4455}
\end{equation}
and
\begin{equation}
 u_\varepsilon\rightharpoonup u ~~\mbox{ in}~~L^{\frac{10}{3}}_{loc}(\bar{\Omega}\times[0,\infty))
 \label{zjscz2.5ffgtt297x96302266622tt4}
\end{equation}
 with some triple $(n, c, u)$ which is a global weak solution of \dref{1.1} in the sense of Definition \ref{df1}.
\end{lemma}
\begin{proof}
 From Lemmas \ref{lemmaghjffggssddgghhmk4563025xxhjklojjkkk}, \ref{lemmddaghjsffggggsddgghhmk4563025xxhjklojjkkk},
 \ref{4455lemma45630hhuujjuuyytt}, \ref{qqqqlemma45630hhuujjuuyytt} and the Aubin--Lions lemma (\cite{Simon}),
  we can derive \dref{zjscz2.5297x963ddfgh0ddfggg6662222tt3}--\dref{zjscz2.fgghh5297x963ddfgh0ddfggg6662222tt3} and \dref{zjscz2.5297x96302222t666t4}--\dref{zjscz2.5ffgtt297x96302266622tt4} holds.
Next, let $g_\varepsilon(x, t) := -c_\varepsilon+F_{\varepsilon}(n_{\varepsilon})-u_{\varepsilon}\cdot\nabla c_{\varepsilon}.$
With the
notation, the   second equation of \dref{1.1fghyuisda} can be rewritten in the component form as
\begin{equation}
c_{\varepsilon t}-\Delta c_{\varepsilon } = g_\varepsilon.
\label{1.1666ddjjkllllfgffgghheccdfddfgg2345ddvbnmklllhyuisda}
\end{equation}

Case $\frac{1}{3}<\alpha\leq\frac{1}{2}$:
Observing that $$\frac{5}{4}<\frac{4}{3}<\min\{\frac{6\alpha+2}{3},\frac{10}{3}\}~~~\mbox{for}~~\frac{1}{3}<\alpha\leq\frac{1}{2},$$
thus, recalling \dref{bnmbncz2.ffghh5ghhjuyuivvbnnihjj}, \dref{bnmbncz2.ffgddffffhh5ghhjuyuivvbnnihjj}  and \dref{1.1dddfgbhnjmdfgeddvbnmklllhyussdddfgggdisda} and applying the H\"{o}lder inequality, we conclude that $ g_\varepsilon$
is bounded in $L^{\frac{5}{4}} (\Omega\times(0, T))$ for any  $\varepsilon\in(0,1)$,  we may invoke the standard parabolic regularity theory  to \dref{1.1666ddjjkllllfgffgghheccdfddfgg2345ddvbnmklllhyuisda} and  infer that
$(c_{\varepsilon})_{\varepsilon\in(0,1)}$ is bounded in
$L^{\frac{5}{4}} ((0, T); W^{2,\frac{5}{4}}(\Omega))$.
Thus,  by virtue of \dref{wwwwwqqqq1.1dddfgbhnjmdfgeddvbnmklllhyussddisda} and the Aubin--Lions lemma we derive that  the relative compactness of $(c_{\varepsilon})_{\varepsilon\in(0,1)}$ in
$L^{\frac{5}{4}} ((0, T); W^{1,\frac{5}{4}}(\Omega))$. We can pick an appropriate subsequence which is
still written as $(\varepsilon_j )_{j\in \mathbb{N}}$ such that $\nabla c_{\varepsilon_j} \rightarrow z_1$
 in $L^{\frac{5}{4}} (\Omega\times(0, T))$ for all $T\in(0, \infty)$ and some
$z_1\in L^{\frac{5}{4}} (\Omega\times(0, T))$ as $j\rightarrow\infty$, hence $\nabla c_{\varepsilon_j} \rightarrow z_1$ a.e. in $\Omega\times(0, \infty)$
 as $j \rightarrow\infty$.
In view  of \dref{1.1ddfgghhhge666ccdf2345ddvbnmklllhyuisda} and  the Egorov theorem we conclude  that
$z_1=\nabla c,$ and whence \dref{1.1ddhhyujiiifgghhhge666ccdf2345ddvbnmklllhyuisda} holds.
Next, we pay our attention to the case $\frac{1}{2}<\alpha<1$: By straightforward calculations, and using relation $\frac{1}{2}<\alpha<1$, one has
 $$\frac{5}{4}<\frac{5}{3}<\min\{\frac{10\alpha}{3},\frac{10}{3}\}¡£$$
 Therefore, noticing that \dref{11bnmbncz2.ffghh5ghhjuyuivvbnnihjj}, \dref{bnmbncz2.ffgddffffhh5ghhjuyuivvbnnihjj}, and using \dref{1.1dddfgbhnjmdfgeddvbnmklllhyussdddfgggdisda},
it follows from the  H\"{o}lder inequality that
\begin{equation}
g_\varepsilon ~~~\mbox{is bounded in }~~~L^{\frac{5}{4}} (\Omega\times(0, T))~~~\mbox{for any}~~~\varepsilon\in(0,1).
\label{1.1666ddjjklllddfghlfgffgghheccdfddfgg2345ddvbnmklllhyuisda}
\end{equation}
%
%
Employing almost exactly the same arguments as in the proof of Case $\frac{1}{3}<\alpha\leq\frac{1}{2}$, and taking
 advantage of \dref{1.1666ddjjklllddfghlfgffgghheccdfddfgg2345ddvbnmklllhyuisda}, we conclude the estimate \dref{1.1ddfgghhhge666ccdf2345ddvbnmklllhyuisda}. Case $\alpha
\geq1$ is similar to case $\frac{1}{3}<\alpha\leq\frac{1}{2}$, we omit it.
In the following, we shall prove $(n,c,u)$ is a weak solution of problem \dref{1.1} in Definition \ref{df1}.
In fact, $\alpha>\frac{1}{3}$ yields to
$$
r>1,
   $$
   where $r$ is given by \dref{zjscz2.5297x963ddfgh0ddfggg6662222tt3}.
   Therefore,
with the help of  \dref{zjscz2.5297x963ddfgh0ddfggg6662222tt3}--\dref{zjscz2.fgghh5297x963ddfgh0ddfggg6662222tt3}, \dref{zjscz2.5297x96302222t666t4}--\dref{zjscz2.5297x96366602222tt4455}, we can derive  \dref{dffff1.1fghyuisdakkklll}.
Now, by the nonnegativity of $n_\varepsilon$ and $c_\varepsilon$, we derive  $n \geq 0$ and $c\geq 0$. Next, due to
\dref{zjscz2.5297x96366602222tt4455} and $\nabla\cdot u_{\varepsilon} = 0$, we conclude that
$\nabla\cdot u = 0$ a.e. in $\Omega\times (0, \infty)$.
On the other hand, in view of \dref{1.1dddfgbhnjmdfgeddvbnmklllhyussddisda}, \dref{1.1dddfgbhnjmdfgeddvbnhhjjjmklllhyussddisda}
 and \dref{1.1dddgghhjfgbhnjmdfgeddvbnhhjjjmklllhyussddisda}, we conclude that
\begin{equation}n_\varepsilon F'_{\varepsilon}(n_{\varepsilon})S_\varepsilon(x, n_{\varepsilon}, c_{\varepsilon})\nabla c_\varepsilon\rightharpoonup z_2
~~\mbox{in}~~ L^{r}(\Omega\times(0,T))~~\mbox{as}~~\varepsilon: = \varepsilon_j\searrow 0~~\mbox{for each}~~ T\in(0,\infty),
\label{1.1ddddfddffttyygghhyujiiifgghhhgffgge6bhhjh66ccdf2345ddvbnmklllhyuisda}
\end{equation}
where  $r$ is given by \dref{zjscz2.5297x963ddfgh0ddfggg6662222tt3}.
On the other hand, it follows from \dref{x1.73142vghf48rtgyhu}, \dref{3.10gghhjuuloollyuigghhhyy}, \dref{1.ffggvddfghhghhhhbddgbnxxccvvn1}, \dref{zjscz2.5297x963ddfgh0ddfggg6662222tt3}, \dref{zjscz2.fgghh5297x963ddfgh0ddfggg6662222tt3} and \dref{1.1ddhhyujiiifgghhhge666ccdf2345ddvbnmklllhyuisda} that
\begin{equation}n_\varepsilon F'_{\varepsilon}(n_{\varepsilon})S_\varepsilon(x, n_{\varepsilon}, c_{\varepsilon})\nabla c_\varepsilon\rightarrow nS(x, n, c)\nabla c~~\mbox{a.e.}~~\mbox{in}~~\Omega\times(0,\infty)~~\mbox{as}~~\varepsilon: = \varepsilon_j\searrow 0.
\label{1.1ddddfddfftffghhhtyygghhyujiiifgghhhgffgge6bhhjh66ccdf2345ddvbnmklllhyuisda}
\end{equation}
Again by the Egorov theorem, we gain $z_2=nS(x, n, c)\nabla c,$ and hence \dref{1.1ddddfddffttyygghhyujiiifgghhhgffgge6bhhjh66ccdf2345ddvbnmklllhyuisda}
can be rewritten as
\begin{equation}n_\varepsilon F'_{\varepsilon}(n_{\varepsilon})S_\varepsilon(x, n_{\varepsilon}, c_{\varepsilon})\nabla c_\varepsilon\rightharpoonup nS(x, n, c)\nabla c
~~\mbox{in}~~ L^{r}(\Omega\times(0,T))~~\mbox{as}~~\varepsilon: = \varepsilon_j\searrow 0~~\mbox{for each}~~ T\in(0,\infty),
\label{1.1ddddfddffttyygghhyujiiffghhhifgghhhgffgge6bhhjh66ccdf2345ddvbnmklllhyuisda}
\end{equation}
which together with $r>1$
implies  the integrability of $nS(x, n, c)\nabla c$ in \dref{726291hh} as well.
It is not hard to check that
$${\frac{10(3\alpha+1)}{9(\alpha+2)}}>1~~\mbox{if}~~~\frac{1}{3}<\alpha\leq\frac{1}{2}~~\mbox{and}~~{\frac{10\alpha}{3(\alpha+1)}} >1~~\mbox{if}~~~\frac{1}{2}<\alpha<1.
   $$
   Thereupon, recalling \dref{1.1dddfgbhnjmdfgeddvbnmklllhyussddisda}, \dref{1.1dddfgbhnjmdfgeddvbnhhjjjmklllhyussddisda} and \dref{1.1dddgghhjfgbhnjmdfgeddvbnhhjjjmklllhyussddisda}, we infer that for each $T\in(0, \infty)$
  \begin{equation}
n_\varepsilon u_{\varepsilon}\rightharpoonup z_3 ~~\mbox{in}~~ L^{\tilde{r}}(\Omega\times(0,T))~~\mbox{with}~~\tilde{r}= \left\{\begin{array}{ll}
 \frac{10(3\alpha+1)}{9(\alpha+2)}~~\mbox{if}~~\frac{1}{3}<\alpha\leq\frac{1}{2},\\
  \frac{10\alpha}{3(\alpha+1)}~~\mbox{if}~~\frac{1}{2}<\alpha<1,\\
 \frac{5}{3}~~\mbox{if}~~\alpha\geq1.\\
   \end{array}\right.
   ~~\mbox{as}~~\varepsilon: = \varepsilon_j\searrow 0,
   \label{zjscz2.529ffghhh7x963ddfgh0ddfggg6662222tt3}
\end{equation}
This, together with \dref{zjscz2.5297x963ddfgh0ddfggg6662222tt3},  and \dref{zjscz2.5297x96302222t666t4}, implies
 \begin{equation}n_\varepsilon u_\varepsilon\rightarrow nu~~\mbox{a.e.}~~\mbox{in}~~\Omega\times(0,\infty)~~\mbox{as}~~\varepsilon: = \varepsilon_j\searrow 0.
\label{1.1ddddfddfftffghhhtyygghhyujiiifgghhhgffgge6bhhffgggjh66ccdf2345ddvbnmklllhyuisda}
\end{equation}
 Along with \dref{zjscz2.529ffghhh7x963ddfgh0ddfggg6662222tt3} and \dref{1.1ddddfddfftffghhhtyygghhyujiiifgghhhgffgge6bhhffgggjh66ccdf2345ddvbnmklllhyuisda}, the Egorov theorem guarantees that $z_3=nu$, whereupon we derive from \dref{zjscz2.529ffghhh7x963ddfgh0ddfggg6662222tt3} that
 \begin{equation}
n_\varepsilon u_{\varepsilon}\rightharpoonup nu ~~\mbox{in}~~ L^{\tilde{r}}(\Omega\times(0,T))~~\mbox{with}~~\tilde{r}= \left\{\begin{array}{ll}
 \frac{10(3\alpha+1)}{9(\alpha+2)}~~\mbox{if}~~\frac{1}{3}<\alpha\leq\frac{1}{2},\\
  \frac{10\alpha}{3(\alpha+1)}~~\mbox{if}~~\frac{1}{2}<\alpha<1,\\
 \frac{5}{3}~~\mbox{if}~~\alpha\geq1\\
   \end{array}\right.
   ~~\mbox{as}~~\varepsilon: = \varepsilon_j\searrow 0\label{zjscz2.529ffghhh7x963djkkkkdfgh0ddfggg6662222tt3}
\end{equation}
   for each $T\in(0, \infty)$.

As a straightforward consequence of \dref{zjscz2.fgghh5297x963ddfgh0ddfggg6662222tt3} and \dref{zjscz2.5297x96302222t666t4}, it holds that
\begin{equation}
c_\varepsilon u_\varepsilon\rightarrow cu ~~\mbox{ in}~~ L^{1}_{loc}(\bar{\Omega}\times(0,\infty))~~~\mbox{as}~~\varepsilon=\varepsilon_j\searrow0.
 \label{zxxcvvfgggjscddfffcvvfggz2.5297x96302222tt4}
\end{equation}
Thus, the integrability of $nu$ and $cu$ in \dref{726291hh} is verified by \dref{zjscz2.fgghh5297x963ddfgh0ddfggg6662222tt3}
 and \dref{zjscz2.5297x96302222t666t4}.

%
%
%
%
 Next, by  \dref{zjscz2.5297x96302222t666t4} and
using the fact that
 $\|Y_{\varepsilon}\varphi\|_{L^2(\Omega)} \leq \|\varphi\|_{L^2(\Omega)}(\varphi\in L^2_{\sigma}(\Omega))$
and
$Y_{\varepsilon}\varphi \rightarrow \varphi$ in $L^2(\Omega)$ as $\varepsilon\searrow0$, we derive that there exists a positive constant $C_1$ such that
\begin{equation}
\begin{array}{rl}
\left\|Y_{\varepsilon}u_{\varepsilon}(\cdot,t)-u(\cdot,t)\right\|_{L^2(\Omega)}  \leq&\disp{\left\|Y_{\varepsilon}[u_{\varepsilon}(\cdot,t)-u(\cdot,t)]\right\|_{L^2(\Omega)}+
\left\|Y_{\varepsilon}u(\cdot,t)-u(\cdot,t)\right\|_{L^2(\Omega)}}\\
\leq&\disp{\left\|u_{\varepsilon}(\cdot,t)-u(\cdot,t)\right\|_{L^2(\Omega)}+
\left\|Y_{\varepsilon}u(\cdot,t)-u(\cdot,t)\right\|_{L^2(\Omega)}}\\
\rightarrow&\disp{0~~\mbox{as}~~\varepsilon=\varepsilon_j\searrow0}\\
\end{array}
\label{ggjjssdffzccvvvvggjscz2.5297x963ccvbb111kkuu}
\end{equation}
and
\begin{equation}
\begin{array}{rl}
\left\|Y_{\varepsilon}u_{\varepsilon}(\cdot,t)-u(\cdot,t)\right\|_{L^2(\Omega)}^2  \leq&\disp{\left(\|Y_{\varepsilon}u_{\varepsilon}(\cdot,t)|\|_{L^2(\Omega)}+\|u(\cdot,t)|\|_{L^2(\Omega)}\right)^2}\\
\leq&\disp{\left(\|u_{\varepsilon}(\cdot,t)|\|_{L^2(\Omega)}+\|u(\cdot,t)|\|_{L^2(\Omega)}\right)^2}\\
\leq&\disp{C_1~~\mbox{for all}~~t\in(0,\infty)~~\mbox{and}~~\varepsilon\in(0,1). }\\
\end{array}
\label{ggjjssdffzccffggvvvvggjscz2.5297x963ccvbb111kkuu}
\end{equation}
Now, thus, by \dref{zjscz2.5297x96302222t666t4}, \dref{ggjjssdffzccvvvvggjscz2.5297x963ccvbb111kkuu} and \dref{ggjjssdffzccffggvvvvggjscz2.5297x963ccvbb111kkuu} and the dominated convergence theorem, we derive that
\begin{equation}
\begin{array}{rl}
\disp\int_{0}^T\|Y_{\varepsilon}u_{\varepsilon}(\cdot,t)-u(\cdot,t)\|_{L^2(\Omega)}^2dt\rightarrow0 ~~\mbox{as}~~\varepsilon=\varepsilon_j\searrow0 ~~~\mbox{for all}~~T>0,
\end{array}
\label{ggjjssdffzccffggvvvvgccvvvgjscz2.5297x963ccvbb111kkuu}
\end{equation}
which implies that
\begin{equation}
Y_\varepsilon u_\varepsilon\rightarrow u ~~\mbox{in}~~ L_{loc}^2([0,\infty); L^2(\Omega)).
 \label{zjscz2.5297x96302266622tt44}
\end{equation}
Now, combining  \dref{zjscz2.5297x96302222t666t4}   with  \dref{zjscz2.5297x96302266622tt44}, we derive
\begin{equation}
\begin{array}{rl}
Y_{\varepsilon}u_{\varepsilon}\otimes u_{\varepsilon}\rightarrow u \otimes u ~~\mbox{in}~~L^1_{loc}(\bar{\Omega}\times[0,\infty))~~\mbox{as}~~\varepsilon=\varepsilon_j\searrow0.
\end{array}
\label{ggjjssdffzccfccvvfgghjjjvvvvgccvvvgjscz2.5297x963ccvbb111kkuu}
\end{equation}
Therefore, by \dref{1.1ddddfddffttyygghhyujiiffghhhifgghhhgffgge6bhhjh66ccdf2345ddvbnmklllhyuisda}, \dref{zjscz2.529ffghhh7x963djkkkkdfgh0ddfggg6662222tt3}--\dref{zxxcvvfgggjscddfffcvvfggz2.5297x96302222tt4} and \dref{ggjjssdffzccfccvvfgghjjjvvvvgccvvvgjscz2.5297x963ccvbb111kkuu} we conclude that the integrability of
$nS(x,n,c)\nabla c, nu$ and $cu,u\otimes u$ in \dref{726291hh}.
Finally, for any fixed $T\in(0, \infty)$, applying  \dref{zjscz2.5297x963ddfgh0ddfggg6662222tt3}, we can derive
\begin{equation}
\begin{array}{rl}
&\disp\int_0^T\left\|F_{\varepsilon}(n_{\varepsilon}(\cdot,t))-n(\cdot,t)\right\|_{L^r(\Omega)}^rdt\\  \leq&\disp{\disp\int_0^T\left\|F_{\varepsilon}(n_{\varepsilon}(\cdot,t))-F_{\varepsilon}(n(\cdot,t))\right\|_{L^r(\Omega)}^rdt+
\disp\int_0^T\left\|F_{\varepsilon}(n(\cdot,t))-n(\cdot,t)\right\|_{L^r(\Omega)}^rdt}\\
\leq&\disp{\|F'_{\varepsilon}\|_{L^\infty(\Omega\times(0,\infty))}\disp\int_0^T\left\|n_{\varepsilon}(\cdot,t)-n(\cdot,t)\right\|_{L^r(\Omega)}^rdt+
\disp\int_0^T\left\|F_{\varepsilon}(n(\cdot,t))-n(\cdot,t)\right\|_{L^r(\Omega)}^rdt,}\\
\end{array}
\label{ggjjssdffzccvvvvggjddfgscz2.5297x963ccvbb111kkuu}
\end{equation}
where $r$ is the same as \dref{zjscz2.5297x963ddfgh0ddfggg6662222tt3}.
Besides that, we also deduce from \dref{ffggg1.ffggvddfghhghhhhbbhhjjjnxxccvvn1} and $r>1$ that
\begin{equation}
\begin{array}{rl}
\left\|F_{\varepsilon}(n(\cdot,t))-n(\cdot,t)\right\|_{L^r(\Omega\times(0,T))}^r  \leq&\disp{2^r\|n(\cdot,t)\|}\\
\end{array}
\label{ggjjssdffzccvvvvffghhggjddfgscz2.5297x963ccvbb111kkuu}
\end{equation}
for each $t\in(0, T)$, which together with \dref{zjscz2.5297x963ddfgh0ddfggg6662222tt3} shows the integrability of
$\left\|F_{\varepsilon}(n(\cdot,t))-n(\cdot,t)\right\|_{L^r(\Omega\times(0,T))}^r $ on $(0, T).$
Thereupon, by virtue of \dref{1.ffggvddfghhghhhhbddgbnxxccvvn1}, we infer from the dominated convergence theorem that
\begin{equation}
\begin{array}{rl}
\disp\int_0^T\left\|F_{\varepsilon}(n)-n\right\|_{L^r(\Omega)}^rdt\rightarrow0 ~~\mbox{as}~~\varepsilon=\varepsilon_j\searrow0
\end{array}
\label{ggjjssdffzccvvvvggjghhjjddfgscz2.5297x963ccvbb111kkuu}
\end{equation}
for each $T\in(0, \infty)$.
Inserting \dref{ggjjssdffzccvvvvggjghhjjddfgscz2.5297x963ccvbb111kkuu} into \dref{ggjjssdffzccvvvvggjddfgscz2.5297x963ccvbb111kkuu} and using
\dref{zjscz2.5297x963ddfgh0ddfggg6662222tt3} and \dref{1.ffggvddfghhghhhhbbnxxccvvn1}, we can see clearly that
\begin{equation}
\begin{array}{rl}
F_{\varepsilon}(n)\rightarrow n ~~\mbox{in}~~L^r_{loc}(\bar{\Omega}\times[0,\infty))~~\mbox{as}~~\varepsilon=\varepsilon_j\searrow0.
\end{array}
\label{ggjjssdffddfghzccfccvffghhhvfgghjjjvvvvgccvvvgjscz2.5297x963ccvbb111kkuu}
\end{equation}
 Finally, according to \dref{zjscz2.5297x963ddfgh0ddfggg6662222tt3}--\dref{zjscz2.fgghh5297x963ddfgh0ddfggg6662222tt3},
 \dref{zjscz2.5297x96302222t666t4},
\dref{zjscz2.5297x96366602222tt4455},
 \dref{1.1ddfgghhhge666ccdf2345ddvbnmklllhyuisda},
 \dref{1.1ddddfddffttyygghhyujiiffghhhifgghhhgffgge6bhhjh66ccdf2345ddvbnmklllhyuisda}, \dref{zjscz2.529ffghhh7x963djkkkkdfgh0ddfggg6662222tt3},  \dref{zxxcvvfgggjscddfffcvvfggz2.5297x96302222tt4}, \dref{zjscz2.5297x96302266622tt44}, \dref{ggjjssdffzccfccvvfgghjjjvvvvgccvvvgjscz2.5297x963ccvbb111kkuu} and \dref{ggjjssdffddfghzccfccvffghhhvfgghjjjvvvvgccvvvgjscz2.5297x963ccvbb111kkuu}, we may pass to the limit in
the respective weak formulations associated with the the regularized system \dref{1.1fghyuisda} and get
 the integral
identities \dref{eqx45xx12112ccgghh}--\dref{eqx45xx12112ccgghhjjgghh}.
\end{proof}

\section{A priori estimates for the  problem \dref{1.1}}
By a straightforward adaptation of the reasoning in Lemma 2.1 of \cite{Winkler51215}, one can derive the following
basic statement on local solvability and extensibility of solutions to \dref{1.1}.
\begin{lemma}\label{ddffffflemma70}
Let $\Omega \subset \mathbb{R}^3$ be a bounded domain with smooth boundary and
%
%
the initial data $(n_0,c_0,u_0)$ fulfills \dref{ccvvx1.731426677gg}.
Then there exist $T_{max}\in  (0,\infty]$ and
a classical solution $(n, c, u, P)$ of \dref{1.1} in
$\Omega\times(0, T_{max})$ such that
\begin{equation}
 \left\{\begin{array}{ll}
 n\in C^0(\bar{\Omega}\times[0,T_{max}))\cap C^{2,1}(\bar{\Omega}\times(0,T_{max})),\\
  c\in  C^0(\bar{\Omega}\times[0,T_{max}))\cap C^{2,1}(\bar{\Omega}\times(0,T_{max})),\\
  u\in  C^0(\bar{\Omega}\times[0,T_{max}))\cap C^{2,1}(\bar{\Omega}\times(0,T_{max})),\\
  P\in  C^{1,0}(\bar{\Omega}\times(0,T_{max})),\\
   \end{array}\right.\label{1.1ddfgddffggghyuisda}
\end{equation}
 classically solving \dref{1.1} in $\Omega\times[0,T_{max})$.
%
Moreover,  $n$ and $c$ are nonnegative in
$\Omega\times(0, T_{max})$, and
\begin{equation}
\|n(\cdot, t)\|_{L^\infty(\Omega)}+\|c(\cdot, t)\|_{W^{1,\infty}(\Omega)}+\|A^\gamma u(\cdot, t)\|_{L^{2}(\Omega)}\rightarrow\infty~~ \mbox{as}~~ t\nearrow T_{max},
\label{1.ssdrffgg163072x}
\end{equation}
where $\gamma$ is given by \dref{ccvvx1.731426677gg}.
\end{lemma}
In order to discuss the boundedness  and classical  solution of \dref{1.1}, in light of Lemma \ref{ddffffflemma70},
  we can  pick any $s_0\in(0,T_{max})$ and $s_0\leq1$, %
 there exists
$\beta>0$ such that
\begin{equation}\label{eqx45xx1ddfgggg2112}
\|n(\tau)\|_{L^\infty(\Omega)}\leq \beta~~~\|u(\tau)\|_{W^{1,\infty}(\Omega)}\leq \beta~~\mbox{and}~~\|c(\tau)\|_{W^{2,\infty}(\Omega)}\leq \beta~~\mbox{for all}~~\tau\in[0,s_0].
\end{equation}

\begin{lemma}(\cite{Winkler11215,Zhengsdsd6})\label{lemma630jklhhjj}
 Let $l\in[1,+\infty)$ and $r\in[1,+\infty]$ be such that
 \begin{equation}
\left\{\begin{array}{ll}
l<\frac{3r}{3-r}~~\mbox{if}~~
r\leq 3,\\
l\leq\infty~~\mbox{if}~~
r>3.
 \end{array}\right.\label{3.10gghhjuulooll}
\end{equation}
If $\kappa=0$ and for all $K > 0$ there exists $C = C(l, r,K)$ such that 
 \begin{equation}\|n(\cdot, t)\|_{L^r(\Omega)}\leq K~~ \mbox{for all}~~ t\in(0, T_{max}),
\label{3.10gghhjuuloollgghhhy}
\end{equation}
then
 \begin{equation}\|D u(\cdot, t)\|_{L^l(\Omega)}\leq C~~ \mbox{for all}~~ t\in(0, T_{max}),
\label{3.10gghhjuuloollgghhhyhh}
\end{equation}
where $(n, c, u, P)$ is a solution of  \dref{1.1}.
\end{lemma}
\begin{lemma}\label{lemma45xy1222232}(\cite{Hieber,Zhengddkkllssssssssdefr23,Zhengssdddssddddkkllssssssssdefr23})
Suppose  $\gamma\in (1,+\infty)$, $g\in L^\gamma((0, T); L^\gamma(
\Omega))$ and  $v_0\in W^{2,\gamma}(\Omega)$
such that $\disp\frac{\partial v_0}{\partial \nu}=0$.
Let $v$ be a solution of the following initial boundary value
 \begin{equation}
 \left\{\begin{array}{ll}
 v_t -\Delta v+v=g,~~~(x, t)\in
 \Omega\times(0, T ),\\
\disp\frac{\partial v}{\partial \nu}=0,~~~(x, t)\in
 \partial\Omega\times(0, T ),\\
v(x,0)=v_0(x),~~~(x, t)\in
 \Omega.\\
 \end{array}\right.\label{33331.3xcx29}
\end{equation}
Then there exists a positive constant $C_{\gamma}:=C_{\gamma,|\Omega|}$ such that if $s_0\in[0,T)$, $v(\cdot,s_0)\in W^{2,\gamma}(\Omega)(\gamma>N)$ with $\disp\frac{\partial v(\cdot,s_0)}{\partial \nu}=0,$ then
\begin{equation}
\begin{array}{rl}
&\disp{\int_{s_0}^Te^{\gamma  s}(\| v(\cdot,t)\|^{\gamma}_{L^{\gamma}(\Omega)}+\|\Delta v(\cdot,t)\|^{\gamma}_{L^{\gamma}(\Omega)})ds}\\
\leq &\disp{C_{\gamma}\left(\int_{s_0}^Te^{\gamma  s}
\|g(\cdot,s)\|^{\gamma}_{L^{\gamma}(\Omega)}ds+e^{\gamma  s}(\|v_0(\cdot,s_0)\|^{\gamma}_{L^{\gamma}(\Omega)}+\|\Delta v_0(\cdot,s_0)\|^{\gamma}_{L^{\gamma}(\Omega)})\right).}\\
\end{array}
\label{3333cz2.5bbv114}
\end{equation}
\end{lemma}
The proof of the following lemma is very similar to that of Lemmas  \ref{fvfgfflemma45}--\ref{lemmaghjffggssddgghhmk4563025xxhjklojjkkk}, so we omit its proof here.
\begin{lemma}\label{fvfgdfhhhfflemma45}
There exists 
$\tilde{\lambda} > 0$
 such that the solution of \dref{1.1} satisfies
%
%
\begin{equation}
\int_{\Omega}{n}+\int_{\Omega}{c}\leq \tilde{\lambda}~~\mbox{for all}~~ t\in(0, T_{max}).
\label{ddfgczhhhh.548cfg924ghyuji}
\end{equation}
%
%
\end{lemma}
\begin{lemma}\label{lemmaghjggssddggmk2xhjklojjkkk}
Let $\alpha>\frac{1}{3}$,
 $
S(x,n,c)=C_S(1+n)^{-\alpha}
$
 and $\kappa=0.$
Then there exists $C>0$ such that the solution of \dref{1.1} satisfies
\begin{equation}
\begin{array}{rl}
&\disp{\int_{\Omega} n^{2\alpha }+\int_{\Omega}   c^2+\int_{\Omega}  | {u}|^2\leq C~~~\mbox{for all}~~ t\in (0, T_{max}).}\\
\end{array}
\label{czfvgb2.ffghhhh5ghhjuyuccvviihjj}
\end{equation}
Moreover, for $T\in(0, T_{max})$, it holds that
one can find a constant $C > 0$ 
such that
\begin{equation}
\begin{array}{rl}
&\disp{\int_{0}^T\int_{\Omega} \left[  n^{2\alpha-2} |\nabla {n}|^2+ |\nabla {c}|^2+ |\nabla {u}|^2\right]\leq C.}\\
\end{array}
\label{bnmbncz2.5ghdfghhhjuyuivvbnnihjj}
\end{equation}
\end{lemma}
\begin{lemma}\label{lemma45630223116}
Let
 $p=\frac{13}{8}$, $\alpha\in(\frac{1}{3},\frac{3}{4}]$ and $\theta=\frac{3}{2}$.
Then there exists a positive constant $\tilde{l}_0\in(\frac{59}{20},3)$ such that
\begin{equation}
\frac{\frac{5}{6}-\frac{1}{\theta'(p+1-\alpha)}}{\frac{7}{6}-\frac{1}{p+1-\alpha}}+\frac{\frac{1}{\tilde{l}_0}-\frac{1}{\theta (p+1-\alpha)}}{\frac{1}{\tilde{l}_0}+\frac{2}{3}-\frac{1}{p+1-\alpha}}<1,
\label{zjscz2.5297x9630222211444125}
\end{equation}
where $\theta'=\frac{\theta}{\theta-1}=3.$
\end{lemma}
\begin{proof}
It is easy to verify that $$\frac{24}{55}<\frac{1}{p+1-\alpha}\leq\frac{8}{15}$$
and
$$1-\frac{\frac{1}{\tilde{l}_0}-\frac{1}{\theta (p+1-\alpha)}}{\frac{1}{\tilde{l}_0}+\frac{2}{3}-\frac{1}{p+1-\alpha}}=\frac{\frac{2}{3}-\frac{1}{\theta' (p+1-\alpha)}}{\frac{1}{\tilde{l}_0}+\frac{2}{3}-\frac{1}{p+1-\alpha}}.$$
These together with some basic calculation yield to \dref{zjscz2.5297x9630222211444125}.
\end{proof}

Now,  let us derive the following a priori bounded for the solutions of model \dref{1.1}, which plays a key rule in obtaining the main results.
\begin{lemma}\label{lemma45566645630223}
Let  \begin{equation}\label{gddffffngghhhhjjmmx1.731426677gg}
S(x,n,c)=C_S(1+n)^{-\alpha}
\end{equation}
 and $\kappa=0.$
If
  \begin{equation}\label{gddffffnjjmmx1.731426677gg}
\frac{1}{3}<\alpha\leq \frac{3}{4},
\begin{array}{ll}\\
 \end{array}
\end{equation}
then there exists a positive constant $p_0>\frac{3}{2}$ such that 
 the solution of \dref{1.1} from Lemma \ref{ddffffflemma70} satisfies
\begin{equation}
\int_{\Omega}n^{p_0}(x,t)dx\leq C ~~~\mbox{for all}~~ t\in(0,T_{max}).
\label{334444zjscz2.5297x96302222114}
\end{equation}
\end{lemma}
\begin{proof}
Let
$p=\frac{13}{8}$.
%
Taking ${n^{p-1}}$ as the test function for the first equation of
$\dref{1.1}$
 and combining with the second equation and using $\nabla\cdot u=0$, we obtain

\begin{equation}
\begin{array}{rl}
&\disp{\frac{1}{{p}}\frac{d}{dt}\|n\|^{{p}}_{L^{{p}}(\Omega)}+({{p}-1})\int_{\Omega}n^{{{p}-2}}|\nabla n|^2}
\\
=&\disp{-\int_\Omega n^{p-1}\nabla\cdot(n
C_S(1+n)^{-\alpha}\nabla c) }\\
=&\disp{(p-1)\int_\Omega  n^{p-1}
C_S(1+n)^{-\alpha}\nabla n\cdot\nabla c ~~\mbox{for all}~~ t\in(0,T_{max}),}\\
\end{array}
\label{3333cz2.5114114}
\end{equation}
which derives,
\begin{equation}
\begin{array}{rl}
&\disp{\frac{1}{{p}}\frac{d}{dt}\|n\|^{{{p}}}_{L^{{p}}(\Omega)}+({{p}-1})\int_{\Omega}n^{{{p}-2}}|\nabla n|^2}
\\
\leq&\disp{-\frac{{p}+1-\alpha}{{p}}\int_{\Omega} n^{p} +(p-1)\int_\Omega  n^{p-1}
C_S(1+n)^{-\alpha}\nabla n\cdot\nabla c
   + \frac{{p}+1-\alpha}{{p}}\int_\Omega n^{p} ~~\mbox{for all}~~ t\in(0,T_{max}).}\\
\end{array}
\label{3333cz2.5kk1214114114}
\end{equation}

Here, for any $\varepsilon_1>0,$
we invoke the Young inequality 
to find 
that
%
\begin{equation}
\begin{array}{rl}
&\disp{\frac{{p}+1-\alpha}{{p}}\int_\Omega   n^p \leq\varepsilon_1\int_\Omega  n^{{{p}+\frac{2}{3}}} +C_1(\varepsilon_1,p),}
\end{array}
\label{3333cz2.563011228ddff}
\end{equation}
where  
$$C_1(\varepsilon_1,{p})=\frac{\frac{2}{3}}{{p}+\frac{2}{3}}\left(\varepsilon_1\frac{{p}+\frac{2}{3}}{{p}}\right)^{-\frac{p}{\frac{2}{3}} }
\left(\frac{{p}+1-\alpha}{{p}}\right)^{\frac{{p}+\frac{2}{3}}{\frac{2}{3}} }|\Omega|.$$
%
%

 Once more integrating by parts, in view of   \dref{gddffffngghhhhjjmmx1.731426677gg},  we also find
that
\begin{equation}
\begin{array}{rl}
&\disp{(p-1)\int_\Omega  n^{p-1}
C_S(1+n)^{-\alpha}\nabla n\cdot\nabla c }
\\
=&\disp{(p-1)\int_\Omega \nabla \int_0^{n}\tau^{p-1}
C_S(1+\tau)^{-\alpha}d\tau\cdot\nabla c }
\\
=&\disp{-(p-1)\int_\Omega \int_0^{n}\tau^{p-1}
C_S(1+\tau)^{-\alpha}d\tau \Delta c }
\\
\leq&\disp{\frac{C_S({{p}-1})}{p-\alpha}\int_\Omega  n^{p-\alpha}|\Delta c| ,}
\\
\end{array}
\label{3333c334444z2.563019114}
\end{equation}
so that the Young inequality implies
\begin{equation}
\begin{array}{rl}
&\disp{\frac{C_S({{p}-1})}{p-\alpha}\int_\Omega  n^{p-\alpha}|\Delta c|}
\\
\leq&\disp{\int_\Omega  n^{{p}+1-\alpha}+\frac{1}{ { {p}+1-\alpha}}\left[\frac{ { {p}+1-\alpha}}{ {p}-\alpha}\right]^{- ({p}-\alpha) }\left(\frac{C_S({{p}-1})}{p-\alpha} \right)^{ { {p}+1-\alpha}}\int_\Omega |\Delta c|^{ { {p}+1-\alpha}} }
\\
=&\disp{\int_\Omega  n^{{p}+1-\alpha}+{A}_1\int_\Omega |\Delta c|^{ { {p}+1-\alpha}} ,}
\\
\end{array}
\label{3333cz2.563019114gghh}
\end{equation}
where
$$A_1:=\frac{1}{ { {p}+1-\alpha}}\left[\frac{ { {p}+1-\alpha}}{ {p}-\alpha}\right]^{- ({p}-\alpha) }\left(\frac{C_S({{p}-1})}{p-\alpha} \right)^{ { {p}+1-\alpha}}.$$
Thus, inserting \dref{3333cz2.563011228ddff} and \dref{3333cz2.563019114gghh} into \dref{3333cz2.5kk1214114114}, we get
\begin{equation*}
\begin{array}{rl}
\disp\frac{1}{{p}}\disp\frac{d}{dt}\|n\|^{{{p}}}_{L^{{p}}(\Omega)}+(p-1)\int_{\Omega} n^{p-2}|\nabla n|^2\leq&\disp{\varepsilon_1\int_\Omega  n^{{{p}+\frac{2}{3}}} +\int_\Omega n^{{{p}+1}-\alpha}-\frac{{p}+1-\alpha}{{p}}\int_{\Omega} n^{p} }\\
&+\disp{{A}_1\int_\Omega |\Delta c|^{ { {p}+1-\alpha}} +
C_1(\varepsilon_1,{p})~~\mbox{for all}~~ t\in(0,T_{max}).}\\
\end{array}
\end{equation*}
Since, $\alpha>\frac{1}{3}$, yields to ${{p}+1}-\alpha<{{p}+\frac{2}{3}},$ therefore, by the Young inequality, we conclude that
\begin{equation}\label{3223444333c334444z2.563019114}
\begin{array}{rl}
\disp\frac{1}{{p}}\disp\frac{d}{dt}\|n\|^{{{p}}}_{L^{{p}}(\Omega)}+\frac{4(p-1)}{p^2}\|\nabla   n^{\frac{p}{2}}\|_{L^2(\Omega)}^{2}\leq&\disp{2\varepsilon_1\int_\Omega  n^{{{p}+\frac{2}{3}}} -\frac{{p}+1-\alpha}{{p}}\int_{\Omega} n^{p} }\\
&+\disp{{A}_1\int_\Omega |\Delta c|^{ { {p}+1-\alpha}} +
C_2(\varepsilon_1,{p}),}\\
\end{array}
\end{equation}
where
$$C_2(\varepsilon_1,{p})=\frac{\alpha-\frac{1}{3}}{{p}+\frac{2}{3}}\left(\varepsilon_1\frac{{p}+\frac{2}{3}}{{p}+1-\alpha}\right)^{-\frac{p+1-\alpha}
{\alpha-\frac{1}{3}} }
\left(\frac{{p}+1-\alpha}{{p}}\right)^{\frac{{p}+\frac{2}{3}}{\alpha-\frac{1}{3}} }|\Omega|.$$
On the other hand, by the Gagliardo--Nirenberg inequality and \dref{ddfgczhhhh2.5ghju48cfg924ghyuji}, one can get there exist positive constants  $\mu_0$ and $\mu_1$ such that
\begin{equation}
\begin{array}{rl}
\disp\int_{\Omega}n^{p+\frac{2}{3}}=&\disp{\|  n^{\frac{p}{2}}\|^{\frac{2(p+\frac{2}{3})}{p }}_{L^{\frac{2(p+\frac{2}{3})}{p }}(\Omega)}}\\
\leq&\disp{\mu_{0}(\|\nabla   n^{\frac{p}{2}}\|_{L^2(\Omega)}^{\frac{p}{p+\frac{2}{3}}}\|  n^{\frac{p}{2}}\|_{L^\frac{2}{p }(\Omega)}^{1-\frac{p}{p+\frac{2}{3}}}+\|  n^{\frac{p}{2}}\|_{L^\frac{2}{p }(\Omega)})^{\frac{2(p+\frac{2}{3})}{p }}}\\
\leq&\disp{\mu_{1}(\|\nabla   n^{\frac{p}{2}}\|_{L^2(\Omega)}^{2}+1).}\\
\end{array}
\label{123cz2.57151hhddfffkkhhhjddffffgukildrftjj}
\end{equation}
Collecting \dref{3223444333c334444z2.563019114} and \dref{123cz2.57151hhddfffkkhhhjddffffgukildrftjj}, we derive that
\begin{equation}\label{3223444333c334444zsddfff2.563019114}
\begin{array}{rl}
\disp\frac{1}{{p}}\disp\frac{d}{dt}\|n\|^{{{p}}}_{L^{{p}}(\Omega)}\leq&\disp{(2\varepsilon_1-\frac{4(p-1)}{p^2}\frac{1}{\mu_1})\int_\Omega  n^{{{p}+\frac{2}{3}}} -\frac{{p}+1-\alpha}{{p}}\int_{\Omega} n^{p} }\\
&+\disp{{A}_1\int_\Omega |\Delta c|^{ { {p}+1-\alpha}} +
C_3(\varepsilon_1,{p})~~\mbox{for all}~~ t\in(0,T_{max}),}\\
\end{array}
\end{equation}
where $$C_3(\varepsilon_1,{p})=C_2(\varepsilon_1,{p})+\frac{4(p-1)}{p^2}.$$
For any $t\in (s_0,T_{max})$,
employing the variation-of-constants formula to \dref{3223444333c334444zsddfff2.563019114}, we obtain
\begin{equation}
\begin{array}{rl}
&\disp{\frac{1}{{p}}\|n(t) \|^{{{p}}}_{L^{{p}}(\Omega)}}
\\
\leq&\disp{\frac{1}{{p}}e^{-({p}+1-\alpha)(t-s_0)}\|n(s_0) \|^{{{p}}}_{L^{{p}}(\Omega)}+(2\varepsilon_1-\frac{4(p-1)}{p^2}\frac{1}{\mu_1})\int_{s_0}^t
e^{-({p}+1-\alpha)(t-s)}\int_\Omega n^{{{p}+\frac{2}{3}}} ds}\\
&+\disp{{A}_1\int_{s_0}^t
e^{-({p}+1-\alpha)(t-s)}\int_\Omega |\Delta c|^{ {p}+1-\alpha} dxds+ C_3(\varepsilon_1,{p})\int_{s_0}^t
e^{-({p}+1-\alpha)(t-s)}ds}\\
\leq&\disp{(2\varepsilon_1-\frac{4(p-1)}{p^2}\frac{1}{\mu_1}   )\int_{s_0}^t
e^{-({p}+1-\alpha)(t-s)}\int_\Omega n^{{{p}+\frac{2}{3}}} ds+{A}_1\int_{s_0}^t
e^{-({p}+1-\alpha)(t-s)}\int_\Omega |\Delta c|^{ {p}+1-\alpha} dxds}\\
&+\disp{C_4(\varepsilon_1,{p})}\\
\end{array}
\label{3333cz2.5kk1214114114rrgg}
\end{equation}
with
$$
\begin{array}{rl}
C_4:=C_4(\varepsilon_1,{p})=&\disp\frac{1}{{p}}e^{-({p}+1-\alpha)(t-s_0)}\|n(s_0) \|^{{{p}}}_{L^{{p}}(\Omega)}+
 C_3(\varepsilon_1,{p})\int_{s_0}^t
e^{-({p}+1-\alpha)(t-s)}ds.\\
\end{array}
$$
Due to \dref{ddfgczhhhh.548cfg924ghyuji}, employing Lemma \ref{lemma630jklhhjj}, we derive that  
\begin{equation}\|  Du(\cdot, t)\|_{L^l(\Omega)}\leq C_5~~ \mbox{for all}~~ t\in(0, T_{max})~~ \mbox{and for any}~~~l<\frac{3}{2},
\label{3.10gghhjukklllkklllooppuloollgghhhyhh}
\end{equation}
so that  the Sobolev imbedding theorem implies that
\begin{equation}\|  u(\cdot, t)\|_{L^{l_0}(\Omega)}\leq C_6~~ \mbox{for all}~~ t\in(0, T_{max})~~ \mbox{and for any}~~~l_0<3.
\label{3.10gghhjukklllkkllloffghhjjoppuloollgghhhyhh}
\end{equation}
Now, due to
Lemma  \ref{lemma45xy1222232} 
and the second equation of \dref{1.1} and using the H\"{o}lder inequality, we have
\begin{equation}\label{3333cz2.5kke34567789999001214114114rrggjjkk}
\begin{array}{rl}
&\disp{{A}_1\int_{s_0}^t
e^{-({p}+1-\alpha)(t-s)}\int_\Omega |\Delta c|^{ {p}+1-\alpha} ds}
\\
=&\disp{{A}_1e^{-({p}+1-\alpha)t}\int_{s_0}^t
e^{({p}+1-\alpha)s}\int_\Omega |\Delta c|^{ {p}+1-\alpha} ds}\\
\leq&\disp{2^{{p}+1-\alpha}{A}_1e^{-({p}+1-\alpha)t}C_ {{p}+1-\alpha }(\int_{s_0}^t
\int_\Omega e^{({p}+1-\alpha)s}(|u\cdot\nabla c|^ {{p}+1-\alpha }+n^ {{p}+1-\alpha }) ds+e^{({p}+1-\alpha)s_0}\|c(s_0,t)\|^ {{p}+1-\alpha }_{W^{2,  {{p}+1-\alpha }}})}\\
\leq&\disp{2^{{p}+1-\alpha}{A}_1e^{-({p}+1-\alpha)t}C_ {{p}+1-\alpha }\int_{s_0}^t
 e^{({p}+1-\alpha)s}(\|u\|_{L^{\theta(p+1-\alpha)}(\Omega)}^{p+1-\alpha}\|\nabla c\|_{L^{\theta'(p+1-\alpha)}(\Omega)}^{p+1-\alpha}+n^ {{p}+1-\alpha }) ds+C_7}\\
\end{array}
\end{equation}
for all $t\in(s_0, T_{max})$,
where $\theta=\frac{3}{2},\theta'=\frac{\theta}{\theta-1}=3$,
 $$C_7=2^{{p}+1-\alpha}{A}_1e^{-({p}+1-\alpha)t}C_{ {p}+1-\alpha}e^{({p}+1-\alpha)s_0}\|c(s_0,t)\|^ {{p}+1-\alpha }_{W^{2,  {{p}+1-\alpha }}}.$$
Next, with the help of the Gagliardo--Nirenberg inequality and \dref{czfvgb2.5ghhjuyuccvviihjj}, we derive that
\begin{equation}\label{3333cz2.5kkett677734567789999001214114114rrggjjkk}
\begin{array}{rl}
&\disp{\|\nabla c\|_{L^{\theta'(p+1-\alpha)}(\Omega)}^{p+1-\alpha}}
\\
\leq&\disp{C_8\|\Delta c\|_{L^{(p+1-\alpha)}(\Omega)}^{a(p+1-\alpha)}\| c\|_{L^{2}(\Omega)}^{(1-a)(p+1-\alpha)}+C_8\| c\|_{L^{2}(\Omega)}^{p+1-\alpha}}\\
\leq&\disp{C_9\|\Delta c\|_{L^{(p+1-\alpha)}(\Omega)}^{a(p+1-\alpha)}+C_9}\\
\end{array}
\end{equation}
with some constants $C_8>0 $ and $C_9  > 0$, where
$$a=\frac{\frac{5}{6}-\frac{1}{\theta'(p+1-\alpha)}}{\frac{7}{6}-\frac{1}{p+1-\alpha}}\in(0,1).$$
We derive from the Young inequality that for any $\delta\in(0,1)$,
\begin{equation}\label{3333cz2.5kkett677734gghhh567789999001214114114rrggjjkk}\begin{array}{rl}
&\disp{\|u\|_{L^{\theta(p+1-\alpha)}(\Omega)}^{p+1-\alpha}\|\nabla c\|_{L^{\theta'(p+1-\alpha)}(\Omega)}^{p+1-\alpha}}\\
\leq&\disp{C_9\|\Delta c\|_{L^{(p+1-\alpha)}(\Omega)}^{a(p+1-\alpha)}\|u\|_{L^{\theta (p+1-\alpha)}(\Omega)}^{p+1-\alpha}+C_9\|u\|_{L^{\theta (p+1-\alpha)}(\Omega)}^{p+1-\alpha}}\\
\leq&\disp{\delta\|\Delta c\|_{L^{(p+1-\alpha)}(\Omega)}^{p+1-\alpha}+C_{10}\|u\|_{L^{\theta (p+1-\alpha)}(\Omega)}^{\frac{p+1-\alpha}{1-a}}+
C_{9}\|u\|_{L^{\theta (p+1-\alpha)}(\Omega)}^{p+1-\alpha},}\\
\end{array}
\end{equation}
where $C_{10}=(1-a)\left(\delta\times\frac{1}{a}\right)^{-\frac{a}{1-a}}C_9^{\frac{1}{1-a}}.$

Inserting \dref{3333cz2.5kkett677734gghhh567789999001214114114rrggjjkk} into  \dref{3333cz2.5kke34567789999001214114114rrggjjkk}, we conclude that
\begin{equation}\label{3333cz2ddfgggg.5kke345677ddfff89999001214114114rrggjjkk}
\begin{array}{rl}
&\disp{{A}_1\int_{s_0}^t
e^{-({p}+1-\alpha)(t-s)}\int_\Omega |\Delta c|^{ {p}+1-\alpha} ds}
\\
\leq&\disp{2^{{p}+1-\alpha}{A}_1e^{-({p}+1-\alpha)t}C_ {{p}+1-\alpha }\delta\int_{s_0}^t
 e^{({p}+1-\alpha)s}\|\Delta c\|_{L^{(p+1-\alpha)}(\Omega)}^{p+1-\alpha} ds}\\
&+\disp{2^{{p}+1-\alpha}{A}_1e^{-({p}+1-\alpha)t}C_ {{p}+1-\alpha }\int_{s_0}^t
 e^{({p}+1-\alpha)s}[C_{10}\|u\|_{L^{\theta (p+1-\alpha)}(\Omega)}^{\frac{p+1-\alpha}{1-a}}+
C_{9}\|u\|_{L^{\theta (p+1-\alpha)}(\Omega)}^{p+1-\alpha}] ds}\\
&\disp{+2^{{p}+1-\alpha}{A}_1e^{-({p}+1-\alpha)t}C_ {{p}+1-\alpha }\int_{s_0}^t
 e^{({p}+1-\alpha)s}n^ {{p}+1-\alpha } ds+C_7}\\
\end{array}
\end{equation}
for all $t\in(s_0, T_{max})$.
Therefore, choosing $\delta=\frac{1}{2}\frac{1}{2^{{p}+1-\alpha}C_ {{p}+1-\alpha }}$ yields to 
\begin{equation}\label{3333cz2.5kke345677ddfff89999001214114114rrggjjkk}
\begin{array}{rl}
&\disp{{A}_1\int_{s_0}^t
e^{-({p}+1-\alpha)(t-s)}\int_\Omega |\Delta c|^{ {p}+1-\alpha} ds}
\\
\leq&\disp{2^{{p}+2-\alpha}{A}_1e^{-({p}+1-\alpha)t}C_ {{p}+1-\alpha }C_{10}\int_{s_0}^t
 e^{({p}+1-\alpha)s}\|u\|_{L^{\theta (p+1-\alpha)}(\Omega)}^{\frac{p+1-\alpha}{1-a}}ds}\\
&+\disp{2^{{p}+2-\alpha}{A}_1e^{-({p}+1-\alpha)t}C_ {{p}+1-\alpha }C_{9}\int_{s_0}^t
\|u\|_{L^{\theta (p+1-\alpha)}(\Omega)}^{p+1-\alpha} ds}\\
&\disp{+2[2^{{p}+1-\alpha}{A}_1e^{-({p}+1-\alpha)t}C_ {{p}+1-\alpha }\int_{s_0}^t
 e^{({p}+1-\alpha)s}n^ {{p}+1-\alpha } ds+C_7].}\\
\end{array}
\end{equation}
On the other hand, by Lemma \ref{lemma45630223116}, we may choose $\frac{59}{20}<\tilde{l}_0<3$  such that
\begin{equation}\label{3333llllllcz2.5kke345677ddfff89999001214114114rrggjjkk}\frac{\frac{5}{6}-\frac{1}{\theta'(p+1-\alpha)}}{\frac{7}{6}-\frac{1}{p+1-\alpha}}+\frac{\frac{1}{\tilde{l}_0}-\frac{1}{\theta (p+1-\alpha)}}{\frac{1}{\tilde{l}_0}+\frac{2}{3}-\frac{1}{p+1-\alpha}}<1.
\end{equation}
Therefore,  it follows from the Gagliardo--Nirenberg inequality, \dref{3.10gghhjukklllkkllloffghhjjoppuloollgghhhyhh} and the Young
inequality that there exist  constants $C_{11} = C_{11}(p) > 0$ and $C_{12} = C_{12}(p) > 0$ such that
\begin{equation}\label{3333cz2.5kke345677ddff89001214114114rrggjjkk}
\begin{array}{rl}\|u\|_{L^{\theta (p+1-\alpha)}(\Omega)}^{\frac{p+1-\alpha}{1-a}}\leq& \| Au\|_{L^{p+1-\alpha}(\Omega)}^{\frac{p+1-\alpha}{1-a}\tilde{a}}\| u\|_{L^{\tilde{l}_0}(\Omega)}^{\frac{p+1-\alpha}{1-a}(1-\tilde{a})}\\
\leq& \| Au\|_{L^{p+1-\alpha}(\Omega)}^{\frac{p+1-\alpha}{1-a}\tilde{a}}C_{11}\\
\leq& \| Au\|_{L^{p+1-\alpha}(\Omega)}^{p+1-\alpha}+C_{12}\\
\end{array}
\end{equation}
with 
$$\tilde{a}=\frac{\frac{1}{\tilde{l}_0}-\frac{1}{\theta (p+1-\alpha)}}{\frac{1}{\tilde{l}_0}+\frac{2}{3}-\frac{1}{p+1-\alpha}}\in(0,1).$$
Here we have use the fact that
$\frac{p+1-\alpha}{1-a}\tilde{a}=
(p+1-\alpha)\frac{\frac{7}{6}-\frac{1}{p+1-\alpha}}{\frac{1}{3}-\frac{1}{\theta (p+1-\alpha)}}\frac{\frac{1}{\tilde{l}_0}-\frac{1}{\theta (p+1-\alpha)}}
{\frac{1}{\tilde{l}_0}+\frac{2}{3}-\frac{1}{p+1-\alpha}}<p+1-\alpha$ by \dref{3333llllllcz2.5kke345677ddfff89999001214114114rrggjjkk}.
In light  of $\frac{1}{1-a}>1$,
similarly, we derive that
\begin{equation}\label{3333czffgyyu2.5kke345677ddff89001214114114rrggjjkk}
\begin{array}{rl}\|u\|_{L^{\theta (p+1-\alpha)}(\Omega)}^{p+1-\alpha}\leq&  \| Au\|_{L^{p+1-\alpha}(\Omega)}^{p+1-\alpha}+C_{13}.\\
\end{array}
\end{equation}
Collecting \dref{3333cz2.5kke345677ddfff89999001214114114rrggjjkk}, \dref{3333cz2.5kke345677ddff89001214114114rrggjjkk} and \dref{3333czffgyyu2.5kke345677ddff89001214114114rrggjjkk}, we derive that
\begin{equation}\label{3333cz2.5kke345677ddffdfrrtyhhhjjf89999001214114114rrggjjkk}
\begin{array}{rl}
&\disp{{A}_1\int_{s_0}^t
e^{-({p}+1-\alpha)(t-s)}\int_\Omega |\Delta c|^{ {p}+1-\alpha} ds}
\\
\leq&\disp{2^{{p}+2-\alpha}{A}_1e^{-({p}+1-\alpha)t}C_ {{p}+1-\alpha }C_{10}\int_{s_0}^t
 e^{({p}+1-\alpha)s}[ \| Au\|_{L^{p+1-\alpha}(\Omega)}^{p+1-\alpha}+C_{12}] ds}\\
&+\disp{2^{{p}+2-\alpha}{A}_1e^{-({p}+1-\alpha)t}C_ {{p}+1-\alpha }C_{9}\int_{s_0}^t
 e^{({p}+1-\alpha)s}[ \| Au\|_{L^{p+1-\alpha}(\Omega)}^{p+1-\alpha}+C_{13})] ds}\\
&\disp{+2^{{p}+2-\alpha}{A}_1e^{-({p}+1-\alpha)t}C_ {{p}+1-\alpha }\int_{s_0}^t
\int_\Omega e^{({p}+1-\alpha)s}n^ {{p}+1-\alpha } ds+2C_7}\\
\leq&\disp{2^{{p}+2-\alpha}{A}_1e^{-({p}+1-\alpha)t}C_ {{p}+1-\alpha } [C_{10}+ C_{9}]\int_{s_0}^t
 e^{({p}+1-\alpha)s}\| Au\|_{L^{p+1-\alpha}(\Omega)}^{p+1-\alpha} ds}\\
&\disp{+2^{{p}+2-\alpha}{A}_1e^{-({p}+1-\alpha)t}C_ {{p}+1-\alpha }\int_{s_0}^t
\int_\Omega e^{({p}+1-\alpha)s}n^ {{p}+1-\alpha } ds+C_{14}},\\
\end{array}
\end{equation}
where $C_{14}=2[2^{{p}+1-\alpha}{A}_1e^{-({p}+1-\alpha)t}C_ {{p}+1-\alpha }(C_{10}C_{12}+C_{9}C_{13})+C_7].$
Putting $\tilde{u}(\cdot, s) := e^su(\cdot, s), s\in (s_0, t)$,
we obtain from the third  equation in \dref{1.1} that
\begin{equation}\label{33ffgh33cz2.5kke345677ddffdfrrtyhhhjjf89999001214114114rrggjjkk}\tilde{u}_s=\Delta \tilde{u}+\tilde{u}+e^sn\nabla \phi+e^s\nabla P, \end{equation}
which implies that
\begin{equation}\label{33ffgh33cz2.5kke345677ddffdfrrtyhhhjjf89999001214114114rrggjjkk}\tilde{u}_s+ A\tilde{u}=\mathcal{P}(\tilde{u}+e^sn\nabla \phi+e^s\nabla P), \end{equation}
where $\mathcal{P}$ denotes the Helmholtz projection mapping $L^2(\Omega)$ onto its subspace $L^2_{\sigma}(\Omega)$ of all
solenoidal vector field.
Thus
by $p<2$ and \dref{3.10gghhjukklllkkllloffghhjjoppuloollgghhhyhh},
we derive from Lemma \ref{lemma45xy1222232} (see also Theorem 2.7 of \cite{Giga1215}) that there exist positive
constants $C_{15},C_{16},C_{17}$ and $C_{18}$
such that
\begin{equation}
\begin{array}{rl}
&\disp{\int_{s_0}^te^{(p+1-\alpha)  s}\|A u(\cdot,t)\|^{p+1-\alpha}_{L^{p+1-\alpha}(\Omega)}ds}\\
\leq &\disp{C_{15}\left(\int_{s_0}^te^{(p+1-\alpha)s}
(\|u(\cdot,s)\|^{p+1-\alpha}_{L^{p+1-\alpha}(\Omega)}+\|n(\cdot,s)\|^{p+1-\alpha}_{L^{p+1-\alpha}(\Omega)})ds+e^{(p+1-\alpha)  t}+1\right)}\\
\leq &\disp{C_{16}\left(\int_{s_0}^te^{(p+1-\alpha)s}
(\|u(\cdot,s)\|^{p+1-\alpha}_{L^{\tilde{l}_0}(\Omega)}|\Omega|^{\frac{\tilde{l}_0-p-1+\alpha}{\tilde{l}_0}}+\|n(\cdot,s)\|^{p+1-\alpha}_{L^{p+1-\alpha}(\Omega)})ds+e^{(p+1-\alpha)  t}+1\right)}\\
\leq &\disp{C_{17}\int_{s_0}^te^{(p+1-\alpha)s}\|n(\cdot,s)\|^{p+1-\alpha}_{L^{p+1-\alpha}(\Omega)}ds+(1+C_{18})e^{(p+1-\alpha)  t}.}\\
\end{array}
\label{cz2.5bbhjjkkkiiooov114}
\end{equation}
Here we have used the fact that $$\frac{15}{8}\leq p+1-\alpha<\frac{55}{24}<\frac{59}{20}<\tilde{l}_0.$$
Inserting \dref{cz2.5bbhjjkkkiiooov114} into \dref{3333cz2.5kke345677ddffdfrrtyhhhjjf89999001214114114rrggjjkk}, we derive that
\begin{equation}\label{3333cz2.5kke345677ddfddffgghhhjjkkffffdfrrtyhhhjjf89999001214114114rrggjjkk}
\begin{array}{rl}
&\disp{{A}_1\int_{s_0}^t
e^{-({p}+1-\alpha)(t-s)}\int_\Omega |\Delta c|^{ {p}+1-\alpha} ds}
\\
\leq&\disp{2^{{p}+2-\alpha}{A}_1e^{-({p}+1-\alpha)t}C_ {{p}+1-\alpha }  [C_{10}+ C_{9}]\left(C_{17}\int_{s_0}^te^{(p+1-\alpha)s}
\|n(\cdot,s)\|^{p+1-\alpha}_{L^{p+1-\alpha}(\Omega)}ds+(1+C_{18})e^{(p+1-\alpha)  t}\right)}\\
&\disp{+2^{{p}+2-\alpha}{A}_1e^{-({p}+1-\alpha)t}C_ {{p}+1-\alpha }\int_{s_0}^t
\int_\Omega e^{({p}+1-\alpha)s}n^ {{p}+1-\alpha } ds+C_{14}}\\
\leq&\disp{  C_{19}\int_{s_0}^te^{(p+1-\alpha)s}
\|n(\cdot,s)\|^{p+1-\alpha}_{L^{p+1-\alpha}(\Omega)}ds+C_{20}},\\
\end{array}
\end{equation}
where $C_{19}=2[2^{{p}+1-\alpha}{A}_1e^{-({p}+1-\alpha)t}C_ {{p}+1-\alpha } [C_{10}+ C_{9}]C_{17}+2^{{p}+1-\alpha}{A}_1e^{-({p}+1-\alpha)t}C_ {{p}+1-\alpha }]$ and  $$C_{20}:=C_{20}(p)=2^{{p}+2-\alpha}{A}_1e^{-({p}+1-\alpha)t}C_ {{p}+1-\alpha }  [C_{10}+ C_{9}](1+C_{18})e^{(p+1-\alpha)  t}+C_{14}.$$
Collecting \dref{3333cz2.5kk1214114114rrgg} and  \dref{3333cz2.5kke345677ddfddffgghhhjjkkffffdfrrtyhhhjjf89999001214114114rrggjjkk}, applying Lemma \ref{lemma45630223116} and the Young inequality, we derive that
\begin{equation}
\begin{array}{rl}
&\disp{\frac{1}{{p}}\|n(t) \|^{{{p}}}_{L^{{p}}(\Omega)}}
\\
\leq&\disp{(2\varepsilon_1-\frac{4(p-1)}{p^2}\frac{1}{\mu_1}) \int_{s_0}^t
e^{-({p}+1-\alpha)(t-s)}\int_\Omega n^{{{p}+\frac{2}{3}}} ds}\\
&\disp{+C_{19} \int_{s_0}^t
e^{-({p}+1-\alpha)(t-s)}\int_\Omega n^{{p}+1-\alpha} ds+C_{21}}\\
\leq&\disp{(3\varepsilon_1-\frac{4(p-1)}{p^2}\frac{1}{\mu_1}) \int_{s_0}^t
e^{-({p}+1-\alpha)(t-s)}\int_\Omega n^{{{p}+\frac{2}{3}}} ds+C_{22}}\\
\end{array}
\label{3333cz2.5kk121fttyuiii4114114rrgg}
\end{equation}
with $C_{21}=C_{20}+C_4(\varepsilon_1,{p})$ and $C_{22}=\frac{\alpha-\frac{1}{3}}{{p}+\frac{2}{3}}\left(\varepsilon_1\frac{{p}+\frac{2}{3}}{{p}+1-\alpha}\right)^{-\frac{p+1-\alpha}{\alpha-\frac{1}{3}} }
\left(C_{19}\right)^{\frac{{p}+\frac{2}{3}}{\alpha-\frac{1}{3}} }+C_{21}$
  Thus, choosing $\delta $ and $\varepsilon_1$ small enough
   (e.g. $\varepsilon_1<\frac{(p-1)}{p^2}\frac{1}{\mu_1}$)
    in \dref{3333cz2.5kk121fttyuiii4114114rrgg}, using  \dref{eqx45xx1ddfgggg2112} and the H\"{o}lder inequality, we derive that there exits a positive constant $p_0>\frac{3}{2}$ such that 
  \begin{equation}
\int_{\Omega}n^{p_0}(x,t)dx\leq C_{23} ~~~\mbox{for all}~~ t\in(0,T_{max}).
\label{334444zjscz2.5297dfggggx96302222114}
\end{equation}
The proof of Lemma \ref{lemma45566645630223} is completed.
\end{proof}

If we can find parameters that allow for an application of Lemmas  \ref{lemma45566645630223} and \ref{lemmaghjggssddggmk2xhjklojjkkk} at the same
time, we can conclude boundedness of $n$. This is the goal we pursue in the following lemma:
\begin{lemma}\label{lemmaghjggdddddssddggmk2xhjklojjkkk}
Let $\alpha>\frac{1}{3}$,
 $
S(x,n,c)=C_S(1+n)^{-\alpha}
$
 and $\kappa=0.$
Then there exists a positive constant $q_0>\frac{3}{2}$ such that 
 the solution of \dref{1.1} from Lemma \ref{ddffffflemma70} satisfies
\begin{equation}
\int_{\Omega}n^{q_0}(x,t)dx\leq C ~~~\mbox{for all}~~ t\in(0,T_{max}).
\label{334444zjscz2.52dddd97x96302222114}
\end{equation}
\end{lemma}
\begin{proof}
%
%
%
Let 
\begin{equation}
q_0=\left\{\begin{array}{ll}
p_0~~\mbox{if}~~
\frac{1}{3}<\alpha\leq \frac{3}{4},\\
2\alpha~~\mbox{if}~~
\alpha> \frac{3}{4},
 \end{array}\right.\label{ddffggg3.10gghhjuulooll}
\end{equation}
where $p_0$ is the same as  Lemma \ref{lemma45566645630223}.
 Then obviously, $q_0>\frac{3}{2}$, hence, in view of Lemmas  \ref{lemma45566645630223} and \ref{lemmaghjggssddggmk2xhjklojjkkk}, yields to \dref{334444zjscz2.52dddd97x96302222114}. The proof of Lemma \ref{lemmaghjggdddddssddggmk2xhjklojjkkk} is completed.
\end{proof}
With the help of Lemma \ref{lemmaghjggdddddssddggmk2xhjklojjkkk}, in light of the Gagliardo--Nirenberg inequality and an application of well-known arguments
from parabolic regularity theory, we can derive the following Lemma:
\begin{lemma}\label{lemmddaghjssddgghhmk4563025xxhjklojjkkk}
 Assume the hypothesis of Theorem \ref{theoremddffggg3} holds.  Then for $p>2$,
 one can find a constant $C>0$ such that the solution of \dref{1.1} satisfies
\begin{equation}
\begin{array}{rl}
&\disp{\int_{\Omega}   c^{p}\leq C~~~\mbox{for all}~~ t\in (0, T_{max}).}\\
\end{array}
\label{czfvgb2.5ghffghjuyuccvviihjj}
\end{equation}
\end{lemma}
\begin{proof}
Firstly, for any $p>2,$ taking  ${c^{p-1}}$ as the test function for the second  equation of \dref{1.1} and using $\nabla\cdot u=0$, the H\"{o}lder inequality  and \dref{334444zjscz2.52dddd97x96302222114} yields  that
\begin{equation}
\begin{array}{rl}
&\disp\frac{1}{p}\disp\frac{d}{dt}\|{c}\|^{{{p}}}_{L^{{p}}(\Omega)}+(p-1)
\int_{\Omega} {c^{p-2}}|\nabla c|^2+ \int_{\Omega} c^{p}\\
=&\disp{\int_{\Omega} nc^{p-1}}\\
\leq&\disp{\left(\int_{\Omega}n^{\frac{3}{2}}\right)^{\frac{2}{3}}\left(\int_{\Omega}c^{3(p-1)}\right)^{\frac{1}{3}}}\\
\leq&\disp{C_1\left(\int_{\Omega}c^{3(p-1)}\right)^{\frac{1}{3}}~~~\mbox{for all}~~t\in (0, T_{max}),}\\
\end{array}
\label{hhxxcdfssxxdccffgghvvjjcz2.5}
\end{equation}
where in the last inequality we have used the fact that \dref{334444zjscz2.52dddd97x96302222114} and the H\"{o}lder inequality.
Now, due to \dref{czfvgb2.ffghhhh5ghhjuyuccvviihjj}, in light of the Gagliardo--Nirenberg inequality, we derive that there exist positive constants $C_{2}$ and $C_{3}$ such that
\begin{equation}
\begin{array}{rl}
\disp\left(\int_{\Omega}c^{3(p-1)}\right)^{\frac{1}{3}} =&\disp{\| { c^{\frac{p}{2}}}\|^{{\frac{2(p-1)}{p}}}_{L^{\frac{6(p-1)}{p}}(\Omega)}}\\
\leq&\disp{C_{2}\left(\| \nabla{ c^{\frac{p}{2}}}\|^{\mu_1}_{L^{2}(\Omega)}\|{ c^{\frac{p}{2}}}\|^{{1-\mu_1}}_{L^{\frac{4}{p}}(\Omega)}+
\|{ c^{\frac{p}{2}}}\|_{L^{\frac{4}{p}}(\Omega)}\right)^{\frac{2(p-1)}{p}}}\\
\leq&\disp{C_{3}(\| \nabla{ c^{\frac{p}{2}}}\|^{\frac{2(p-1)}{p}\mu_1}_{L^{2}(\Omega)}
+1)}\\
=&\disp{C_{3}(\| \nabla{ c^{\frac{p}{2}}}\|^{2\frac{3p-5}{3p-2}}_{L^{2}(\Omega)}
+1)~~~\mbox{for all}~~t\in (0, T_{max})}\\
\end{array}
\label{ddffbnmbnddfgffggjjkkuuiicz2ddfvgbhh.htt678ddfghhhyuiihjj}
\end{equation}
with some positive constants $C_{2}$ and $C_{3}$ and
$$\mu_1=\frac{\frac{3p}{4}-\frac{3p}{6(p-1)}}{-\frac{1}{2}+\frac{3p}{4}}=
p\frac{\frac{3}{4}-\frac{3}{6(p-1)}}{-\frac{1}{2}+\frac{3p}{4}}\in(0,1).$$
Inserting \dref{ddffbnmbnddfgffggjjkkuuiicz2ddfvgbhh.htt678ddfghhhyuiihjj} into \dref{hhxxcdfssxxdccffgghvvjjcz2.5}, in view of the  fact that $2\frac{3p-5}{3p-2}<2,$ therefore,
by using the Young inequality, we derive that
\begin{equation}
\begin{array}{rl}
&\disp\frac{1}{p}\disp\frac{d}{dt}\|{c}\|^{{{p}}}_{L^{{p}}(\Omega)}+\frac{p-1}{2}
\int_{\Omega} {c^{p-2}}|\nabla c|^2+ \int_{\Omega} c^{p}\leq C_4~~~\mbox{for all}~~t\in (0, T_{max}),\\
\end{array}
\label{hhxxcdfssxxdccffgghvvsdfffggjjcz2.5}
\end{equation}
which combined with an ODE comparison argument entails \dref{czfvgb2.5ghffghjuyuccvviihjj}.
\end{proof}
Underlying the estimates established above (Lemmas \ref{fvfgdfhhhfflemma45}--\ref{lemmaghjggssddggmk2xhjklojjkkk}), 
 we can derive the following boundedness results
by invoking a Moser-type iteration and the standard parabolic regularity
arguments (see the proof of Lemmas \ref{sss222lemma4444556645630223} and \ref{kkklemmaghjmk4563025xxhjklojjkkk}).
\begin{lemma}\label{lemma45630hhuujj}
Let $\alpha> \frac{1}{3}$ and $\gamma$ be as in \dref{ccvvx1.731426677gg}.  Then one can find a positive constant $C$
 such that 
\begin{equation}
\|n(\cdot,t)\|_{L^\infty(\Omega)}  \leq C ~~\mbox{for all}~~ t\in(0,T_{max})
\label{zjscz2.5297x9630111kk}
\end{equation}
and
\begin{equation}
\|c(\cdot,t)\|_{W^{1,\infty}(\Omega)}  \leq C ~~\mbox{for all}~~ t\in(0,T_{max})
\label{zjscz2.5297x9630111kkhh}
\end{equation}
as well as
\begin{equation}
\|u(\cdot,t)\|_{W^{1,\infty}(\Omega)}  \leq C ~~\mbox{for all}~~ t\in(0,T_{max}).
\label{zjscz2.5297x9630111kkhhffrr}
\end{equation}
Moreover, we also have
\begin{equation}
\|A^\gamma u(\cdot,t)\|_{L^{2}(\Omega)}  \leq C ~~\mbox{for all}~~ t\in(0,T_{max}).
\label{zjscz2.5297x9630111kkhhffrreerr}
\end{equation}
\end{lemma}
\begin{proof}
Employing  the same arguments as in the proof of Lemmas \ref{fvfgdfhhhfflemma45}--\ref{lemmaghjggssddggmk2xhjklojjkkk}, and taking advantage of \dref{334444zjscz2.52dddd97x96302222114} and \dref{czfvgb2.5ghffghjuyuccvviihjj}, we conclude the estimates \dref{zjscz2.5297x9630111kk}--\dref{zjscz2.5297x9630111kkhhffrreerr}. The proof of Lemma \ref{lemma45630hhuujj} is completed.
\end{proof}

Combining Lemma  \ref{ddffffflemma70} and  Lemma \ref{lemma45630hhuujj}, we readily  prove Theorem \ref{theoremddffggg3}.

{\bf Proof of Theorem \ref{theoremddffggg3}}:  In view of Lemma \ref{lemma45630hhuujj}, $\|u(\cdot, t)\|_{L^\infty(\Omega)},\|c(\cdot, t)\|_{W^{1,\infty}(\Omega)}$ and $\|A^\gamma u(\cdot, t)\|_{L^{2}(\Omega)}$ are bounded uniformly with respect to $t\in(0, T_{max})$. Thereupon the assertion of Theorem \ref{theoremddffggg3} is immediately obtained from Lemma \ref{ddffffflemma70}.

{\bf Acknowledgement}:
This work is partially supported by  the National Natural
Science Foundation of China (No. 11601215), Shandong Provincial
Science Foundation for Outstanding Youth (No. ZR2018JL005), Shandong Provincial
Natural Science Foundation,  China (No. ZR2016AQ17) and the Doctor Start-up Funding of Ludong University (No. LA2016006).

\end{document}